\documentclass[a4paper]{article}
\usepackage [a4paper,left=2.5cm,bottom=1.7cm,right=2.5cm,top=1.7cm]{geometry}
\usepackage[english]{babel}
\usepackage{amssymb}
\usepackage{amsmath,amsthm}
\usepackage{subcaption}
\usepackage{tabularx,array}
\usepackage{graphicx}
\usepackage{enumitem} 
\newtheorem{theorem}{Theorem}
\newtheorem{lemma}{Lemma}
\newtheorem{corollary}{Corollary}
\newtheorem{proposition}{Proposition}

\newcommand{\keywords}[1]{\par\addvspace\baselineskip
\noindent\enspace\ignorespaces#1}
\begin{document}
\title{Unbiased truncated quadratic variation for volatility estimation in jump diffusion processes.}

\author{Chiara Amorino$^{*}$, Arnaud Gloter\thanks{Laboratoire de Math\'ematiques et Mod\'elisation d'Evry, CNRS, Univ Evry, Universit\'e Paris-Saclay, 91037, Evry, France.}} 

\maketitle

\begin{abstract}
The problem of integrated volatility estimation for an Ito semimartingale is considered under
discrete high-frequency observations in short time horizon. We provide an asymptotic expansion for the integrated volatility that gives us, in detail, the contribution deriving from the jump part. The knowledge of such a contribution allows us to build an unbiased version of the truncated quadratic variation, in which the bias is visibly reduced. In earlier results to have the original truncated realized volatility well-performed the condition $\beta > \frac{1}{2(2 - \alpha)}$ on $\beta$ (that is such that $(\frac{1}{n})^\beta$ is the threshold of the truncated quadratic variation) and on the degree of jump activity $\alpha$ was needed (see \cite{Mancini}, \cite{J}). In this paper we theoretically relax this condition and we show that our unbiased estimator achieves excellent numerical results for any couple ($\alpha$, $\beta$).
\end{abstract}

\section{Introduction}
In this paper, we consider the problem of estimating the integrated volatility of a discretely-observed one-dimensional It\^o semimartingale over a finite interval. 
The class of It\^o semimartingales has many applications in various area such as neuroscience, physics and finance. Indeed, it includes the stochastic Morris-Lecar neuron model \cite{5 GLM} as well as important examples taken from finance such as the Barndorff-Nielsen-Shephard model \cite{2 GLM}, the Kou model \cite{13 GLM} and the Merton model \cite{22 GLM}; to name just a few. \\
In this work we aim at estimating the integrated volatility based on discrete observations $X_{t_0}, ... , X_{t_n}$ of the process $X$, with $t_i = i \frac{T}{n}$. Let $X$ be a solution of 
$$X_t= X_0 + \int_0^t b_s ds + \int_0^t a_s dW_s + \int_0^t \int_{\mathbb{R} \backslash \left \{0 \right \}} \gamma(X_{s^-}) \, z \, \tilde{\mu}(ds, dz), \quad t \in \mathbb{R}_+,$$
with $W=(W_t)_{t \ge 0}$ a one dimensional Brownian motion and $\tilde{\mu}$ a compensated Poisson random measure. We also require the volatility $a_t$ to be an It\^o semimartingale.

We consider here the setting of high frequency observations, i.e. $\Delta_n : = \frac{T}{n} \rightarrow 0$ as $n \rightarrow \infty$. We want to estimate $IV := \frac{1}{T}\int_0^T a^2_s f(X_s) ds$, where $f$ is a polynomial growth function. Such a quantity has already been widely studied in the literature because of its great importance in finance. Indeed, taking $f \equiv 1$, $IV$ turns out being the so called integrated volatility that has particular relevance in measuring and forecasting the asset risks; its estimation on the basis of discrete observations of $X$ is one of the long-standing problems. \\
In the sequel we will present some known results denoting by $IV$ the classical integrated volatility, that is we are assuming $f$ equals to 1.

When X is continuous, the canonical way for estimating the integrated volatility is to use the realized volatility or approximate quadratic variation at time T:
$$[X,X]_T^n := \sum_{i =0}^{n -1}(\Delta X_i)^2, \qquad \mbox{where } \Delta X_i = X_{t_{i + 1}} - X_{t_i}.$$
Under very weak assumptions on $b$ and $a$ (namely when $\int_0^T b^2_s ds$ and $\int_0^T a^4_s ds$ are finite for all $t \in (0,T]$), we have a central limit theorem (CLT) with rate $\sqrt{n}$: the processes $\sqrt{n}([X,X]_T^n - IV)$ converge in the sense of stable convergence in law for processes, to a limit $Z$ which is defined on an extension of the space and which conditionally is a centered Gaussian variable whose conditional law is characterized by its (conditional) variance $V_T := 2 \int_0^T a^4_s ds$.

When $X$ has jumps, the variable $[X, X]_T^n$ no longer converges to $IV$. However, there are other known methods to estimate the integrated volatility. \\
The first type of jump-robust volatility estimators are the \textit{Multipower variations} (cf \cite{Power 1}, \cite{Power 2}, \cite{13 in Maillavin}), which we do not explicitly recall here. These estimators satisfy a CLT with rate $\sqrt{n}$ but with a conditional variance bigger than $V_T$ (so they are rate-efficient but not variance-efficient). \\
The second type of volatility estimators, introduced by Jacod and Todorov in \cite{JT}, is based on estimating locally the volatility from the empirical characteristic function of the increments of the process over blocks of decreasing length but containing an increasing number of observations, and then summing the local volatility estimates. \\
Another method to estimate the integrated volatility in jump diffusion processes, introduced by Mancini in \cite{Mancini soglia}, is the use of the \textit{truncated realized volatility} or \textit{truncated quadratic variance} (see \cite{13 in Maillavin}, \cite{Mancini}):
$$\hat{IV}_T^n :=  \sum_{i = 0}^{n -1} (\Delta X_i)^2 1_{\left \{ |\Delta X_i| \le v_n \right \}},$$
where $v_n$ is a sequence of positive truncation levels, typically of the form $(\frac{1}{n})^\beta$ for some $\beta\in (0, \frac{1}{2})$. \\
Below we focus on the estimation of $IV$ through the implementation of the truncated quadratic variation, that is based on the idea of summing only the squared increments of X whose absolute value is smaller than some threshold $v_n$. \\
It is shown in \cite{J} that $\hat{IV}_T^n$ has exactly the same limiting properties as $[X,X]_T^n$ does for some $\alpha \in [0,1)$ and $\beta \in [\frac{1}{2(2 - \alpha)}, \frac{1}{2})$. The index $\alpha$ is the degree of jump activity or Blumenthal-Getoor index
$$\alpha := \inf \left \{ r \in [0,2] : \int_{|x| \le 1} |x|^r F(dx) < \infty \right \},$$
where $F$ is a L\'evy measure which accounts for the jumps of the process and it is such that the compensator $\bar{\mu}$ has the form $\bar{\mu}(dt, dz)= F(z)dz dt$. \\
Mancini has proved in \cite{Mancini} that, when the jumps of $X$ are those of a stable process with index $\alpha \ge 1$, the truncated quadratic variation is such that 
\begin{equation}
(\hat{IV}_T^n - IV) \overset{\mathbb{P}}{\sim} (\frac{1}{n})^{\beta (2 - \alpha)}.
\label{eq: result Mancini}
\end{equation}
This rate is less than $\sqrt{n}$ and no proper CLT is available in this case.

In this paper, in order to estimate $IV := \frac{1}{T}\int_0^T a^2_s f(X_s) ds$, we consider in particular the truncated quadratic variation defined in the following way:  $$Q_n : =  \sum_{i= 0}^{n-1} f(X_{t_i})(X_{t_{i+1}} - X_{t_i})^2 \varphi_{\Delta_{n}^\beta}(X_{t_{i+1}} - X_{t_i}),$$
where $\varphi$ is a $C^\infty$ function that vanishes when the increments of the data are too large compared to the typical increments of a continuous diffusion process, and thus can be used to filter the contribution of the jumps. \\
We aim to extend the results proved in short time in \cite{Mancini} characterising precisely the noise introduced by the presence of jumps and finding consequently some corrections to reduce such a noise. \\
The main result of our paper is the asymptotic expansion for the integrated volatility. Compared to earlier results, our asymptotic expansion provides us precisely the limit to which $n^{\beta(2 - \alpha)}(Q_n - IV)$ converges when $(\frac{1}{n})^{\beta(2 - \alpha)} > \sqrt{n}$, which matches with the condition $\beta < \frac{1}{2(2 - \alpha)}$. \\
Our work extends equation \eqref{eq: result Mancini} (obtained in \cite{Mancini}). Indeed, we find
$$Q_n - IV =  \frac{Z_n}{\sqrt{n}} + (\frac{1}{n})^{\beta (2 - \alpha)} c_\alpha \int_\mathbb{R} \varphi (u) |u|^{1 - \alpha} du \int_0^T |\gamma|^\alpha(X_s) f(X_s) ds + o_\mathbb{P}((\frac{1}{n})^{\beta(2 - \alpha)}),$$
where $Z_n \xrightarrow{\mathcal{L}} N(0, 2 \int_0^T a^4_s f^2(X_s) ds)$ stably with respect to $X$. The asymptotic expansion here above allows us to deduce the behaviour of the truncated quadratic variation for each couple $(\alpha, \beta)$, that is a plus compared to \eqref{eq: result Mancini}. \\
Furthermore, providing we know $\alpha$ (and if we do not it is enough to estimate it previously, see for example \cite{alpha} or \cite{Fabien alpha}), we can improve the performance of the truncated quadratic variation subtracting the bias due to the presence of jumps to the original estimator or taking particular functions $\varphi$ that make the bias derived from the jump part equal to zero. Using the asymptotic expansion of the integrated volatility we also provide the rate of the error left after having applied the corrections. It derives from the Brownian increments mistakenly truncated away, when the truncation is tight.  \\
Moreover, in the case where the volatility is constant, we show numerically that the corrections gained by the knowledge of the asymptotic expansion for the integrated volatility allows us to reduce visibly the noise for any $\beta \in (0, \frac{1}{2})$ and $\alpha \in (0, 2)$. It is a clear improvement because, if the original truncated quadratic variation was a well-performed estimator only if $\beta > \frac{1}{2(2 - \alpha)}$ (condition that never holds for $\alpha \ge 1$), the unbiased truncated quadratic variation achieves excellent results for any couple $(\alpha, \beta)$.

The outline of the paper is the following. In Section 2 we present the assumptions on the process X. In Section 3.1 we define the truncated quadratic variation, while Section 3.2 contains the main results of the paper. In Section 4 we show the numerical performance of the unbiased estimator. The Section 5 is devoted to the statement of propositions useful for the proof of the main results, that is given in Section 6. In Section 7 we give some technical tools about Malliavin calculus, required for the proof of some propositions, while other proofs and some technical results are presented in the Appendix.   

\section{Model, assumptions}\label{S:Model}
The underlying process $X$ is a one dimensionale It\^o semimartingale on the space $(\Omega, \mathcal{F}, (\mathcal{F}_t)_{t \ge 0}, \mathbb{P})$, where $(\mathcal{F}_t)_{t \ge 0}$ is a filtration, and observed at times $t_i = \frac{i}{n}$, for $i = 0, 1, … , n$. \\
Let $X$ be a solution to
\begin{equation}
X_t= X_0 + \int_0^t b_s ds + \int_0^t a_s dW_s + \int_0^t \int_{\mathbb{R} \backslash \left \{0 \right \}} \gamma(X_{s^-}) \, z \, \tilde{\mu}(ds, dz), \quad t \in \mathbb{R}_+,
\label{eq: model}
\end{equation}
where $W=(W_t)_{t \ge 0}$ is a one dimensional Brownian motion and $\tilde{\mu}$ a compensated Poisson random measure on which conditions will be given later. \\ 
We will also require the volatility $a_t$ to be an It\^o semimartingale and it thus can be represented as
\begin{equation}
a_t= a_0 + \int_0^t \tilde{b}_s ds + \int_0^t \tilde{a}_s dW_s + \int_0^t \hat{a}_s d\hat{W}_s + \int_0^t \int_{\mathbb{R} \backslash \left \{0 \right \}} \tilde{\gamma}_s \, z \, \tilde{\mu}(ds, dz) + \int_0^t \int_{\mathbb{R} \backslash \left \{0 \right \}} \hat{\gamma}_s \, z \, \tilde{\mu}_2(ds, dz).
\label{eq: model vol}
\end{equation}
The jumps of $a_t$ are driven by the same Poisson compensated random measure $\tilde{\mu}$ as $X$ plus another Poisson compensated measure $\tilde{\mu}_2$. We need also a second Brownian motion $\hat{W}$: in the case of "pure leverage" we would have $ \hat{a} \equiv 0$ and $\hat{W}$ is not needed; in the case of "no leverage" we rather have $\tilde{a} \equiv 0$. In the mixed case both $W$ and $\hat{W}$ are needed.

\subsection{Assumptions}
The first assumption is a structural assumption describing the driving
terms $W, \hat{W}$, $\tilde{\mu}$ and $\tilde{\mu}_2$; the second one being a set of conditions on the coefficients
implying in particular the existence of the various stochastic integrals involved above. \\ 
\\
\textbf{A1}: The processes $W$ and $\hat{W}$ are two independent Brownian motion, $\mu$ and $\mu_2$ are Poisson random measures on $[0, \infty) \times \mathbb{R}$ associated to the L\'evy processes $L=(L_t)_{t \ge 0}$ and $L_2 = (L^2_t)_{t \ge 0}$ respectively, with $L_t:= \int_0^t \int_\mathbb{R} z \tilde{\mu} (ds, dz)$ and $L^2_t:= \int_0^t \int_\mathbb{R} z \tilde{\mu}_2 (ds, dz)$. The compensated measures are $\tilde{\mu}= \mu - \bar{\mu}$ and $\tilde{\mu}_2= \mu_2 - \bar{\mu}_2$; we suppose that the compensator has the following form: $\bar{\mu}(dt,dz): = F(dz) dt $, $\bar{\mu}_2(dt,dz): = F_2(dz) dt $. Conditions on the Levy measures $F$ and $F_2$ will be given in A3 and A4. The initial condition $X_0$, $a_0$, $W$, $\hat{W}$, $L$ and $L_2$ are independent. The Brownian motions and the L\'evy processes are adapted with respect to the filtration $(\mathcal{F}_t)_{t \ge 0}$. We suppose moreover that there exists $X$, solution of \eqref{eq: model}. \\
\\
\textbf{A2}: The processes $b$, $\tilde{b}$, $\tilde{a}$, $\hat{a}$, $\tilde{\gamma}$, $\hat{\gamma}$ are bounded, $\gamma$ is Lipschitz. The processes $b$, $\tilde{a}$ are c\'adl\'ag adapted, $\gamma$, $\tilde{\gamma}$ and $\hat{\gamma}$ are predictable, $\tilde{b}$ and $\hat{a}$ are progressively measurable. Moreover it exists an $\mathcal{F}_t$ -measurable random variable $K_t$ such that
$$\mathbb{E}[|b_{t + h} - b_t|^2 | \mathcal{F}_t] \le K_t \, |h|; \quad \forall p \ge 1, \, \mathbb{E}[|K_t|^p]< \infty. $$
\\
We observe that the last condition on $b$ holds true regardless if, for example, $b_t = b(X_t)$; $b : \mathbb{R} \rightarrow \mathbb{R}$ Lipschitz. \\
The next assumption ensures the existence of the moments: \\
\\
\textbf{A3}: For all $q >0$, $\int_{|z|> 1} |z|^q F(dz) < \infty$ and $\int_{|z|> 1} |z|^q F_2(dz) < \infty$. Moreover, $\mathbb{E}[|X_0|^q] < \infty$ and $\mathbb{E}[|a_0|^q] < \infty$. \\
\\
\textbf{A4 (Jumps)}:
\begin{enumerate}
\item The jump coefficient $\gamma$ is bounded from below, that is $\inf_{x \in \mathbb{R}}|\gamma(x)|:= \gamma_{min} >0$.
\label{it:1}
\item The L\'evy measures $F$ and $F_2$ are absolutely continuous with respect to the Lebesgue measure and we denote $F(z) = \frac{F(dz)}{dz}$, $F_2(z) = \frac{F_2(dz)}{dz}$.
\label{it:3}
\item The L\'evy measure $F$ satisfies  $F(dz)= \frac{\bar{g}(z)}{|z|^{1 + \alpha}} dz$, where $\alpha \in (0,2)$ and $\bar{g}: \mathbb{R}\rightarrow \mathbb{R}$ is a continuous symmetric nonnegative bounded function with $\bar{g}(0)= 1$.
\label{it:2}
\item The function $\bar{g}$ is differentiable on $\left \{ 0 < |z| \le \eta \right \}$ for some $\eta> 0$ with continuous derivative such that $\sup_{0 < |z| \le \eta} |\frac{\bar{g}'}{\bar{g}}| < \infty$.
\label{it: plus Maillavin}
\item The jump coefficient $\gamma$ is upper bounded, i.e. $\sup_{x \in \mathbb{R}}|\gamma(x)| := \gamma_{max} < \infty$.
\label{it: 4}
\item The Levy measure $F_2$ satisfies $\int_\mathbb{R} |z|^2 F_2(z) dz < \infty$.
\end{enumerate}
The first and the fifth points of the assumptions here above are useful to compare size of jumps of $X$ and $L$. The fourth point is required to use Malliavin calculus and it is satisfied by a large class of processes: $\alpha$- stable process ($\bar{g} = 1$), truncated $\alpha$-stable processes ($\bar{g} = \tau$, a truncation function), tempered stable process ($\bar{g}(z) = e^{- \lambda |z|}$, $\lambda > 0$). \\
In the following, we will use repeatedly some moment inequalities for jump diffusion, which are gathered in Lemma \ref{lemma: Moment inequalities} below and showed in the Appendix.
\begin{lemma}
Suppose that A1 - A4 hold. Then, for all $t > s$, \\
1)for all $p \ge 2$, $\mathbb{E}[|a_t - a_s|^p] \le c |t-s|$; for all $q > 0$ $\sup_{t \in [0,T]} \mathbb{E}[|a_t|^q] < \infty$. \\
2) for all $p \ge 2$, $p \in \mathbb{N}$, $\mathbb{E}[|a_t - a_s|^p|\mathcal{F}_s] \le c|t-s|$. \\
3) for all $p \ge 2$, $\mathbb{E}[|X_t - X_s|^p]^\frac{1}{p} \le c |t-s|^\frac{1}{p}$; for all $q > 0$ $\sup_{t \in [0,T]} \mathbb{E}[|X_t|^q] < \infty$, \\
4) for all $p \ge 2$, $p \in \mathbb{N}$, $\mathbb{E}[|X_t - X_s|^p|\mathcal{F}_s] \le c|t-s|(1 + |X_s|^p)$. \\
5) for all $p \ge 2$, $p \in \mathbb{N}$, $\sup_{h \in [0,1]} \mathbb{E}[|X_{s+h}|^p|\mathcal{F}_s] \le c(1 + |X_s|^p)$. \\
6) for all $p > 1$, $\mathbb{E}[|X_t^c - X_s ^c|^p]^\frac{1}{p} \le |t - s|^\frac{1}{2}$ and $\mathbb{E}[|X_t^c - X_s ^c|^p|\mathcal{F}_s]^\frac{1}{p} \le c |t - s|^\frac{1}{2}(1 + |X_s|^p)$, \\
where we have denoted by $X^c$ the continuous part of the process $X$, which is such that
$$X^c_{t} - X^c_{s} := \int_{s}^{t}a_u dW_u + \int_{s}^{t} b_u du.$$
\label{lemma: Moment inequalities}
\end{lemma}

\section{Setting and main results}\label{S:Construction_and_main}
The process $X$ is observed at regularly spaced times $t_i = i \Delta_n = \frac{i \, T}{n}$ for $i = 0, 1, ... , n$, within
a finite time interval $[0,T]$. We can assume, WLOG, that $T = 1$. \\
Our goal is to estimate the integrated volatility $IV := \frac{1}{T}\int_0^T a^2_s f(X_s) ds$, where $f$ is a polynomial growth function. To do it, we propose the estimator $Q_n $, based on the truncated quadratic variation introduced by Mancini in \cite{Mancini soglia}. Given that the quadratic variation was a good estimator for the integrated volatility in the continuous framework, the idea is to filter the contribution of the jumps and to keep only the intervals in which we judge no jumps happened. We use the size of the increment of the process $X_{t_{i+1}} - X_{t_i}$ in order to judge if a jump occurred or not in the interval $[t_i, t_{i + 1})$: as it is hard for the increment of $X$ with continuous transition to overcome the threshold $\Delta_{n}^\beta = (\frac{1}{n})^\beta$ for $\beta \le \frac{1}{2}$, we can assert the presence of a jump in $[t_i, t_{i + 1})$ if $|X_{t_{i+1}} - X_{t_i}| > \Delta_{n}^\beta $. \\ 
We set
\begin{equation}
Q_n : = \sum_{i= 0}^{n-1} f(X_{t_i})(X_{t_{i+1}} - X_{t_i})^2 \varphi_{\Delta_{n}^\beta}(X_{t_{i+1}} - X_{t_i}),
\label{eq: definition Qn}
\end{equation}
where $$\varphi_{\Delta_{n}^\beta}(X_{t_{i+1}} - X_{t_i}) = \varphi( \frac{X_{t_{i+1}} - X_{t_i}}{\Delta_{n}^\beta}),$$ with $\varphi$ a smooth version of the indicator function, such that
$\varphi(\zeta) = 0$ for each $ \zeta$, with $|\zeta| \ge 2$ and $\varphi(\zeta) = 1$ for each $ \zeta $, with $ |\zeta| \le 1$. \\
It is worth noting that, if we consider an additional constant $k$ in $\varphi$ (that becomes $\varphi_{k \Delta^\beta_{n}}(X_{t_{i+1}} - X_{t_i})= \varphi( \frac{X_{t_{i+1}} - X_{t_i}}{k \Delta_{n}^\beta})$), the only difference is the interval on which the function is $1$ or $0$: it will be $1$ for $|X_{t_{i+1}} - X_{t_i}| \le k \Delta_{n}^\beta$; $0$ for $|X_{t_{i+1}} - X_{t_i}| \ge 2k \Delta_{n}^\beta$. Hence, for shortness in notations, we restrict the theoretical analysis to the situation where $k = 1$ while, for applications, we may take the threshold level as $k\Delta_{n}^\beta$ with $k \neq 1$.
 
\subsection{Main results}
The main result of this paper is the asymptotic expansion for the truncated integrated volatility.\\
We show first of all it is possible to decompose the truncated quadratic variation, separating the continuous part from the contribution of the jumps. We consider right after the difference between the truncated quadratic variation and the discretized volatility, showing it consists on the statistical error (which derives from the continuous part), on a noise term due to the jumps and on a third term which is negligible compared to the other two. From such an expansion it appears clearly the condition on $(\alpha, \beta)$ which specifies whether or not the truncated quadratic variation performs well for the estimation of the integrated volatility. It is also possible to build some unbiased estimators. Indeed, through Malliavin calculus, we identify the main bias term which arises from the presence of the jumps. We study then its asymptotic behavior and, by making it equal to zero or by removing it from the original truncated quadratic variation, we construct some corrected estimators. \\
We define as $\tilde{Q}_n^J$ the jumps contribution present in the original estimator $Q_n$:
\begin{equation}
{\tilde{Q}}_n^J : = n^{\beta (2 - \alpha)} \sum_{i= 0}^{n-1} (\int_{t_i}^{t_{i + 1}}\int_{\mathbb{R} \backslash \left \{0 \right \}} \gamma(X_{s^-}) \, z \, \tilde{\mu}(ds, dz))^2 f(X_{t_i}) \varphi_{\Delta_{n}^\beta}(X_{t_{i+1}} - X_{t_i}).
\label{eq: definition tilde Qn}
\end{equation}
Denoting as $o_\mathbb{P}((\frac{1}{n})^k)$ a quantity such that $\frac{o_\mathbb{P}((\frac{1}{n})^k)}{(\frac{1}{n})^k}\overset{\mathbb{P}}{\rightarrow} 0$,
the following decomposition holds true:
\begin{theorem}
Suppose that A1 - A4 hold and that $\beta \in (0, \frac{1}{2})$ and $\alpha \in (0,2)$ are given in definition \eqref{eq: definition Qn} and in the third point of A4, respectively. Then, as $n \rightarrow \infty$,
\begin{equation}
Q_n = \sum_{i = 0}^{n - 1} f(X_{t_i})(X^c_{t_{i+1}} - X^c_{t_i})^2 + (\frac{1}{n})^{\beta(2 - \alpha)} \tilde{Q}_n^J + \mathcal{E}_n =
\label{eq: Qn parte continua}
\end{equation}
\begin{equation}
= \sum_{i = 0}^{n - 1} f(X_{t_i})(\int_{t_i}^{t_{i + 1}}a_s dW_s)^2 + (\frac{1}{n})^{\beta(2 - \alpha)} \tilde{Q}_n^J + \mathcal{E}_n,
\label{eq: estensione Qn}
\end{equation}
where $\mathcal{E}_n$ is both $o_\mathbb{P}((\frac{1}{n})^{\beta(2 - \alpha)})$ and, for each $\tilde{\epsilon}  > 0$, $o_\mathbb{P}((\frac{1}{n})^{(1 - \alpha \beta - \tilde{\epsilon}) \land (\frac{1}{2} - \tilde{\epsilon})})$. 
\label{th: estensione Qn}
\end{theorem}
To show Theorem \ref{th: estensione Qn} here above, the following lemma will be useful. It illustrates the error we commit when the truncation is tight and therefore the Brownian increments are mistakenly truncated away.
\begin{lemma}
Suppose that A1 - A4 hold. Then, $\forall \epsilon > 0$, 
 $$\sum_{i = 0}^{n - 1} f(X_{t_i})(X^c_{t_{i + 1}} - X^c_{t_i})^2(\varphi_{\Delta_{n}^\beta}(X_{t_{i+1}} - X_{t_i})- 1) = o_\mathbb{P} ((\frac{1}{n})^{1 - \alpha \beta - \epsilon}).$$
\label{lemma: brownian increments}
\end{lemma}
Theorem \ref{th: estensione Qn} anticipates that the size of the jumps part is $(\frac{1}{n})^{\beta(2 - \alpha)}$ (see Theorem \ref{th: reformulation th T fixed}) while the size of the Brownian increments wrongly removed is upper bounded by $(\frac{1}{n})^{1 - \alpha \beta - \epsilon}$ (see Lemma \ref{lemma: brownian increments}). As $\beta \in (0, \frac{1}{2}) $, we can always find an $\epsilon > 0$ such that $1 - \alpha \beta - \epsilon > \beta(2 - \alpha)$ and therefore the bias derived from a tight truncation is always smaller compared to those derived from a loose truncation. 
However, as we will see, after having removed the contribution of the jumps such a small downward bias will represent the main error term if $\alpha \beta > \frac{1}{2} $. \\
In order to eliminate the bias arising from the jumps, we want to identify the term $\tilde{Q}_n^J$ in details. For that purpose we introduce 
\begin{equation}
\hat{Q}_n := (\frac{1}{n})^{\frac{2}{\alpha} - \beta(2 - \alpha)} \sum_{i = 0}^{n-1} f(X_{t_i}) \gamma^2(X_{t_i}) d(\gamma(X_{t_i}) n^{ \beta - \frac{1}{\alpha}}),
\label{eq: definizione finale hat Q}
\end{equation}
where $d(\zeta) : = \mathbb{E}[(S_1^\alpha)^2 \varphi(S_1^\alpha \zeta)]$; $(S_t^\alpha)_{t \ge 0}$ is an $\alpha$-stable process. \\
We want to move from $\tilde{Q}_n^J$ to $\hat{Q}_n$. 
The idea is to move from our process, that in small time behaves like a conditional rescaled L\'evy process, to an $\alpha$ stable distribution.
\begin{proposition}
Suppose that A1 - A4 hold. Let $(S_t^\alpha)_{t \ge 0}$ be an $\alpha$-stable process. Let $g$ be a measurable bounded function such that $\left \| g  \right \|_{pol} := \sup_{x \in \mathbb{R}} (\frac{|g(x)|}{1 + |x|^p}) < \infty$, for some $p\ge 1$, $p\ge \alpha$  hence
\begin{equation}
|g(x)| \le \left \| g  \right \|_{pol} (|x|^p + 1).
\label{eq: conditon on h} 
\end{equation}
Moreover we denote $\left \| g \right \|_\infty: = \sup_{x \in \mathbb{R}} |g(x)|$.
Then, for any $\epsilon > 0$, $0 < h < \frac{1}{2}$,
\begin{equation}
|\mathbb{E}[g(h^{- \frac{1}{\alpha}} L_{h})] - \mathbb{E}[g(S_1^\alpha)]| \le  C_\epsilon h \, |\log(h) | \left \| g \right \|_\infty + C_\epsilon h^\frac{1}{\alpha}\left \| g \right \|_\infty^{1 - \frac{\alpha}{p} - \epsilon} \left \| g \right \|_{pol}^{\frac{\alpha}{p} + \epsilon} |\log(h) |+
\label{eq: tesi prop stable}
\end{equation}
$$+ C_\epsilon h^\frac{1}{\alpha}\left \| g \right \|_\infty^{1+ \frac{1}{p} - \frac{\alpha}{p} + \epsilon} \left \| g \right \|_{pol}^{- \frac{1}{p} + \frac{\alpha}{p} - \epsilon} |\log(h) |1_{\left \{ \alpha > 1 \right \}} ,$$
where $C_\epsilon$ is a constant independent of $h$.
\label{prop: estimation stable}
\end{proposition}
Proposition \ref{prop: estimation stable} requires some Malliavin calculus. 
The proof of Proposition \ref{prop: estimation stable} as well as some technical tools will be found in Section \ref{S:Proof_propositions}. \\
The previous proposition is an extension of Theorem $4.2$ in \cite{Maillavin} and it is useful when $\left \| g \right \|_\infty$ is large, compared to $\left \| g \right \|_{pol}$. For instance, it is the case if consider the function $g(x): = |x|^2 1_{|x| \le M} $ for $M$ large. \\
\\
We need Proposition \ref{prop: estimation stable} to prove the following theorem, in which we consider the difference between the truncated quadratic variation and the discretized volatility. We make explicit its decomposition into the statistical error and the noise term due to the jumps, identified as $\hat{Q}_n$.
\begin{theorem}
Suppose that A1- A4 hold and that $\beta \in (0, \frac{1}{2})$ and $\alpha \in (0,2)$ are given in Definition \ref{eq: definition Qn} and in the third point of A4, respectively. Then, as $n \rightarrow \infty$,
\begin{equation}
Q_n - \frac{1}{n} \sum_{i = 0}^{n - 1} f(X_{t_i}) a^2_{t_i} = \frac{Z_n}{\sqrt{n}}+ (\frac{1}{n})^{\beta(2 - \alpha)}\hat{Q}_n + \mathcal{E}_n, \label{eq:tesi teo 2 e 3}
\end{equation}
where $\mathcal{E}_n$ is always $o_\mathbb{P}((\frac{1}{n})^{\beta(2 - \alpha)})$ and, adding the condition $\beta > \frac{1}{4- \alpha}$, it is also $o_\mathbb{P}((\frac{1}{n})^{(1 - \alpha \beta - \tilde{\epsilon}) \land (\frac{1}{2} - \tilde{\epsilon})})$. Moreover $Z_n \xrightarrow{\mathcal{L}}N(0,2\int_0^T a^4_s f^2(X_s) ds)$ stably with respect to $X$.
\label{th: 2 e 3 insieme}
\end{theorem}
We recognize in the expansion \eqref{eq:tesi teo 2 e 3} the statistical error of model without jumps given by $Z_n$, whose variance is equal to the so called quadricity. As said above, the term $\hat{Q}_n$ is a bias term arising from the presence of jumps and given by \eqref{eq: definizione finale hat Q}. From this explicit expression it is possible to remove the bias term (see Section \ref{S: Applications}). \\
The term $\mathcal{E}_n $ is an additional error term that is always negligible compared to the bias deriving from the jump part $(\frac{1}{n})^{\beta(2 - \alpha)}\hat{Q}_n$ (that is of order $(\frac{1}{n})^{\beta(2 - \alpha)}$ by Theorem \ref{th: reformulation th T fixed} below). \\
The bias term admits a first order expansion that does not require the knowledge of the density of $S^\alpha$.
\begin{proposition}
Suppose that A1 - A4 hold and that $\beta \in (0, \frac{1}{2})$ and $\alpha \in (0,2)$ are given in Definition \ref{eq: definition Qn} and in the third point of Assumption 4, respectively. Then 
\begin{equation}
\hat{Q}_n = \frac{1}{n} c_\alpha \sum_{i = 0}^{n-1} f(X_{t_i}) |\gamma(X_{t_i})|^\alpha (\int_\mathbb{R} \varphi(u) | u| ^{1 - \alpha} du) + \tilde{\mathcal{E}}_n,
\label{eq: limite hat Qn con densita}
\end{equation}
with 
\begin{equation}
c_\alpha =
   \begin{cases}
   \frac{\alpha(1 - \alpha)}{4 \Gamma(2 - \alpha) \cos(\frac{\alpha \pi}{2})} \qquad \mbox{if } \alpha \neq 1, \, \alpha <2 \\
   \frac{1}{2 \pi} \qquad \qquad \qquad \quad \mbox{if }  \, \alpha = 1.
   \end{cases}
\label{eq: def calpha}
\end{equation}
$\tilde{\mathcal{E}}_n = o_\mathbb{P}(1)$ and, if $\alpha < \frac{4}{3}$, it is also $n^{\beta(2 - \alpha)} o_\mathbb{P}((\frac{1}{n})^{(1 - \alpha \beta - \tilde{\epsilon}) \land (\frac{1}{2} - \tilde{\epsilon})}) = o_\mathbb{P}((\frac{1}{n})^{(\frac{1}{2} - 2 \beta + \alpha \beta - \tilde{\epsilon})\land(1 - 2 \beta - \tilde{\epsilon})})$.
\label{prop: conv hat Qn}
\end{proposition}
We have not replaced directly the right hand side of \eqref{eq: limite hat Qn con densita} in \eqref{eq:tesi teo 2 e 3}, observing that $(\frac{1}{n})^{\beta(2 - \alpha)} \tilde{\mathcal{E}}_n = \mathcal{E}_n$, because $(\frac{1}{n})^{\beta(2 - \alpha)} \tilde{\mathcal{E}}_n$ is always $o_\mathbb{P}((\frac{1}{n})^{\beta(2 - \alpha)})$ but to get it is also $o_\mathbb{P}((\frac{1}{n})^{(1 - \alpha \beta - \tilde{\epsilon}) \land (\frac{1}{2} - \tilde{\epsilon})})$ the additional condition $\alpha < \frac{4}{3}$ is required. \\
Proposition \ref{prop: conv hat Qn} provides the contribution of the jumps in detail, identifying a main term. Recalling we are dealing with some bias, it comes naturally to look for some conditions to make it equal to zero and to study its asymptotic behaviour in order to remove its limit.
\begin{corollary}
Suppose that A1 - A4 hold and that $\alpha \in (0, \frac{4}{3})$, $\beta \in (\frac{1}{4 - \alpha}$, $(\frac{1}{2 \alpha} \land \frac{1}{2}))$. If $\varphi$ is such that $\int_\mathbb{R} | u| ^{1 - \alpha}\varphi(u) du = 0$ then, $\forall \tilde{\epsilon} > 0$, 
\begin{equation}
Q_n - \frac{1}{n} \sum_{i = 0}^{n - 1} f(X_{t_i}) a^2_{t_i} = \frac{Z_n}{\sqrt{n}}+ o_\mathbb{P}((\frac{1}{n})^{\frac{1}{2} -\tilde{\epsilon}}),
\label{eq: eq per le appli}
\end{equation}with $Z_n$ defined as in Theorem \ref{th: 2 e 3 insieme} here above.
\label{cor: cond rimozione rumore}
\end{corollary}
It is always possible to build a function $\varphi$ for which the condition here above is respected (see Section \ref{S: Applications}). \\
We have supposed $\alpha < \frac{4}{3}$ in order to say that the error we commit identifying the contribution of the jumps as the first term in the right hand side of \eqref{eq: limite hat Qn con densita} is always negligible compared to the statistical error.
Moreover, taking $\beta < \frac{1}{2 \alpha}$ we get $1 - \alpha \beta > \frac{1}{2}$ and therefore also the bias studied in Lemma \ref{lemma: brownian increments} becomes upper bounded by a quantity which is roughly $o_\mathbb{P}(\frac{1}{\sqrt{n}})$. \\
Equation \eqref{eq: eq per le appli} gives us the behaviour of the unbiased estimator, that is the truncated quadratic variation after having removed the noise derived from the presence of jumps.
 Taking $\alpha$ and $\beta$ as discussed above we have, in other words, reduced the error term $\mathcal{E}_n$ to be $o_\mathbb{P}((\frac{1}{n})^{\frac{1}{2} -\tilde{\epsilon}})$, which is roughly the same size as the statistical error.

We observe that, if $\alpha \ge \frac{4}{3}$ but $\gamma = k \in \mathbb{R}$, the result still holds if we choose $\varphi$ such that
$$\int_\mathbb{R} u^{2}\varphi(u)  \, f_\alpha (\frac{1}{k}u (\frac{1}{n})^{ \beta - \frac{1}{\alpha}})du =0,$$
where $f_\alpha$ is the density of the $\alpha$-stable process. 
Indeed, following \eqref{eq: definizione finale hat Q}, the jump bias $\hat{Q}_n$ is now defined as 
$$ (\frac{1}{n})^{\frac{2}{\alpha} - \beta(2 - \alpha)} \sum_{i = 0}^{n-1} f(X_{t_i}) k^2 d(k \, n^{ \beta - \frac{1}{\alpha}}) = (\frac{1}{n})^{\frac{2}{\alpha} - \beta(2 - \alpha)} \sum_{i = 0}^{n-1} f(X_{t_i}) k^2 \int_\mathbb{R} z^2 \varphi(z k (\frac{1}{n})^{\frac{1}{\alpha} - \beta})f_\alpha(z) dz = $$
$$= (\frac{1}{n})^{\frac{2}{\alpha} - \beta(2 - \alpha)} \sum_{i = 0}^{n-1} f(X_{t_i}) k^2 (\frac{1}{n})^{3 (\beta - \frac{1}{\alpha})} \frac{1}{k^3} \int_\mathbb{R} u^{2}\varphi(u)  \, f_\alpha (\frac{1}{k}u (\frac{1}{n})^{ \beta - \frac{1}{\alpha}})du =0,$$
where we have used a change of variable. \\
\\
Another way to construct an unbiased estimator is to study how the main bias detailed in \eqref{eq: limite hat Qn con densita} asymptotically behaves and to remove it from the original estimator.
\begin{theorem}
Suppose that A1 - A4 hold. Then, as $n \rightarrow \infty$,
\begin{equation}
{\hat{Q}}_n \overset{\mathbb{P}}{\rightarrow} c_\alpha \int_\mathbb{R} \varphi(u) |u|^{1 - \alpha} du \int_0^T |\gamma(X_s)|^\alpha f(X_s) ds.
\label{eq: convergenza hat Q tempo corto}
\end{equation}
Moreover
\begin{equation}
Q_n - IV = \frac{Z_n}{\sqrt{n}}+ (\frac{1}{n})^{\beta(2 - \alpha)}c_\alpha \int_\mathbb{R} \varphi(u) |u|^{1 - \alpha} du \int_0^T |\gamma(X_s)|^\alpha f(X_s) ds + o_\mathbb{P}((\frac{1}{n})^{\beta(2 - \alpha)}),
\label{eq: tesi finale tempo corto}
\end{equation}
where $Z_n \xrightarrow{\mathcal{L}} N(0, 2 \int_0^T a^4_s f^2(X_s) ds)$ stably with respect to $X$.
\label{th: reformulation th T fixed}
\end{theorem}
It is worth noting that, in both \cite{Condition Jacod beta} and \cite{Mancini}, the integrated volatility estimation in short time is dealt and they show that the truncated quadratic variation has rate $\sqrt{n}$ if $\beta > \frac{1}{2(2 - \alpha)}$. \\
We remark that the jump part is negligible compared to the statistic error if $n^{- 1} < n^{- \frac{1}{2 \beta(2 - \alpha)}}$ and so $\beta > \frac{1}{2(2 - \alpha)}$, that is the same condition given in the literature. \\
However, if we take $(\alpha, \beta)$ for which such a condition doesn't hold, we can still use that we know in detail the noise deriving from jumps to implement corrections that still make the unbiased estimator well-performed (see Section \ref{S: Applications}). \\
\\
We require the activity $\alpha$ to be known, for conducting bias correction. If it is unknown, we need to estimate it previously (see for example the methods proposed by Todorov in \cite{alpha} and by Mies in \cite{Fabien alpha}). Then, a question could be how the estimation error in $\alpha$ would affect the rate of the bias-corrected estimator. We therefore assume that $\hat{\alpha}_n = \alpha + O_\mathbb{P} (a_n)$, for some rate sequence $a_n$. Replacing $\hat{\alpha}_n$ in \eqref{eq: tesi finale tempo corto} it turns out that the error derived from the estimation of $\alpha$ does not affect the correction if $a_n (\frac{1}{n})^{\beta(2 - \alpha)} < (\frac{1}{n})^\frac{1}{2}$, which means that $a_n$ has to be smaller than $(\frac{1}{n})^{\frac{1}{2}- \beta(2 - \alpha)}$. We recall that $\beta \in (0, \frac{1}{2})$ and $\alpha \in (0,2)$. Hence, such a condition is not a strong requirement and it becomes less and less restrictive when $\alpha$ gets smaller or $\beta$ gets bigger.

\section{Unbiased estimation in the case of constant volatility} \label{S: Applications}
In this section we consider a concrete application of the unbiased volatility estimator in a jump diffusion model and we investigate its numerical performance. \\
We consider our model \eqref{eq: model} in which we assume, in addition, that the functions $a$ and $\gamma$ are both constants. \\
Suppose that we are given a discrete sample $X_{t_0}, ... , X_{t_n}$ with $t_i = i \Delta_n = \frac{i}{n}$ for $i =0, ... , n$. \\
We now want to analyze the estimation improvement; to do it we compare the classical error committed using the truncated quadratic variation with the unbiased estimation derived by our main results. \\
We define the estimator we are going to use, in which we have clearly taken $f \equiv 1$ and we have introduced a threshold $k$ in the function $\varphi$, so it is 
\begin{equation}
Q_n =  \sum_{i = 0}^{n - 1}(X_{t_{i + 1}} - X_{t_i})^2\varphi_{k \Delta_{n}^\beta}(X_{t_{i + 1}} - X_{t_i}).
\label{eq: Qn applications}
\end{equation}
If normalized, the error committed estimating the volatility is $E_1 : = (Q_n - \sigma^2) \sqrt{n}$. \\
We start from \eqref{eq: limite hat Qn con densita} that in our case, taking into account the presence of $k$, is
\begin{equation}
\hat{Q}_n = c_\alpha \gamma^\alpha k^{2 - \alpha}(\int_\mathbb{R} \varphi(u) |u|^{1 - \alpha} du ) + \tilde{\mathcal{E}}_n.
\label{eq: hatQn applications}
\end{equation}
We now get different methods to make the error smaller. \\
First of all we can replace \eqref{eq: hatQn applications} in \eqref{eq:tesi teo 2 e 3} and so we can reduce the error by subtracting a correction term, building the new estimator $Q_n^c : = Q_n - (\frac{1}{n})^{\beta (2 - \alpha)} c_\alpha \gamma^\alpha k^{2 - \alpha}(\int_\mathbb{R} \varphi(u) |u|^{1 - \alpha} du )$. The error committed estimating the volatility with such a corrected estimator is $E_2 : = (Q_n^c - \sigma^2) \sqrt{n}$. \\
Another approach consists of taking a particular function $\tilde{\varphi}$ that makes the main contribution of $\hat{Q}_n$ equal to $0$. We define $\tilde{\varphi}(\zeta) = \varphi(\zeta) + c \psi(\zeta)$, with $\psi$ a $\mathcal{C}^\infty$ function such that $\psi(\zeta) = 0$ for each $\zeta$, $|\zeta| \ge 2$ or $|\zeta | \le 1$. In this way, for any $c \in \mathbb{R} \setminus \left \{ 0\right \}$, $\tilde{\varphi}$ is still a smooth version of the indicator function such that  $\tilde{\varphi}(\zeta) = 0$ for each $\zeta$,  $|\zeta| \ge 2$ and $\tilde{\varphi}(\zeta) = 1$ for each $\zeta$,  $|\zeta| \le 1$. We can therefore leverage the arbitrariness in $c$ to make the main contribution of $\hat{Q}_n$ equal to zero, choosing $\tilde{c} := - \frac{\int_\mathbb{R} \varphi(u)|u|^{1 - \alpha} du }{\int_\mathbb{R} \psi(u)|u|^{1 - \alpha} du}$, which is such that $\int_\mathbb{R} (\varphi + \tilde{c}\psi(u))|u|^{1 - \alpha} du =0$. \\
Hence, it is possible to achieve an improved estimation of the volatility by used the truncated quadratic variation $Q_{n,c} : = \sum_{i = 0}^{ n - 1} (X_{t_{i + 1}} - X_{t_i})^2 (\varphi + \tilde{c}\psi)(\frac{X_{t_{i + 1}} - X_{t_i}}{ k \Delta_{n}^{\beta}}) $. To make it clear we will analyze the quantity $E_3 : = (Q_{n,c} - \sigma^2) \sqrt{n}$. \\
Another method widely used in numerical analysis to improve the rate of convergence of a sequence is the so-called Richardson extrapolation. We observe that the first term on the right hand side of \eqref{eq: hatQn applications} does not depend on $n$ and so we can just write $\hat{Q}_n = \hat{Q} + \tilde{\mathcal{E}}_n$. Replacing it in \eqref{eq:tesi teo 2 e 3} we get
$$Q_n = \sigma^2 + \frac{Z_n}{\sqrt{n}} + \frac{1}{n^{\beta(2 - \alpha)}} \hat{Q} + \mathcal{E}_n \qquad \mbox{and}$$
$$Q_{2n} = \sigma^2 + \frac{Z_{2n}}{\sqrt{2n}} + \frac{1}{2^{\beta(2 - \alpha)}} \frac{1}{n^{\beta(2 - \alpha)}} \hat{Q} + \mathcal{E}_{2n},$$
where we have also used that $(\frac{1}{n})^{\beta(2 - \alpha)} \tilde{\mathcal{E}}_n = \mathcal{E}_n$. We can therefore use $\frac{Q_n - 2^{\beta (2 - \alpha )} Q_{2 n}}{1 - 2^{\beta(2 - \alpha)}}$ as improved estimator of $\sigma^2$. \\
We give simulation results for $E_1$, $E_2$ and $E_3$ in the situation where $\sigma=1$. The given mean and the deviation standard are each based on 500 Monte Carlo samples. We choose to simulate a tempered stable process (that is $F$ satisfies $F(dz) = \frac{e^{-|z|}}{|z|^{1 + \alpha}}$) in the case $\alpha < 1$ while, in the interest of computational efficiency, we will exhibit results gained from the simulation of a stable L\'evy process in the case $\alpha \ge 1$ ($F(dz) = \frac{1}{|z|^{1 + \alpha}}$). \\
We have taken the smooth functions $\varphi$ and $\psi$ as below:
\begin{equation}
\varphi(x) =
    \begin{cases}
    1 \qquad \mbox{if } |x| < 1 \\
    e^{\frac{1}{3} + \frac{1}{|x|^2 - 4}} \qquad \mbox{if } 1 \le |x| < 2 \\
    0 \qquad \mbox{if } |x| \ge 2
    \end{cases}
\end{equation}
\begin{equation}
\psi_M(x) =
    \begin{cases}
    0 \qquad \quad \mbox{if   } |x| \le 1 \mbox{ or } |x| \ge M \\
    e^{\frac{1}{3} + \frac{1}{|3 - x|^2 - 4}} \qquad \mbox{if } 1 < |x| \le \frac{3}{2} \\
    e^{\frac{1}{|x|^2 -M} - \frac{5}{21} + \frac{4}{4M^2 - 9}} \qquad \mbox{if } \frac{3}{2} < |x| < M ;
    \end{cases}
\end{equation}
choosing opportunely the constant $M$ in the definition of $\psi_M$ we can make its decay slower or faster. We observe that the theoretical results still hold even if the support of $\tilde{\varphi}$ changes as $M$ changes and so it is $[-M, M]$ instead of [-2, 2].  \\ 
Concerning the constant $k$ in the definition of $\varphi$, we fix it equal to $3$ in the simulation of the tempered stable process, while its value is $2$ in the case $\alpha > 1$, $\beta = 0.2$ and, in the case $\alpha > 1$ and $\beta = 0.49$, it increases as $\alpha $ and $\gamma$ increase. \\
The results of the simulations are given in columns 3-6 of Table \ref{tab: beta 0.2} for $\beta = 0.2$ and in columns 3-6 of Table \ref{tab: beta 0.49} for $\beta = 0.49$. \\
\begin{table}[h]
    \begin{subtable}[h]{0.45\textwidth}
        \centering
        \begin{tabular}{||c|c|c|c|c|c||}
        \hline
        \multicolumn{1}{||c|}{\textbf{$\alpha$}}  &
        \multicolumn{1}{c|}{\textbf{$\gamma$}}  &
        \multicolumn{1}{c|}{\textbf{Mean}}  &
        \multicolumn{1}{c|}{\textbf{Rms}}  &
        \multicolumn{1}{c|}{\textbf{Mean}}  &
        \multicolumn{1}{c||}{\textbf{Mean}}  \\
            &   & $E_1$ & $E_1$ &  $E_2$ & $E_3$ \\
        \hline \hline
        0.1 & 1 & 3.820 & 3.177 & 0.831 & 0.189 \\
            & 3 & 5.289 & 3.388 & 1.953 & -0.013 \\ 
        0.5 & 1 & 15.168 & 9.411 & 0.955 & 1.706 \\
            & 3 & 14.445 & 5.726 &2.971 & 0.080 \\
        0.9 & 1 & 13.717 & 4.573 & 4.597 & 0.311 \\
            & 3 & 42.419 & 6.980 & 13.664 & -0.711\\ \hline \hline
        1.2 & 1 & 32.507 & 11.573 &0.069 & 2.137\\
            & 3 & 112.648 & 21.279 &-0.915 & 0.800\\   
        1.5 & 1 & 50.305 & 12.680 & 0.195 &0.923\\
            & 3 & 250.832 & 27.170 & -5.749 &3.557\\      
        1.9 & 1 & 261.066 & 20.729 & -0.530 & 9.139\\
            & 3 & 2311.521 &  155.950 &-0.304 & -35.177 \\ \hline
        \end{tabular}
        \caption{$\beta = 0.2$}
        \label{tab: beta 0.2}
    \end{subtable}
    \hfill
    \begin{subtable}[h]{0.45\textwidth}
        \begin{tabular}{||c|c|c|c|c|c||}
        \hline
        \multicolumn{1}{||c|}{\textbf{$\alpha$}}  &
        \multicolumn{1}{c|}{\textbf{$\gamma$}}  &
         \multicolumn{1}{c|}{\textbf{Mean}}  &
        \multicolumn{1}{c|}{\textbf{Rms}}  &
        \multicolumn{1}{c|}{\textbf{Mean}}  &
        \multicolumn{1}{c||}{\textbf{Mean}}  \\
            &   & $E_1$ & $E_1$ &  $E_2$ & $E_3$ \\
        \hline \hline
        0.1 & 1 & 1.092 & 1.535 & 0.307 &  -0.402 \\
            & 3 & 1.254 & 1.627 & 0.378 & -0.372 \\
        0.5 & 1 & 2.503 & 1.690 & 0.754 & -0.753 \\
            & 3 & 4.680 & 2.146 & 1.651 & -0.824 \\
        0.9 & 1 & 2.909 & 1.548 & 0.217 & 0.416 \\
            & 3 & 8.042 & 1.767 & 0.620 & -0.404\\ \hline \hline
        1.2 & 1 & 7.649 & 1.992 & -0.944 &	-0.185\\
            & 3 & 64.937 & 9.918 & -1.692 & -2.275\\   
        1.5 & 1 & 25.713 & 3.653 & -1.697 & 3.653\\
            & 3 & 218.591 & 21.871 & -4.566 & -13.027 \\      
        1.9 & 1 & 238.379 & 14.860 & -6.826 & 16.330 \\
            & 3 & 2357.553 & 189.231 & 3.827 & -87.353 \\ \hline
        \end{tabular}
        \caption{$\beta = 0.49$}
        \label{tab: beta 0.49}
    \end{subtable}
    \caption{Monte Carlo estimates of $E_1$, $E_2$ and $E_3$ from 500 samples. We have here fixed $n = 700$; $\beta = 0.2$ in the first table and $\beta = 0.49$ in the second one.}
    \label{tab:tabella totale}
\end{table}
\\
It appears that the estimation we get using the truncated quadratic variation performs worse as soon as $\alpha$ and $\gamma$ become bigger (see column 3 in both Tables \ref{tab: beta 0.2} and \ref{tab: beta 0.49}). However, after having applied the corrections, the error seems visibly reduced. A proof of which lies, for example, in the comparison between the error and the root mean square: before the adjustment in both Tables \ref{tab: beta 0.2} and \ref{tab: beta 0.49} the third column dominates the fourth one, showing that the bias of the original estimator dominates the standard deviation while, after the implementation of our main results, we get $E_2$ and $E_3$ for which the bias is much smaller. \\
We observe that for $\alpha < 1$, in both cases $\beta =0.2$ and $\beta = 0.49$, it is possible to choose opportunely M (on which $\psi$'s decay depends) to make the error $E_3$ smaller than $E_2$. On the other hand, for $\alpha >1$, the approach who consists of subtracting the jump part to the error results better than the other, since $E_3$ is in this case generally bigger than $E_2$, but to use this method the knowledge of $\gamma$ is required.
It is worth noting that both the approaches used, that lead us respectively to $E_2$ and $E_3$, work well for any $\beta \in (0, \frac{1}{2})$. \\
We recall that, in \cite{Condition Jacod beta}, the condition found on $\beta$ to get a well-performed estimator was 
\begin{equation}
\beta > \frac{1}{2(2 - \alpha)},
\label{eq: cond jacod on beta}
\end{equation}
that is not respected in the case $\beta = 0.2$. Our results match the ones in \cite{Condition Jacod beta}, since the third column in Table \ref{tab: beta 0.49}  (where $\beta =0.49$) is generally smaller than the third one in Table \ref{tab: beta 0.2} (where $\beta =0.2$). We emphasise nevertheless that, comparing columns 5 and 6 in the two tables, there is no evidence of a dependence on $\beta$ of $E_2$ and $E_3$. \\
The price you pay is that, to implement our corrections, the knowledge of $\alpha$ is request. Such corrections turn out to be a clear improvement also because for $\alpha$ that is less than 1 the original estimator \eqref{eq: Qn applications} is well-performed only for those values of the couple $(\alpha, \beta)$ which respect the condition \eqref{eq: cond jacod on beta} while, for $\alpha \ge 1$, there is no $\beta \in (0, \frac{1}{2})$ for which such a condition can hold. That's the reason why, in the lower part of both Tables \ref{tab: beta 0.2} and \ref{tab: beta 0.49},  $E_1$ is so big. \\
Using our main results, instead, we get $E_2$ and $E_3$ that are always small and so we obtain two corrections which make the unbiased estimator always well-performed without adding any requirement on $\alpha$ or $\beta$.

\section{Preliminary results}\label{S:Propositions}
In the sequel, for $\delta \ge 0$, we will denote as $R_i(\Delta_n^\delta )$ any random variable which is $\mathcal{F}_{t_i}$ measurable and such that, for any $q \ge 1$, 
\begin{equation}
\exists c > 0: \quad  \left \| \frac{R_i(\Delta_n^\delta )}{\Delta_n^\delta } \right \|_{L^q} \le c  < \infty,
\label{eq: definition R}
\end{equation}
with $c$ independent of $i,n$. \\
$R_i$ represent the term of rest and have the following useful property, consequence of the just given definition:
\begin{equation}
R_i(\Delta_{n}^\delta)= \Delta_{n}^\delta R_i(\Delta_{n}^0).
\label{propriety power R}
\end{equation}
We point out that it does not involve the linearity of $R_i$, since the random variables $R_i$ on the left and on the right side are not necessarily the same but only two on which the control (\ref{eq: definition R}) holds with $\Delta_{n}^\delta $ and $\Delta_{n}^0$, respectively. \\ \\
In order to prove the main result, the following proposition will be useful. \\
We define, for $i \in \left \{ 0, ... , n-1 \right \}$,
\begin{equation}
\Delta X_i^J : = \int_{t_i}^{t_{i + 1}}\int_{\mathbb{R} \backslash \left \{0 \right \}} \gamma(X_{s^-}) \, z \, \tilde{\mu}(ds, dz) \qquad \mbox{and} \qquad \Delta \tilde{X}_i^J : = \int_{t_i}^{t_{i + 1}}\int_{\mathbb{R} \backslash \left \{0 \right \}} \gamma(X_{t_i}) \, z \, \tilde{\mu}(ds, dz).
\label{eq: def salti}
\end{equation}
We want to bound the error we commit moving from $\Delta X_i^J$ to $\Delta \tilde{X}_i^J$, denoting as $o_{L^1}(\Delta_{n}^k)$ a quantity such that $ \mathbb{E}_i[|o_{L^1}(\Delta_{n}^k)|] = R_i(\Delta_{n}^k)$, with the notation $\mathbb{E}_i[.] = \mathbb{E}[.|\mathcal{F}_{t_i}]$.
\begin{proposition}
Suppose that A1- A4 hold. Then
\begin{equation}
(\Delta X_i^J)^2 \varphi_{\Delta_{n}^\beta}(\Delta X_i) = (\Delta \tilde{X}_i^J )^2 \varphi_{\Delta_{n}^\beta}(\Delta \tilde{X}_i^J ) + o_{L^1}(\Delta_{n}^{\beta(2 - \alpha) + 1)}),
\label{eq: espansione salti}
\end{equation}
\begin{equation}
(\int_{t_i}^{t_{i+1}}a_s dW_s)\Delta X_i^J \varphi_{\Delta_{n}^\beta}(\Delta X_i) =  (\int_{t_i}^{t_{i+1}}a_s dW_s)\Delta \tilde{X}_i^J\varphi_{\Delta_{n}^\beta}(\Delta \tilde{X}_i^J ) + o_{L^1}(\Delta_{n}^{\beta(2 - \alpha) + 1)}).
\label{eq: salti con browniano}
\end{equation}
Moreover, for each $\tilde{\epsilon} > 0$ and $f$ the function introduced in the definition of $Q_n$,
\begin{equation}
\sum_{i = 0}^{n - 1} f(X_{t_i})(\Delta X_i^J)^2 \varphi_{\Delta_{n}^\beta}(\Delta X_i) =  \sum_{i = 0}^{n - 1} f(X_{t_i})(\Delta \tilde{X}_i^J )^2 \varphi_{\Delta_{n}^\beta}(\Delta \tilde{X}_i^J ) + o_\mathbb{P}(\Delta_n^{(1 - \alpha \beta - \tilde{\epsilon}) \land (\frac{1}{2} - \tilde{\epsilon})}), 
\label{eq: aggiunta prop1 salti}
\end{equation}
\begin{equation}
\sum_{i = 0}^{n - 1} f(X_{t_i})(\int_{t_i}^{t_{i+1}}a_s dW_s)\Delta X_i^J \varphi_{\Delta_{n}^\beta}(\Delta X_i) = \sum_{i = 0}^{n - 1} f(X_{t_i})(\int_{t_i}^{t_{i+1}}a_s dW_s)\Delta \tilde{X}_i^J\varphi_{\Delta_{n}^\beta}(\Delta \tilde{X}_i^J ) + o_{\mathbb{P}}(\Delta_n^{(1 - \alpha \beta - \tilde{\epsilon}) \land (\frac{1}{2} - \tilde{\epsilon})}). 
\label{eq: aggiunta prop1 browniano}
\end{equation}
\label{prop: espansione salti}
\end{proposition}
Proposition \ref{prop: espansione salti} will be showed in the Appendix. \\
In the proof of our main results, also the following lemma will be repeatedly used. 
\begin{lemma}
Let us consider $\Delta X_i^J$ and $\Delta \tilde{X}_i^J$ as defined in \eqref{eq: def salti}. Then 
\begin{enumerate}
    \item For each $q \ge 2$ $\exists \epsilon > 0$ such that
    \begin{equation}
    \mathbb{E}[|\Delta X_i^J 1_{\left \{ |\Delta X_i^J| \le 4 \Delta_{n}^\beta \right \}}|^q |\mathcal{F}_{t_i} ] = R_i(\Delta_{n}^{1 + \beta(q - \alpha)}) = R_i(\Delta_{n}^{1 + \epsilon}).
    \label{eq: estensione salti lemma 10}
    \end{equation}
    \begin{equation}
    \mathbb{E}[|\Delta \tilde{X}_i^J 1_{\left \{ |\Delta \tilde{X}_i^J| \le 4 \Delta_{n}^\beta \right \}}|^q |\mathcal{F}_{t_i} ] = R_i( \Delta_{n}^{1 + \beta(q - \alpha)}) = R_i(\Delta_{n}^{1 + \epsilon}).
    \label{eq: estensione tilde salti lemma 10}
    \end{equation}
    \item For each $q \ge 1$ we have
    \begin{equation}
    \mathbb{E}[|\Delta X_i^J 1_{\left \{ \frac{\Delta_{n}^\beta}{4} \le |\Delta X_i^J| \le 4 \Delta_{n}^\beta \right \}}|^{q} |\mathcal{F}_{t_i} ] = R_i( \Delta_{n}^{1 + \beta(q - \alpha)}).
    \label{eq: estensione q = 1 + epsilon lemma 10}
    \end{equation}
\end{enumerate}
\begin{proof}
Reasoning as in Lemma 10 in \cite{Chapitre 1} we easily get \eqref{eq: estensione salti lemma 10}. Observing that $\Delta \tilde{X}_i^J$ is a particular case of $\Delta X_i^J$ where $\gamma$ is fixed, evaluated in $X_{t_i}$, it follows that \eqref{eq: estensione tilde salti lemma 10} can be obtained in the same way of \eqref{eq: estensione salti lemma 10}. Using the bound on $\Delta X_i^J$ obtained from the indicator function we get that the left hand side of \eqref{eq: estensione q = 1 + epsilon lemma 10} is upper bounded by 
$$c \Delta_{n}^{\beta q} \mathbb{E}[ 1_{\left \{ \frac{\Delta_{n}^\beta}{4} \le |\Delta X_i^J| \le 4 \Delta_{n}^\beta \right \}} |\mathcal{F}_{t_i} ] \le \Delta_{n}^{\beta q} R_i(\Delta_{n}^{1 - \alpha \beta}),$$
where in the last inequality we have used Lemma 11 in \cite{Chapitre 1} on the interval $[t_i, t_{i + 1}]$ instead of on $[0, h]$. From property \eqref{propriety power R} of $R_i$ we get \eqref{eq: estensione q = 1 + epsilon lemma 10}.
\end{proof}
\label{lemma: estensione 10 capitolo 1}
\end{lemma}

\section{Proof of main results} \label{S: Proof main}
We show Lemma \ref{lemma: brownian increments}, required for the proof of Theorem \ref{th: estensione Qn}.

\subsection{Proof of Lemma \ref{lemma: brownian increments}.}
\begin{proof}
By the definition of $X^c$ we have 
$$|\sum_{i = 0}^{n - 1} f(X_{t_i})(X^c_{t_{i + 1}} - X^c_{t_i})^2(\varphi_{\Delta_{n}^\beta}(\Delta X_i)- 1)| \le$$
$$ \le c \sum_{i = 0}^{n - 1} |f(X_{t_i})|\big(|\int_{t_i}^{t_{i + 1}}a_s dW_s|^2 + |\int_{t_i}^{t_{i + 1}} b_s ds|^2\big)|\varphi_{\Delta_{n}^\beta}(\Delta X_i)- 1| = : |I_{2,1}^n| + |I_{2,2}^n|. $$
In the sequel the constant $c$ may change value from line to line. \\
Concerning $I_{2,1}^n$, using Holder inequality we have
\begin{equation}
\mathbb{E}[|I_{2,1}^n|] \le c \sum_{i = 0}^{n - 1} \mathbb{E}[|f(X_{t_i})|\mathbb{E}_i[|\int_{t_i}^{t_{i + 1}}a_s dW_s|^{2p}]^\frac{1}{p} \mathbb{E}_i[|\varphi_{\Delta_{n}^\beta}(\Delta X_i)- 1|^q]^\frac{1}{q}],
\label{eq: I21 start}
\end{equation}
where $\mathbb{E}_i$ is the conditional expectation wit respect to $\mathcal{F}_{t_i}$. \\
We now use Burkholder-Davis-Gundy inequality to get, for $p_1 \ge 2$,
\begin{equation}
\mathbb{E}_i[|\int_{t_i}^{t_{i+1}}a_s dW_s|^{p_1}]^\frac{1}{p_1} \le \mathbb{E}_i[|\int_{t_i}^{t_{i+1}}a^2_s ds|^\frac{p_1}{2}]^\frac{1}{p_1} \le R_i(\Delta_{n}^\frac{p_1}{2})^\frac{1}{p_1} = R_i(\Delta_{n}^\frac{1}{2}),
\label{eq: bdg}
\end{equation}
where in the last inequality we have used that $a^2_s$ has bounded moments as a consequence of Lemma \ref{lemma: Moment inequalities}. We now observe that, from the definition of $\varphi$ we know that $\varphi_{\Delta_{n}^\beta}(\Delta X_i)- 1$ is different from $0$ only if $|\Delta X_i| > \Delta_{n}^\beta$. We consider two different sets: $|\Delta X_i^J| < \frac{1}{2} \Delta_{n}^\beta$ and $|\Delta X_i^J| \ge \frac{1}{2} \Delta_{n}^\beta$. We recall that $\Delta X_i = \Delta X_i^c + \Delta X_i^J$ and so, if $|\Delta X_i| > \Delta_{n}^\beta$ and $|\Delta X_i^J| < \frac{1}{2} \Delta_{n}^\beta$, then it means that $|\Delta X_i^c|$ must be more than $\frac{1}{2} \Delta_{n}^\beta$.
Using a conditional version of Tchebychev inequality we have that, $\forall r > 1$,
\begin{equation}
\mathbb{P}_i(|\Delta X^c_i| \ge \frac{1}{2} \Delta_{n}^\beta ) \le c \frac{\mathbb{E}_i[|\Delta X^c_i|^r]}{\Delta_{n}^{\beta r }} \le R_i(\Delta_{n}^{(\frac{1}{2} - \beta) r}),
\label{eq: prob parte continua}
\end{equation}
where $\mathbb{P}_i$ is the conditional probability with respect to $\mathcal{F}_{t_i}$; the last inequality follows from the sixth point of Lemma \ref{lemma: Moment inequalities}. If otherwise $|\Delta X_i^J| \ge \frac{1}{2} \Delta_{n}^\beta$, then we introduce the set $N_{i,n}: = \left \{ |\Delta L_s| \le \frac{2 \Delta_{n}^\beta}{\gamma_{min}}; \forall s \in (t_i, t_{i + 1}]  \right \}$. We have
$\mathbb{P}_i(\left \{|\Delta X_i^J| \ge \frac{1}{2} \Delta_{n}^\beta \right \} \cap (N_{i,n})^c) \le \mathbb{P}_i((N_{i,n})^c)$, with 
\begin{equation}
\mathbb{P}_i((N_{i,n})^c) = \mathbb{P}_i(\exists s \in (t_i, t_{i+ 1}] : |\Delta L_s| > \frac{\Delta_{n}^\beta}{2\gamma_{min}} ) \le c \int_{t_i}^{t_{i+1}} \int_{\frac{ \Delta_{n}^\beta}{2\gamma_{min}}}^\infty F(z) dz ds \le c \Delta_{n}^{1 - \alpha \beta},
\label{eq: proba Nin c}
\end{equation}
where we have used the third point of A4. Furthermore, using Markov inequality,
\begin{equation}
\mathbb{P}_i(\left \{|\Delta X_i^J| \ge \frac{1}{2} \Delta_{n}^\beta \right \} \cap N_{i,n}) \le c \mathbb{E}_i[|\Delta X_i^J|^r 1_{N_{i,n}}] \Delta_{n}^{- \beta r} \le R_i( \Delta_{n}^{- \beta r + 1 + \beta(r - \alpha)}) = R_i(\Delta_{n}^{1 - \beta \alpha}),
\label{eq: salti con Nin}
\end{equation}
where we have used the first point of Lemma \ref{lemma: estensione 10 capitolo 1}, observing that $1_{N_{i,n}}$ acts like the indicator function in \eqref{eq: estensione salti lemma 10} (see also (219) in \cite{Chapitre 1}). 
 Now using \eqref{eq: prob parte continua}, \eqref{eq: proba Nin c}, \eqref{eq: salti con Nin} and the arbitrariness of $r$ we have 
\begin{equation}
\mathbb{P}_i(|\Delta X_i| > \Delta_{n}^\beta) = \mathbb{P}_i(|\Delta X_i| > \Delta_{n}^\beta, |\Delta X_i^J| < \frac{1}{2} \Delta_{n}^\beta ) + \mathbb{P}_i(|\Delta X_i| > \Delta_{n}^\beta, |\Delta X_i^J| \ge \frac{1}{2} \Delta_{n}^\beta) \le  R_i(\Delta_{n}^{1 - \alpha \beta}).
\label{eq: proba varphi diverso da 1}
\end{equation}
Taking $p$ big and $q$ next to $1$ in \eqref{eq: I21 start} and replacing there \eqref{eq: bdg} with $p_1 = 2p$ and \eqref{eq: proba varphi diverso da 1} we get, $\forall \epsilon > 0$,
$$n^{1 - \alpha \beta - \tilde{\epsilon}}\mathbb{E}[|I_{2,1}^n|] \le n^{1 - \alpha \beta - \tilde{\epsilon}} c \sum_{i=1}^{n - 1}\mathbb{E}[|f(X_{t_i})| R_i(\Delta_{n}) R_i(\Delta_{n}^{1 - \alpha \beta - \epsilon})] \le (\frac{1}{n})^{\tilde{\epsilon} - \epsilon} \frac{c}{n} \sum_{i=1}^{n - 1} \mathbb{E}[|f(X_{t_i})| R_i(1)].  $$
Now, for each $\tilde{\epsilon} > 0$, we can always find an $\epsilon$ smaller than it, that is enough to get that $\frac{I_{2,1}^n}{(\frac{1}{n})^{1 - \alpha \beta - \tilde{\epsilon}}}$ goes to zero in $L^1$ and so in probability.
Let us now consider $I_{2,2}^n$. We recall that $b$ is uniformly bounded by a constant, therefore
\begin{equation}
(\int_{t_i}^{t_{i + 1}} b_s ds)^2 \le c \Delta_{n}^2.
\label{eq: stima I22}
\end{equation}
Acting moreover on $|\varphi_{\Delta_{n,i}^\beta}(\Delta X_i)- 1|$ as we did here above it follows
$$n^{1 - \alpha \beta - \tilde{\epsilon}}\mathbb{E}[|I_{2,2}^n|] \le n^{1 - \alpha \beta - \tilde{\epsilon}} c \sum_{i=1}^{n - 1}\mathbb{E}[|f(X_{t_i})| R_i(\Delta_{n}^2) R_i(\Delta_{n}^{1 - \alpha \beta - \epsilon})] \le (\frac{1}{n})^{ 1 + \tilde{\epsilon} - \epsilon} \frac{c}{n} \sum_{i=1}^{n - 1} \mathbb{E}[|f(X_{t_i})| R_i(1)]  $$
and so $I_{2,2}^n = o_\mathbb{P}((\frac{1}{n})^{1 - \alpha \beta - \tilde{\epsilon}})$.
\end{proof}

\subsection{Proof of Theorem \ref{th: estensione Qn}.}
We observe that, using the dynamic \eqref{eq: model} of $X$ and the definition of the continuous part $X^c$, we have that
\begin{equation}
X_{t_{i + 1}} - X_{t_i} = (X^c_{t_{i + 1}} - X^c_{t_i}) + \int_{t_i}^{t_{i + 1}}\int_{\mathbb{R} \backslash \left \{0 \right \}} \gamma(X_{s^-}) \, z \, \tilde{\mu}(ds, dz).
\label{eq: incremento X funzione di Xc}
\end{equation}
Replacing \eqref{eq: incremento X funzione di Xc} in definition \eqref{eq: definition Qn} of $Q_n$ we have
$$Q_n = \sum_{i = 0}^{n - 1} f(X_{t_i})(X^c_{t_{i + 1}} - X^c_{t_i})^2 +  \sum_{i = 0}^{n - 1} f(X_{t_i})(X^c_{t_{i + 1}} - X^c_{t_i})^2(\varphi_{\Delta_{n}^\beta}(\Delta X_i)- 1)+  $$
\begin{equation}
+ 2 \sum_{i = 0}^{n - 1} f(X_{t_i})(X^c_{t_{i + 1}} - X^c_{t_i})(\Delta X_i^J) \varphi_{\Delta_{n}^\beta}(\Delta X_i)  +  \sum_{i = 0}^{n - 1} f(X_{t_i})(\Delta X_i^J)^2\varphi_{\Delta_{n}^\beta}(\Delta X_i)= : \sum_{j= 1}^4 I_j^n.
\label{eq: riformulazione Qn}
\end{equation}
Comparing \eqref{eq: riformulazione Qn} with \eqref{eq: Qn parte continua}, using also definition \eqref{eq: definition tilde Qn} of $\tilde{Q}_n$, it follows that our goal is to show that $I_2^n + I_3^n = \mathcal{E}_n$, that is both $o_\mathbb{P}(\Delta_n^{\beta(2 - \alpha)})$ and $o_\mathbb{P}(\Delta_n^{(1 - \alpha \beta - \tilde{\epsilon}) \land (\frac{1}{2} - \tilde{\epsilon})})$. We have already shown in Lemma \ref{lemma: brownian increments} that $I_2^n = o_{\mathbb{P}}(\Delta_n^{1 - \alpha \beta - \tilde{\epsilon}})$. As $(1 - \alpha \beta - \tilde{\epsilon}) \land (\frac{1}{2} - \tilde{\epsilon}) < 1 - \alpha \beta - \tilde{\epsilon}$ and $\beta(2 - \alpha) < 1 - \alpha \beta - \tilde{\epsilon}$, we immediately get $I_2^n = \mathcal{E}_n$. \\
Let us now consider $I_3^n$. From the definition of the process $(X_t^c)$ it is
$$ 2 \sum_{i = 0}^{n - 1} f(X_{t_i})[\int_{t_i}^{t_{i+ 1}} b_s ds + \int_{t_i}^{t_{i+ 1}} a_s dW_s] \Delta X_i^J \varphi_{\Delta_{n}^\beta}(\Delta X_i)= : I_{3,1}^n + I_{3,2}^n.$$
We use on $I_{3,1}^n$ Cauchy-Schwartz inequality, \eqref{eq: stima I22} and Lemma 10 in \cite{Chapitre 1}, getting
$$\mathbb{E}[|I_{3,1}^n|] \le 2 \sum_{i = 0}^{n - 1} \mathbb{E}[ |f(X_{t_i})| R_i(\Delta_{n}^{1 + \beta(2 - \alpha)})^\frac{1}{2} R_i( \Delta_{n}^2)^\frac{1}{2}] \le \Delta_n^{\frac{1}{2} + \frac{\beta}{2}(2 - \alpha)} \frac{1}{n} \sum_{i = 0}^{n - 1} \mathbb{E}[|f(X_{t_i})| R_i(1)], $$
where we have also used property \eqref{propriety power R} on $R$. We observe it is $\frac{1}{2} + \beta -\frac{\alpha \beta}{2} > \frac{1}{2}$ if and only if $\beta(1 - \frac{\alpha}{2}) > 0$, that is always true. We can therefore say that $I_{3,1}^n = o_\mathbb{P}(\Delta_n^\frac{1}{2})$ and so
\begin{equation}
I_{3,1}^n = o_\mathbb{P}(\Delta_n^{(\frac{1}{2} - \tilde{\epsilon}) \land (1 - \alpha \beta - \tilde{\epsilon})}).
\label{eq: I31 nuovo}
\end{equation}
Moreover,
\begin{equation}
\frac{\mathbb{E}[|I_{3,1}^n|]}{\Delta_n^{\beta(2 - \alpha)}} \le \Delta_n^{\frac{1}{2} - \beta + \frac{\alpha \beta}{2}} \frac{1}{n}\sum_{i = 0}^{n - 1} \mathbb{E}[|f(X_{t_i})| R_i( 1)], 
\label{eq: conv I31}
\end{equation}
that goes to zero using the polynomial growth of $f$, the definition of $R$, the fifth point of Lemma \ref{lemma: Moment inequalities}. Moreover, we have observed that the exponent on $\Delta_n$ is positive for $\beta < \frac{1}{2} \frac{1}{(1 - \frac{\alpha}{2})}$, that is always true. \\
Concerning $I_{3,2}^n$, we start proving that $I_{3,2}^n = o_\mathbb{P}(\Delta_n^{\beta(2 - \alpha)})$. From \eqref{eq: salti con browniano} in Proposition \ref{prop: espansione salti} we have
\begin{equation}
\frac{I_{3,2}^n}{\Delta_n^{\beta(2 - \alpha)}} = \frac{2}{\Delta_n^{\beta(2 - \alpha)}} \sum_{i = 0}^{n - 1} f(X_{t_i}) \Delta \tilde{X}_i^J \varphi_{\Delta_{n}^\beta}(\Delta \tilde{X}_i^J) \int_{t_i}^{t_{i + 1}} a_s dW_s + \frac{2}{\Delta_n^{\beta(2 - \alpha)}} \sum_{i = 0}^{n - 1} f(X_{t_i}) o_{L^1}(\Delta_{n}^{\beta(2 - \alpha) + 1}).
\label{eq: riformulo dopo prop 1}
\end{equation}
By the definition of $o_{L^1}$ the last term here above goes to zero in norm $1$ and so in probability. The first term of \eqref{eq: riformulo dopo prop 1} can be seen as 
\begin{equation}
\frac{2}{\Delta_n^{\beta(2 - \alpha)}} \sum_{i = 0}^{n - 1} f(X_{t_i}) \Delta \tilde{X}_i^J \varphi_{\Delta_{n}^\beta}(\Delta \tilde{X}_i^J) [\int_{t_i}^{t_{i + 1}} a_{t_i} dW_s + \int_{t_i}^{t_{i + 1}} (a_s- a_{t_i}) dW_s].
\label{eq: main I32}
\end{equation}
On the first term of \eqref{eq: main I32} here above we want to use Lemma 9 of \cite{Genon Catalot} in order to get that it converges to zero in probability, so we have to show the following:
\begin{equation}
\frac{2}{\Delta_n^{\beta(2 - \alpha)}} \sum_{i = 0}^{n - 1} \mathbb{E}_i[f(X_{t_i}) \Delta \tilde{X}_i^J \varphi_{\Delta_{n}^\beta}(\Delta \tilde{X}_i^J)\int_{t_i}^{t_{i + 1}} a_{t_i} dW_s] \xrightarrow{\mathbb{P}} 0,
\label{eq: tesi 1 Genon Catalot}
\end{equation}
\begin{equation}
\frac{4}{\Delta_n^{2\beta(2 - \alpha)}} \sum_{i = 0}^{n - 1} \mathbb{E}_i[f^2(X_{t_i}) (\Delta \tilde{X}_i^J)^2 \varphi^2_{\Delta_{n}^\beta}(\Delta \tilde{X}_i^J)(\int_{t_i}^{t_{i + 1}} a_{t_i} dW_s)^2] \xrightarrow{\mathbb{P}} 0,
\label{eq: tesi 2 Genon Catalot}
\end{equation}
where $\mathbb{E}_i[.] = \mathbb{E}[. | \mathcal{F}_{t_i}]$. \\
Using the independence between $W$ and $L$ we have that the left hand side of \eqref{eq: tesi 1 Genon Catalot} is
\begin{equation}
\frac{2}{\Delta_n^{\beta(2 - \alpha)}} \sum_{i = 0}^{n - 1} f(X_{t_i}) \mathbb{E}_i[ \Delta \tilde{X}_i^J \varphi_{\Delta_{n}^\beta}(\Delta \tilde{X}_i^J)] \mathbb{E}_i[\int_{t_i}^{t_{i + 1}} a_{t_i} dW_s] = 0.
\label{eq: I32 centrato}
\end{equation}
Now, in order to prove \eqref{eq: tesi 2 Genon Catalot}, we use Holder inequality with $p$ big and $q$ next to $1$ on its left hand side, getting it is upper bounded by
$$\Delta_n^{ - 2 \beta(2 - \alpha)} \sum_{i = 0}^{n - 1} f^2(X_{t_i}) \mathbb{E}_i[(\int_{t_i}^{t_{i + 1}} a_{t_i} dW_s)^{2p}]^\frac{1}{p} \mathbb{E}_i[ |\Delta \tilde{X}_i^J \varphi_{\Delta_{n}^\beta}(\Delta \tilde{X}_i^J)|^{2q}]^\frac{1}{q} \le $$
\begin{equation}
\le \Delta_n^{ - 2 \beta(2 - \alpha)} \sum_{i = 0}^{n - 1} f^2(X_{t_i})R_i( \Delta_{n})R_i( \Delta_{n}^{\frac{1}{q} + \frac{\beta}{q}(2q - \alpha)}) \le \Delta_n^{1 - 2 \beta(2 - \alpha) + 2 \beta - \alpha \beta - \epsilon} \frac{1}{n} \sum_{i = 0}^{n - 1} f^2(X_{t_i}) R_i( 1),
\label{eq: fine per tesi 2 Catalot}
\end{equation}
where we have used \eqref{eq: bdg}, \eqref{eq: estensione tilde salti lemma 10} and property \eqref{propriety power R} of $R$. We observe that the exponent on $\Delta_n$ is positive if $\beta < \frac{1}{2 - \alpha} - \epsilon$ and we can always find an $\epsilon >0$ such that it is true. Hence \eqref{eq: fine per tesi 2 Catalot} goes to zero in norm $1$ and so in probability. \\
Concerning the second term of \eqref{eq: main I32}, using Cauchy-Schwartz inequality and \eqref{eq: estensione tilde salti lemma 10} we have 
$$\mathbb{E}_i[|\Delta \tilde{X}_i^J \varphi_{\Delta_{n}^\beta}(\Delta \tilde{X}_i^J)| |\int_{t_i}^{t_{i + 1}} [a_s- a_{t_i}] dW_s|] \le \mathbb{E}_i[|\Delta \tilde{X}_i^J \varphi_{\Delta_{n}^\beta}(\Delta \tilde{X}_i^J)|^2]^\frac{1}{2} \mathbb{E}_i[|\int_{t_i}^{t_{i + 1}} [a_s- a_{t_i}] dW_s|^2]^\frac{1}{2} \le $$
\begin{equation}
\le R_i(\Delta_{n}^{\frac{1}{2} + \frac{\beta}{2}(2 - \alpha)}) \mathbb{E}_i[\int_{t_i}^{t_{i + 1}} |a_s - a_{t_i}|^2 ds]^\frac{1}{2} \le \Delta_{n}^{\frac{1}{2} + \frac{\beta}{2}(2 - \alpha)}R_i( 1) \Delta_n \le \Delta_{n,i}^{\frac{3}{2} + \frac{\beta}{2}(2 - \alpha)}R_i(1),
\label{eq: altra parte I32}
\end{equation}
where we have also used the second point of Lemma \ref{lemma: Moment inequalities} and the property \eqref{propriety power R} of $R$. Replacing \eqref{eq: altra parte I32} in the second term of \eqref{eq: main I32} we get it is upper bounded in norm 1 by
\begin{equation}
\Delta_n^{\frac{1}{2} - \beta + \frac{\alpha \beta}{2}} \frac{1}{n} \sum_{i = 0}^{n-1} \mathbb{E}[ |f(X_{t_i})| R_i( 1)],
\label{eq: fine I32}
\end{equation}
that goes to zero since the exponent on $\Delta_n$ is more than $0$ for $\beta < \frac{1}{2} \frac{1}{(1 - \frac{\alpha}{2})}$, that is always true. Using
\eqref{eq: riformulo dopo prop 1} - \eqref{eq: tesi 2 Genon Catalot} and \eqref{eq: fine I32} we get
\begin{equation}
\frac{I_{3,2}^n}{\Delta_n^{\beta(2 - \alpha)}} \xrightarrow{\mathbb{P}} 0.
\label{eq: convergence I32}
\end{equation}
We now want to show that $I_{3,2}^n$ is also $o_\mathbb{P}(\Delta_n^{(\frac{1}{2} - \tilde{\epsilon}) \land (1 - \alpha \beta - \tilde{\epsilon})})$. \\
Using \eqref{eq: aggiunta prop1 browniano} in Proposition \ref{prop: espansione salti} we get it is enough to prove that
\begin{equation}
 \frac{1}{\Delta_n^{\frac{1}{2} - \tilde{\epsilon}}} \sum_{i = 0}^{n - 1}f(X_{t_i}) [\Delta \tilde{X}_i^J \varphi_{\Delta_{n}^\beta}(\Delta \tilde{X}_i^J) \int_{t_i}^{t_{i + 1}} a_s dW_s ]\xrightarrow{\mathbb{P}} 0,
\label{eq: I32 nuovo inizio}
\end{equation}
where the left hand side here above can be seen as \eqref{eq: main I32}, with the only difference that now we have $\Delta_n^{\frac{1}{2} - \tilde{\epsilon}}$ instead of $\Delta_n^{\beta(2 - \alpha)}$. We have again, acting like we did in \eqref{eq: I32 centrato} and \eqref{eq: fine per tesi 2 Catalot},
\begin{equation}
\frac{2}{\Delta_n^{\frac{1}{2} - \tilde{\epsilon}}}  \sum_{i = 0}^{n - 1} f(X_{t_i}) \mathbb{E}_i [\Delta \tilde{X}_i^J \varphi_{\Delta_{n}^\beta}(\Delta \tilde{X}_i^J) \int_{t_i}^{t_{i + 1}} a_{t_i} dW_s ] \xrightarrow{\mathbb{P}} 0
\label{eq: I32 nuovo conv}
\end{equation}
and 
\begin{equation}
\frac{4}{\Delta_n^{2(\frac{1}{2} - \tilde{\epsilon})}} \sum_{i = 0}^{n - 1} \mathbb{E}_i[f^2(X_{t_i}) (\Delta \tilde{X}_i^J)^2 \varphi^2_{\Delta_{n}^\beta}(\Delta \tilde{X}_i^J)(\int_{t_i}^{t_{i + 1}} a_{t_i} dW_s)^2] \le \Delta_n^{2 \tilde{\epsilon} + 2 \beta - \alpha \beta - \epsilon} \frac{1}{n} \sum_{i = 0}^{n - 1} f^2(X_{t_i}) R_i(1),
\label{eq: I32 nuovo carre}
\end{equation}
that goes to zero in norm 1 and so in probability. Using also \eqref{eq: altra parte I32} we have that
\begin{equation}
\frac{2}{\Delta_n^{\frac{1}{2} - \tilde{\epsilon}}}  \sum_{i = 0}^{n - 1} \mathbb{E}_i[ |f(X_{t_i}) \Delta \tilde{X}_i^J \varphi_{\Delta_{n}^\beta}(\Delta \tilde{X}_i^J) \int_{t_i}^{t_{i + 1}}[ a_s - a_{t_i}] dW_s|] \le \Delta_n^{\frac{\beta}{2}(2 - \alpha) + \tilde{\epsilon}} \frac{1}{n} \sum_{i = 0}^{n - 1} |f(X_{t_i})| R_i(1),
\label{eq: I32 nuovo fine}
\end{equation}
that, again, goes to zero in norm 1 and so in probability since the exponent on $\Delta_n$ is always positive. Using \eqref{eq: I32 nuovo inizio} - \eqref{eq: I32 nuovo fine} we get $I_{3,2}^n = o_\mathbb{P}(\Delta_n^{\frac{1}{2} - \tilde{\epsilon}})$ and so
\begin{equation}
I_{3,2}^n = o_\mathbb{P}(\Delta_n^{(\frac{1}{2} - \tilde{\epsilon}) \land (1 - \alpha \beta - \tilde{\epsilon})}).
\label{eq: finale I32 nuovo}
\end{equation}
From Lemma \ref{lemma: brownian increments}, \eqref{eq: I31 nuovo}, \eqref{eq: conv I31}, \eqref{eq: convergence I32} and \eqref{eq: finale I32 nuovo} it follows \eqref{eq: Qn parte continua}. \\
\\
Now, in order to prove \eqref{eq: estensione Qn}, we recall the definition of $X_t^c$:
\begin{equation}
X^c_{t_{i + 1}} - X^c_{t_i} = \int_{t_i}^{t_{i + 1}} b_s ds + \int_{t_i}^{t_{i + 1}} a_s dW_s.
\label{eq: definition Xc}
\end{equation}
Replacing \eqref{eq: definition Xc} in \eqref{eq: Qn parte continua} and comparing it with \eqref{eq: estensione Qn} it follows that our goal is to show that 
$$A_1^n + A_2^n : =  \sum_{i = 0}^{n - 1} f(X_{t_i}) (\int_{t_i}^{t_{i + 1}} b_s ds)^2 + 2 \sum_{i = 0}^{n-1} f(X_{t_i}) (\int_{t_i}^{t_{i + 1}} b_s ds)(\int_{t_i}^{t_{i + 1}} a_s dW_s) = \mathcal{E}_n.  $$
Using \eqref{eq: stima I22} and property \eqref{propriety power R} of $R$ we know that 
\begin{equation}
\frac{\mathbb{E}[|A_1^n|]}{\Delta_n^{\beta(2 - \alpha)}} \le \frac{1}{\Delta_n^{\beta(2 - \alpha)}} \sum_{i = 0}^{n-1} \mathbb{E}[|f(X_{t_i})|R_i( \Delta_{n}^2)] \le \Delta_n^{1 - \beta(2 - \alpha)} \frac{1}{n} \sum_{i = 0}^{n-1}\mathbb{E}[|f(X_{t_i})| R_i(1)]
\label{eq: estim A1}
\end{equation}
and 
\begin{equation}
\frac{\mathbb{E}[|A_1^n|]}{\Delta_n^{\frac{1}{2} - \tilde{\epsilon}}} \le \Delta_n^{\frac{1}{2} + \tilde{\epsilon}} \frac{1}{n} \sum_{i = 0}^{n-1}\mathbb{E}[|f(X_{t_i})| R_i(1)],
\label{eq: estim A1 nuovo}
\end{equation}
that go to zero since the exponent on $\Delta_n$ is always more than $0$, $f$ has both polynomial growth and the moment are bounded. \\
Let us now consider $A_2^n$. By adding and subtracting $b_{t_i}$ in the first integral, as we have already done, we get that 
$$A_2^n = \sum_{i = 0}^{n-1} \zeta_{n,i} + A_{2,2}^n : = 2 \sum_{i = 0}^{n-1} f(X_{t_i})(\int_{t_i}^{t_{i + 1}} b_{t_i} ds)(\int_{t_i}^{t_{i + 1}} a_s dW_s) + 2\sum_{i = 0}^{n-1} f(X_{t_i})(\int_{t_i}^{t_{i + 1}} [b_s - b_{t_i}] ds)(\int_{t_i}^{t_{i + 1}} a_s dW_s).$$
Using Lemma 9 in \cite{Genon Catalot}, we want to show that
\begin{equation}
\sum_{i=0}^{n-1} \zeta_{n,i} = \mathcal{E}_n
\label{eq: A21 nuovo}
\end{equation}
and so that the following convergences hold:
\begin{equation}
\frac{1}{\Delta_n^{\beta(2 - \alpha)}}\sum_{i=0}^{n-1} \mathbb{E}_i[ \zeta_{n,i}] \xrightarrow{\mathbb{P}} 0 \qquad \frac{1}{\Delta_n^{\frac{1}{2} - \tilde{\epsilon}}}\sum_{i=0}^{n-1} \mathbb{E}_i[ \zeta_{n,i}] \xrightarrow{\mathbb{P}} 0;
\label{eq: estim A21}
\end{equation}
\begin{equation}
\frac{1}{\Delta_n^{2\beta(2 - \alpha)}}\sum_{i=0}^{n-1} \mathbb{E}_i[ \zeta^2_{n,i}] \xrightarrow{\mathbb{P}} 0 \qquad \frac{1}{\Delta_n^{2(\frac{1}{2} - \tilde{\epsilon})}}\sum_{i=0}^{n-1} \mathbb{E}_i[ \zeta^2_{n,i}] \xrightarrow{\mathbb{P}} 0.
\label{eq: estim A21 carre}
\end{equation}
We have
$$\sum_{i = 0}^{n-1} \mathbb{E}_i[ \zeta_{n,i}]= \frac{2}{\Delta_n^{\beta(2 - \alpha)}}  \sum_{i = 0}^{n-1} f(X_{t_i}) \Delta_{n} b_{t_i} \mathbb{E}_i[\int_{t_i}^{t_{i + 1}} a_s dW_s] = 0$$
and so the two convergences in \eqref{eq: estim A21} both hold.
Concerning \eqref{eq: estim A21 carre}, using \eqref{eq: bdg} we have
$$\Delta_n^{1 - 2 \beta(2 - \alpha)} \frac{c}{n} \sum_{i = 0}^{n-1} f^2(X_{t_i})b^2_{t_i} \mathbb{E}_i[(\int_{t_i}^{t_{i + 1}} a_s dW_s)^2] \le \Delta_n^{2 - 2 \beta(2 - \alpha)}\frac{c}{n} \sum_{i = 0}^{n-1} f^2(X_{t_i})b^2_{t_i}R_i(1) $$
and 
$$\Delta_n^{1 - 2 (\frac{1}{2} - \tilde{\epsilon})} \frac{c}{n} \sum_{i = 0}^{n-1} f^2(X_{t_i})b^2_{t_i} \mathbb{E}_i[(\int_{t_i}^{t_{i + 1}} a_s dW_s)^2] \le \Delta_n^{1 + 2 \tilde{\epsilon}} \frac{c}{n} \sum_{i = 0}^{n-1} f^2(X_{t_i})b^2_{t_i}R_i(1), $$
that go to zero in norm $1$ and so in probability since $\Delta_n$ is always positive. It follows \eqref{eq: estim A21 carre} and so \eqref{eq: A21 nuovo}.
Concerning $A_{2,2}^n$, using Holder inequality, \eqref{eq: bdg}, the assumption on $b$ gathered in A2 and Jensen inequality it is
$$\mathbb{E}[|A_{2,2}^n|]\le c  \sum_{i = 0}^{n-1} \mathbb{E}[|f(X_{t_i})| \mathbb{E}_i[(\int_{t_i}^{t_{i + 1}} | b_s - b_{t_i}| ds)^q]^\frac{1}{q} R_i(\Delta_{n}^\frac{1}{2}) ]\le $$
$$ \le c \sum_{i = 0}^{n-1} \mathbb{E}[|f(X_{t_i})| ( \Delta_{n}^{q - 1} \int_{t_i}^{t_{i + 1}}\mathbb{E}_i[|b_s - b_{t_i}|^q] ds)^\frac{1}{q} R_i(\Delta_{n}^\frac{1}{2})] \le c \sum_{i = 0}^{n-1} \mathbb{E}[|f(X_{t_i})| ( \Delta_{n}^{q - 1} \int_{t_i}^{t_{i + 1}} \Delta_{n} ds)^\frac{1}{q} R_i(\Delta_{n}^\frac{1}{2})]. $$
So we get
\begin{equation}
 \frac{\mathbb{E}[|A_{2,2}^n|]}{\Delta_n^{\beta(2 - \alpha)}}\le \Delta_n^{\frac{1}{q} + \frac{1}{2} - \beta(2 - \alpha)} \frac{c}{n} \sum_{i = 0}^{n-1}\mathbb{E}[|f(X_{t_i})|R_i(1)] \qquad \mbox{and} 
\label{eq: estim A22}
\end{equation}
\begin{equation}
 \frac{\mathbb{E}[|A_{2,2}^n|]}{\Delta_n^{\frac{1}{2} - \tilde{\epsilon}}}\le \Delta_n^{\frac{1}{q} + \tilde{\epsilon}} \frac{c}{n} \sum_{i = 0}^{n-1}\mathbb{E}[|f(X_{t_i})|R_i(1)].
\label{eq: estim A22 nuovo}
\end{equation}
Since it holds for $q\ge 2$, the best choice is to take $q=2$, in this way we get that \eqref{eq: estim A22} and \eqref{eq: estim A22 nuovo} go to $0$ in norm $1$, using the polynomial growth of $f$, the boundedness of the moments, the definition of $R_i$ and the fact that the exponent on $\Delta_n$ is in both cases more than zero, because of $\beta < \frac{1}{2 - \alpha}$. \\
From \eqref{eq: estim A1}, \eqref{eq: estim A1 nuovo}, \eqref{eq: estim A21}, \eqref{eq: estim A22} and \eqref{eq: estim A22 nuovo} it follows \eqref{eq: estensione Qn}.

\subsection{Proof of Theorem \ref{th: 2 e 3 insieme}}
\begin{proof}
From Theorem \ref{th: estensione Qn} it is enough to prove that 
\begin{equation}
\sum_{i = 0}^{n-1} f(X_{t_i}) (\int_{t_i}^{t_{i+1}} a_s dW_s)^2  - \frac{1}{n} \sum_{i = 0}^{n-1} f (X_{t_i}) a^2_{t_i} = \frac{Z_n}{\sqrt{n}} + \mathcal{E}_n,
\label{eq:primo punto teo 2 3}
\end{equation}
and 
$$\tilde{Q}_n^J = \hat{Q}_n + \frac{1}{\Delta_n^{\beta(2 - \alpha)}} \mathcal{E}_n,$$
where $\mathcal{E}_n$ is always $o_\mathbb{P}(\Delta_n^{\beta(2 - \alpha)})$ and, if $\beta > \frac{1}{4 - \alpha}$, then it is also $o_\mathbb{P}(\Delta_n^{(\frac{1}{2} - \tilde{\epsilon}) \land (1 - \alpha \beta - \tilde{\epsilon})})$. We can rewrite the last equation here above as
\begin{equation}
\tilde{Q}^J_n = \hat{Q}_n + o_\mathbb{P}(1)
\label{eq:secondo punto teo 2 3}
\end{equation}
and, for $\beta > \frac{1}{4 - \alpha}$, 
\begin{equation}
\tilde{Q}^J_n = \hat{Q}_n + \frac{1}{\Delta_n^{\beta(2 - \alpha)}} o_\mathbb{P}(\Delta_n^{(\frac{1}{2} - \tilde{\epsilon}) \land (1 - \alpha \beta - \tilde{\epsilon})}).
\label{eq:nuovo punto teo 2 3}
\end{equation}
Indeed, using them and \eqref{eq: estensione Qn} it follows \eqref{eq:tesi teo 2 e 3}. Hence we are now left to prove \eqref{eq:primo punto teo 2 3} - \eqref{eq:nuovo punto teo 2 3}.\\ \\
\textit{Proof of \eqref{eq:primo punto teo 2 3}}.\\
We can see the left hand side of \eqref{eq:primo punto teo 2 3} as 
\begin{equation}
\sum_{i = 0}^{n-1} f(X_{t_i}) [(\int_{t_i}^{t_{i+1}} a_s dW_s)^2- \int_{t_i}^{t_{i+1}} a^2_s ds] +  \sum_{i = 0}^{n-1} f(X_{t_i})\int_{t_i}^{t_{i+1}} [a^2_s - a^2_{t_i}] ds = : M_n^Q + B_n.
\label{eq: def Bn}
\end{equation}
We want to show that $B_n = \mathcal{E}_n$, it means that it is both $o_\mathbb{P}(\Delta_n^{\beta(2 - \alpha)})$ and $o_\mathbb{P}(\Delta_n^{(\frac{1}{2} - \tilde{\epsilon}) \land (1 - \alpha \beta - \tilde{\epsilon})})$. We write
\begin{equation}
a^2_s - a^2_{t_i} = 2a_{t_i}(a_s - a_{t_i}) + (a_s - a_{t_i})^2,
\label{eq: development a2}
\end{equation}
replacing \eqref{eq: development a2} in the definition of $B_n$ it is $B_n = B_1^n + B_2^n$.
We start by proving that $B_2^n = o_\mathbb{P}(\Delta_n^{\beta(2 - \alpha)})$. Indeed, from the second point of Lemma \ref{lemma: Moment inequalities}, it is
$$\mathbb{E}[|B_2^n|] \le c \sum_{i = 0}^{n - 1}\mathbb{E}[|f(X_{t_i})| \int_{t_i}^{t_{i + 1}} \mathbb{E}_i[|a_s - a_{t_i}|^{2}] ds] \le c \Delta_n^2 \sum_{i = 0}^{n - 1}\mathbb{E}[|f(X_{t_i})|].$$
It follows
\begin{equation}
\frac{\mathbb{E}[|B_2^n|]}{\Delta_n^{\beta(2 - \alpha)}} \le \Delta_n^{1 - \beta(2 - \alpha)}\frac{1}{n} \sum_{i = 0}^{n - 1}\mathbb{E}[ |f|(X_{t_i})] \qquad \mbox{and} \quad \frac{\mathbb{E}[|B_2^n|]}{\Delta_n^{\frac{1}{2} - \tilde{\epsilon}}} \le \Delta_n^{\frac{1}{2} + \tilde{\epsilon}}\frac{1}{n} \sum_{i = 0}^{n - 1}\mathbb{E}[ |f|(X_{t_i})],
\label{eq: stima B2n stabile}
\end{equation}
that go to zero using the polynomial growth of $f$ and the fact that the moments are bounded. We have also observed that the exponent on $\Delta_n$ is always more than $0$. \\
Concerning $B_1^n$, we recall that from \eqref{eq: model vol} it follows
$$a_s - a_{t_i} = \int_{t_i}^s \tilde{b}_u du + \int_{t_i}^s \tilde{a}_u dW_u + \int_{t_i}^s \hat{a}_u d\hat{W}_u + \int_{t_i}^s \int_{\mathbb{R} \backslash \left \{0 \right \}} \tilde{\gamma}_u \, z \, \tilde{\mu}(du, dz) + \int_{t_i}^s \int_{\mathbb{R} \backslash \left \{0 \right \}} \hat{\gamma}_u \, z \, \tilde{\mu}_2(du, dz)$$
and so, replacing it in the definition of $B_1^n$, we get $B_1^n:= I_1^n + I_2^n + I_3^n + I_4^n + I_5^n$. \\
We start considering $I_1^n$ on which we use that $\tilde{b}$ is bounded
$$\mathbb{E}[|I_1^n|] \le 2 \sum_{i = 0}^{n - 1}\mathbb{E}[|f(X_{t_i})||a_{t_i}| \int_{t_i}^{t_{i + 1}}(\int_{t_i}^s c du) ds] \le \Delta_n  \frac{1}{n} \sum_{i = 0}^{n - 1} \mathbb{E}[|f(X_{t_i})||a_{t_i}|].  $$
It follows
\begin{equation}
\frac{\mathbb{E}[|I_1^n|] }{\Delta_n^{\beta(2 - \alpha)}} \le \Delta_n^{1 - \beta(2 - \alpha)} \frac{1}{n} \sum_{i = 0}^{n - 1} \mathbb{E}[|f(X_{t_i})||a_{t_i}|] \qquad \mbox{and }
\label{eq: conv I1n stabile}
\end{equation}
\begin{equation}
\frac{\mathbb{E}[|I_1^n|]}{\Delta_n^{\frac{1}{2} - \tilde{\epsilon}}} \le \Delta_n^{\frac{1}{2} + \tilde{\epsilon}}\frac{1}{n} \sum_{i = 0}^{n - 1} \mathbb{E}[|f(X_{t_i})||a_{t_i}|],
\label{eq: stima I1n nuovo}
\end{equation}
that go to zero because of the polynomial growth of $f$, the boundedness of the moments and the fact that $1 - \beta(2 - \alpha) > 0$. \\
We now act on $I_2^n$ and $I_3^n$ in the same way. Considering $I_2^n$, we define $\zeta_{n,i} : =  2 f(X_{t_i}) a_{t_i} \int_{t_i}^{t_{i + 1}}(\int_{t_i}^s \tilde{a}_u dW_u) ds $. We want to use Lemma 9 in \cite{Genon Catalot} to get that 
\begin{equation}
\frac{I_2^n}{\Delta_n^{\beta(2 - \alpha)}} \xrightarrow{\mathbb{P}} 0 \qquad \mbox{and } \quad \frac{I_2^n}{\Delta_n^{(\frac{1}{2} - \tilde{\epsilon}) \land (1 - \alpha \beta - \tilde{\epsilon} )}} \xrightarrow{\mathbb{P}} 0
\label{eq: conv I2n stabile}
\end{equation}
and so we have to show the following : 
\begin{equation}
\frac{1}{\Delta_n^{\beta (2 - \alpha)}}\sum_{i = 0}^{n-1} \mathbb{E}_i[\zeta_{n,i}]  \xrightarrow{\mathbb{P}} 0, \qquad \frac{1}{\Delta_n^{\frac{1}{2} - \tilde{\epsilon}}}\sum_{i = 0}^{n-1} \mathbb{E}_i[\zeta_{n,i}]  \xrightarrow{\mathbb{P}} 0;
\label{eq: cond 1 genon catalot stabile}
\end{equation}
\begin{equation}
\frac{1}{\Delta_n^{2\beta (2 - \alpha)}}\sum_{i = 0}^{n-1} \mathbb{E}_i[\zeta_{n,i}^2]  \xrightarrow{\mathbb{P}} 0,
\label{eq: cond 2 genon catalot stabile}
\end{equation}
\begin{equation}
\frac{1}{\Delta_n^{2(\frac{1}{2} - \tilde{\epsilon})}}\sum_{i = 0}^{n-1} \mathbb{E}_i[\zeta^2_{n,i}]  \xrightarrow{\mathbb{P}} 0.
\label{eq:cond 2 genon catalot stabile nuovo }
\end{equation}
By the definition of $\zeta_{n,i}$ it is $\mathbb{E}_i[\zeta_{n,i}] = 0$ and so \eqref{eq: cond 1 genon catalot stabile} is clearly true. The left hand side of \eqref{eq: cond 2 genon catalot stabile} is 
\begin{equation}
\Delta_n^{ - 2\beta(2 - \alpha)} 4 \sum_{i = 0}^{n-1} f^2(X_{t_i})a_{t_i}^2 \mathbb{E}_i [(\int_{t_i}^{t_{i + 1}}(\int_{t_i}^s \tilde{a}_u dW_u) ds )^2]. 
\label{eq: I2n intermedio}
\end{equation}
Using Fubini theorem and Ito isometry we have
\begin{equation}
 \mathbb{E}_i [(\int_{t_i}^{t_{i + 1}}(\int_{t_i}^s \tilde{a}_u dW_u) ds )^2] =  \mathbb{E}_i [(\int_{t_i}^{t_{i + 1}}(t_{i + 1} - s) \tilde{a}_s dW_s)^2] = \mathbb{E}_i [\int_{t_i}^{t_{i + 1}}(t_{i + 1} - s^2) \tilde{a}_s^2 ds] \le R_i(\Delta_{n}^3).
\label{eq: utile carre I2n}
\end{equation}
 Because of \eqref{eq: utile carre I2n}, we get that \eqref{eq: I2n intermedio} is upper bounded by
$$ \Delta_n^{2 - 2\beta(2 - \alpha)} \frac{1}{n} \sum_{i = 0}^{n-1}f^2(X_{t_i})a_{t_i}^2 R_i(1),$$
that converges to zero in norm $1$ and so \eqref{eq: cond 2 genon catalot stabile} follows, since $2 - 2\beta(2 - \alpha) > 0$ for $\beta < \frac{1}{2 - \alpha}$, that is always true.
Acting in the same way we get that the left hand side of \eqref{eq:cond 2 genon catalot stabile nuovo } is upper bounded by 
$$ \Delta_n^{1 + 2 \tilde{\epsilon} } \frac{1}{n} \sum_{i = 0}^{n-1}f^2(X_{t_i})a_{t_i}^2 R_i(1),$$
that goes to zero in norm $1$. The same holds clearly for $I_3^n$ instead of $I_2^n$.
In order to show also 
\begin{equation}
\frac{I_4^n}{\Delta_n^{\beta(2 - \alpha)}} \xrightarrow{\mathbb{P}} 0 \qquad \mbox{and } \quad \frac{I_4^n}{\Delta_n^{(\frac{1}{2} - \tilde{\epsilon}) \land (1 - \alpha \beta - \tilde{\epsilon} )}} \xrightarrow{\mathbb{P}} 0,
\label{eq: convergence I3n stabile}
\end{equation}
we define $\tilde{\zeta}_{n,i} :=  2 f(X_{t_i}) a_{t_i} \int_{t_i}^{t_{i + 1}}(\int_{t_i}^s \int_\mathbb{R} \tilde{\gamma}_u z \tilde{\mu}(du, dz)) ds $. We have again $\mathbb{E}_i[\tilde{\zeta}_{n,i}]= 0$ and so \eqref{eq: cond 1 genon catalot stabile} holds with $\tilde{\zeta}_{n,i}$ in place of $\zeta_{n,i}$.
We now act like we did in \eqref{eq: utile carre I2n}, using Fubini theorem and Ito isometry. It follows
$$\mathbb{E}_i [(\int_{t_i}^{t_{i + 1}}(\int_{t_i}^s \int_\mathbb{R}  \tilde{\gamma}_u z \tilde{\mu}(du,dz) ds )^2] =  \mathbb{E}_i [(\int_{t_i}^{t_{i + 1}}\int_\mathbb{R}(t_{i + 1} - s) \tilde{\gamma}_s z \tilde{\mu}(ds,dz))^2]=$$
\begin{equation}
= \mathbb{E}_i[\int_{t_i}^{t_{i + 1}}(t_{i + 1} - s)^2 \tilde{\gamma}_s^2 ds (\int_\mathbb{R} z^2 F(z)dz)] \le R_i(\Delta_{n}^3),
\label{eq: isometria salti}
\end{equation}
having used in the last inequality the definition of $\bar{\mu}(ds, dz)$, the fact that $\int_\mathbb{R} z^2 F(z)dz < \infty$ and the boundedness of $\tilde{\gamma}$. Replacing \eqref{eq: isometria salti} in the left hand side of \eqref{eq: cond 2 genon catalot stabile} and \eqref{eq:cond 2 genon catalot stabile nuovo }, with $\tilde{\zeta}_{n,i}$ in place of $\zeta_{n,i}$, we have
$$\frac{1}{\Delta_n^{2\beta(2 - \alpha)}}\sum_{i=0}^{n-1} \mathbb{E}_i[\tilde{\zeta}^2_{n,i}] \le c \Delta_n^{ - 2\beta(2 - \alpha)} \sum_{i = 0}^{n-1}f^2(X_{t_i}) a^2_{t_i} R_i(\Delta^3_{n}) \le \Delta_n^{2 - 2\beta(2 - \alpha)} \frac{1}{n} \sum_{i = 0}^{n-1} f^2(X_{t_i}) a^2_{t_i} R_i(1)$$
$$\mbox{and } \frac{1}{\Delta_n^{1 - 2 \tilde{\epsilon}}}\sum_{i=0}^{n-1} \mathbb{E}_i[\tilde{\zeta}^2_{n,i}] \le \Delta_n^{1 + 2 \tilde{\epsilon}} \frac{1}{n} \sum_{i = 0}^{n-1}f^2(X_{t_i}) a^2_{t_i} R_i(1).$$
Again, they converge to zero in norm $1$ and thus in probability since $2 - 2\beta(2 - \alpha) > 0$  always holds. Therefore, we get \eqref{eq: convergence I3n stabile}. Clearly, \eqref{eq: convergence I3n stabile} holds also with $I_5^n$ replacing $I_4^n$; the reasoning here above joint with the sixth point of A4 on $F_2$ is proof of that. \\
From \eqref{eq: stima B2n stabile}, \eqref{eq: conv I1n stabile}, \eqref{eq: stima I1n nuovo}, \eqref{eq: conv I2n stabile} and \eqref{eq: convergence I3n stabile} it follows that 
\begin{equation}
B_n = \mathcal{E}_n.
\label{eq: Bn trascurabile in punto 1 teo 2 3}
\end{equation}
Concerning $M_n^Q: = \sum_{i= 0}^{n-1} \hat{\zeta}_{n,i}$, Genon - Catalot and Jacod have proved in \cite{Genon Catalot} that, in the continuous framework, the following conditions are enough to get $\sqrt{n}M_n^Q \rightarrow N(0, 2 \int_0^T f^2(X_s) a^4_s ds)$ stably with respect to $X$:
\begin{itemize}
    \item $\mathbb{E}_i[\hat{\zeta}_{n,i}] = 0$;
    \item $\sum_{i = 0}^{n-1}\mathbb{E}_i[\hat{\zeta}^2_{n,i}] \xrightarrow{\mathbb{P}} 2 \int_0^T  f^2(X_s) a^4_s ds$ ;
    \item $\sum_{i = 0}^{n-1}\mathbb{E}_i[\hat{\zeta}^4_{n,i}] \xrightarrow{\mathbb{P}} 0$;
    \item $\sum_{i = 0}^{n-1} \mathbb{E}_i[\hat{\zeta}_{n,i}(W_{t_{i + 1}} - W_{t_i})] \xrightarrow{\mathbb{P}} 0$;
    \item $\sum_{i = 0}^{n-1} \mathbb{E}_i[\hat{\zeta}_{n,i}(\hat{W}_{t_{i + 1}} - \hat{W}_{t_i})] \xrightarrow{\mathbb{P}} 0$.
\end{itemize}
 Theorem 2.2.15 in \cite{13 in Maillavin} adapts the previous theorem to our framework, in which there is the presence of jumps. \\
We observe that the conditions here above are respected, hence 
\begin{equation}
M_n^Q = \frac{Z_n}{\sqrt{n}}, \mbox{ where } Z_n \xrightarrow{n} N(0,2 \int_0^T  f^2(X_s) a^4_s ds),
\label{eq: conv stabile}
\end{equation}
stably with respect to $X$.
From \eqref{eq: Bn trascurabile in punto 1 teo 2 3} and \eqref{eq: conv stabile}, it follows \eqref{eq:primo punto teo 2 3}. \\ \\
\textit{Proof of \eqref{eq:secondo punto teo 2 3}.} \\
We use Proposition \ref{prop: espansione salti} replacing \eqref{eq: espansione salti} in the definition \eqref{eq: definition tilde Qn} of $\tilde{Q}^J_n$. Recalling that the convergence in norm $1$ implies the convergence in probability it is clear that we have to prove the result on 
$$n^{\beta(2 - \alpha)} \sum_{i = 0}^{n-1} f(X_{t_i})(\Delta \tilde{X}_i^J)^2\varphi_{\Delta_{n}^\beta}(\Delta \tilde{X}_i^J) =$$
\begin{equation}
= n^{\beta(2 - \alpha)} \sum_{i = 0}^{n-1} f(X_{t_i}) \gamma^2(X_{t_i}) \Delta_{n}^\frac{2}{\alpha} (\frac{\Delta \tilde{X}_i^J}{\gamma(X_{t_i})\Delta_{n}^\frac{1}{\alpha}})^2\varphi_{\Delta_{n}^\beta}(\frac{\Delta \tilde{X}_i^J}{\gamma(X_{t_i})\Delta_{n}^\frac{1}{\alpha}}\gamma(X_{t_i})\Delta_{n}^\frac{1}{\alpha}),
\label{eq: ref tesi salti}
\end{equation}
where we have also rescaled the process in order to apply Proposition \ref{prop: estimation stable}. We now define 
\begin{equation}
g_{i,n}(y):= y^2\varphi_{\Delta_{n}^\beta}(y \gamma(X_{t_i})\Delta_{n}^\frac{1}{\alpha}), 
\label{eq: definition g}
\end{equation}
hence we can rewrite \eqref{eq: ref tesi salti} as
$$(\frac{1}{n})^{\frac{2}{\alpha} - \beta(2 - \alpha)} \sum_{i = 0}^{n-1} f(X_{t_i}) \gamma^2(X_{t_i}) [g_{i,n}(\frac{\Delta \tilde{X}_i^J}{\gamma(X_{t_i})\Delta_{n}^\frac{1}{\alpha}}) - \mathbb{E}[g_{i,n}(S_1^\alpha)]] +$$
\begin{equation}
+ (\frac{1}{n})^{\frac{2}{\alpha} - \beta(2 - \alpha)} \sum_{i = 0}^{n-1} f(X_{t_i}) \gamma^2(X_{t_i}) \mathbb{E}[g_{i,n}(S_1^\alpha)] =: \sum_{i = 0}^{n-1}A_{1, i}^n + \hat{Q}_n,
\label{eq: def B1 B2}
\end{equation}
where $S_1^\alpha$ is the $\alpha$-stable process at time $t=1$. We want to show that $\sum_{i = 0}^{n-1}A_{1, i}^n $ converges to zero in probability. With this purpose in mind, we take the conditional expectation of $A_{1, i}^n $ and we apply Proposition \ref{prop: estimation stable} on the interval $[t_i, t_{i+1}]$ instead of on $[0, h]$, observing that property \eqref{eq: conditon on h} holds on $g_{i,n}$ for $p=2$.
By the definition \eqref{eq: definition g} of $g_{i,n}$, we have $\left \| g_{i,n} \right \|_\infty = R_i(\Delta_{n}^{2(\beta - \frac{1}{\alpha})}) $ and  $\left \| g_{i,n} \right \|_{pol} = R_i(1)$. Replacing them in \eqref{eq: tesi prop stable} we have that
$$|\mathbb{E}_i[g_{i,n}(\frac{\Delta \tilde{X}_i^J}{\gamma(X_{t_i})\Delta_{n}^\frac{1}{\alpha}})] - \mathbb{E}[g_{i,n}(S_1^\alpha)]| \le c_{\epsilon, \alpha} \Delta_{n}|\log(\Delta_{n})|R_i(\Delta_{n}^{2(\beta - \frac{1}{\alpha})}) + $$
$$ + c_{\epsilon, \alpha} \Delta_{n}^\frac{1}{\alpha}|\log(\Delta_{n})|R_i(\Delta_{n}^{2(\beta - \frac{1}{\alpha})(1 - \frac{\alpha}{2} - \epsilon)}) + c_{\epsilon, \alpha} \Delta_{n}^\frac{1}{\alpha}|\log(\Delta_{n})|R_i(\Delta_{n}^{2(\beta - \frac{1}{\alpha})(\frac{3}{2} - \frac{\alpha}{2} - \epsilon)})1_{\alpha > 1}.  $$
To get $\sum_{i = 0}^{n-1}A_{1, i}^n : = o_\mathbb{P}(1)$, we want to use Lemma 9 of \cite{Genon Catalot}. We have
$$\sum_{i = 0}^{n-1}|\mathbb{E}_i[A_{1, i}^n ] |\le (\frac{1}{n})^{\frac{2}{\alpha} - \beta(2 - \alpha)} \sum_{i = 0}^{n-1} |f(X_{t_i})| |\gamma^2(X_{t_i})||\log(\Delta_{n})|[\Delta_{n}^{1 + 2 (\beta - \frac{1}{\alpha})} + \Delta_{n}^{\frac{1}{\alpha} + (2 - \alpha - \epsilon)(\beta - \frac{1}{\alpha})}+  $$
\begin{equation}
 + \Delta_{n}^{\frac{1}{\alpha} + (3 - \alpha - \epsilon)(\beta - \frac{1}{\alpha})}1_{\alpha > 1}]R_i(1) \le (\Delta_n^{\alpha \beta} + \Delta_n^{\frac{1}{\alpha} - \epsilon} + \Delta_n^{\beta - \epsilon}1_{\alpha > 1})  \frac{| \log(\Delta_n)|}{n}\sum_{i = 0}^{n-1} |f(X_{t_i})| |\gamma^2(X_{t_i})|R_i(1),
\label{eq: stima B1}
\end{equation}
where we have used property \eqref{propriety power R}. Using the polynomial growth of $f$, the boundedness of the moments and the fifth point of Assumption 4 in order to bound $\gamma$, \eqref{eq: stima B1} converges to $0$ in norm $1$ and so in probability since $\Delta_n^{\alpha \beta} \log(\Delta_n) \rightarrow 0$ for $n\rightarrow \infty$ and we can always find an $\epsilon > 0$ such that $\Delta_n^{\frac{1}{\alpha} - \epsilon}$ does the same. \\
To use Lemma 9 of \cite{Genon Catalot} we have also to show that 
\begin{equation}
(\frac{1}{n})^{\frac{4}{\alpha}  - 2 \beta(2 - \alpha)} \sum_{i = 0}^{n - 1} f^2(X_{t_i}) \gamma^4(X_{t_i}) \mathbb{E}_i[(g_{i,n}(\frac{\Delta \tilde{X}_i^J}{\gamma(X_{t_i})\Delta_{n}^\frac{1}{\alpha}}) - \mathbb{E}[g_{i,n}(S_1^\alpha)])^2] \xrightarrow{P} 0.
\label{eq: conv Ai alla seconda}
\end{equation}
We observe that $\mathbb{E}_i[(g_{i,n}(\frac{\Delta \tilde{X}_i^J}{\gamma(X_{t_i})\Delta_{n}^\frac{1}{\alpha}}) - \mathbb{E}[g_{i,n}(S_1^\alpha)])^2] \le c \mathbb{E}_i[g_{i,n}^2(\frac{\Delta \tilde{X}_i^J}{\gamma(X_{t_i})\Delta_{n}^\frac{1}{\alpha}})] + c \mathbb{E}_i[\mathbb{E}[g_{i,n}(S_1^\alpha)]^2] $. Now, using equation \eqref{eq: estensione tilde salti lemma 10} of Lemma \ref{lemma: estensione 10 capitolo 1}, we observe it is
\begin{equation}
\mathbb{E}_i[g_{i,n}^2(\frac{\Delta \tilde{X}_i^J}{\gamma(X_{t_i})\Delta_{n}^\frac{1}{\alpha}})] = \frac{\Delta_{n}^{- \frac{4}{\alpha}}}{\gamma^4(X_{t_i})} \mathbb{E}_i[(\Delta \tilde{X}_i^J)^4 \varphi^2_{\Delta_n^\beta}(\Delta \tilde{X}_i^J)] = \frac{\Delta_{n}^{- \frac{4}{\alpha}}}{\gamma^4(X_{t_i})}R_i(\Delta_{n}^{1 + \beta(4 - \alpha)}),
\label{eq: estim g2x}
\end{equation}
where $\varphi$ acts as the indicator function. Moreover we observe that
\begin{equation}
\mathbb{E}[g_{i,n}(S_1^\alpha)] = \int_\mathbb{R} z^2 \varphi(\Delta_{n}^{\frac{1}{\alpha} - \beta}\gamma(X_{t_i})z) f_\alpha(z) dz = d(\gamma(X_{t_i})\Delta_{n}^{\frac{1}{\alpha} - \beta}),
\label{eq: val atteso g}
\end{equation}
with $f_\alpha(z)$ the density of the stable process. 
We now introduce the following lemma, that will be shown in the Appendix:
\begin{lemma}
Suppose that Assumptions 1-4 hold. Then, for each $\zeta_n$ such that $\zeta_n \rightarrow 0$ and for each $\hat{\epsilon} > 0$, 
\begin{equation}
d(\zeta_n) = |\zeta_n|^{\alpha - 2} c_\alpha \int_\mathbb{R} |u|^{1 - \alpha} \varphi(u) du + o(|\zeta_n|^{- \hat{\epsilon}} + |\zeta_n|^{2 \alpha - 2 - \hat{\epsilon}}),
\label{eq: dl d}
\end{equation}
where $c_\alpha$ has been defined in \eqref{eq: def calpha}.
\label{lemma: dl d}
\end{lemma}
Since $\frac{1}{\alpha} - \beta > 0$, $\gamma(X_{t_i}) \Delta_{n}^{\frac{1}{\alpha} - \beta}$ goes to zero for $n \rightarrow \infty$ and so we can take $\zeta_n$ as $\gamma(X_{t_i}) \Delta_{n}^{\frac{1}{\alpha} - \beta}$, getting that
\begin{equation}
\mathbb{E}[g_{i,n}(S_1^\alpha)] = d(\gamma(X_{t_i})\Delta_{n}^{\frac{1}{\alpha} - \beta}) = R_i(\Delta_{n}^{(\frac{1}{\alpha} - \beta)(\alpha - 2)}).
\label{eq: replaced dl density}
\end{equation}
Replacing \eqref{eq: estim g2x} and \eqref{eq: replaced dl density} in the left hand side of \eqref{eq: conv Ai alla seconda} we get it is upper bounded by 
$$\sum_{i=0}^{n-1} \mathbb{E}_i[(A_{1,i}^n)^2] = (\frac{1}{n})^{\frac{4}{\alpha}  - 2 \beta(2 - \alpha)} \sum_{i=0}^{n-1} f^2(X_{t_i})\gamma^4(X_{t_i})(R_i(\Delta_{n}^{1 + \beta(4 - \alpha)}) + R_i(\Delta_{n}^{4 \beta - \frac{4}{\alpha} + 2 - 2 \alpha \beta})) \le$$
\begin{equation}
\le \Delta_n^{\alpha \beta \land 1 } \frac{1}{n} \sum_{i=0}^{n-1} f^2(X_{t_i})\gamma^4(X_{t_i})R_i(1),
\label{eq: quadrato A1}
\end{equation}
that converges to zero in norm $1$ and so in probability, as a consequence of the polynomial growth of $f$ and the fact that the exponent on $\Delta_n$ is always positive. From \eqref{eq: stima B1} and \eqref{eq: quadrato A1} it follows 
\begin{equation}
\sum_{i=0}^{n - 1} A_{1,i}^n = o_\mathbb{P}(1).
\label{eq: estim A1n}
\end{equation}
and so \eqref{eq:secondo punto teo 2 3}. \\ \\
\textit{Proof of \eqref{eq:nuovo punto teo 2 3}.} \\
We use Proposition \ref{prop: espansione salti} replacing \eqref{eq: aggiunta prop1 salti} in definition \eqref{eq: definition tilde Qn} of $\tilde{Q}^J_n$. Our goal is to prove that
$$n^{\beta(2 - \alpha)} \sum_{i = 0}^{n-1} f(X_{t_i}) (\Delta \tilde{X}_i^J)^2\varphi_{\Delta_{n}^\beta}(\Delta \tilde{X}_i^J) = \hat{Q}_n + o_\mathbb{P}(\Delta_n^{(\frac{1}{2} - 2 \beta + \alpha \beta - \tilde{\epsilon}) \land (1 - 2 \beta - \tilde{\epsilon})}).$$
On the left hand side of the equation here above we can act like we did in \eqref{eq: ref tesi salti} - \eqref{eq: def B1 B2}. 
To get \eqref{eq:nuovo punto teo 2 3} we are therefore left to show that , if $\beta > \frac{1}{4 - \alpha}$, then $\sum_{i = 0}^{n - 1} A_{1,i}^n$ is also $o_\mathbb{P}(\Delta_n^{(\frac{1}{2} - 2 \beta + \alpha \beta - \tilde{\epsilon}) \land (1 - 2 \beta - \tilde{\epsilon})})$. To prove it, we want to use Lemma 9 of \cite{Genon Catalot}, hence we want to prove the following:
\begin{equation}
\frac{1}{\Delta_n^{\frac{1}{2} - 2 \beta + \alpha \beta - \tilde{\epsilon}}} \sum_{i = 0}^{n - 1} \mathbb{E}_i[A_{1,i}^n] \xrightarrow{\mathbb{P}} 0 \quad \mbox{and }
\label{eq: conv A1n nuovo}
\end{equation}
\begin{equation}
\frac{1}{\Delta_n^{2(\frac{1}{2} - 2 \beta + \alpha \beta - \tilde{\epsilon})}} \sum_{i = 0}^{n - 1} \mathbb{E}_i[(A_{1,i}^n)^2] \xrightarrow{\mathbb{P}} 0.
\label{eq: conv A1n carre nuovo}
\end{equation}
Using \eqref{eq: stima B1} we have that, if $\alpha > 1$, then the left hand side of \eqref{eq: conv A1n nuovo} is in module upper bounded by
$$\frac{\Delta_n^{\beta - \epsilon} |\log(\Delta_n)|}{\Delta_n^{\frac{1}{2} - 2 \beta + \alpha \beta - \tilde{\epsilon}}} \frac{1}{n} \sum_{i = 0}^{n - 1}|f(X_{t_i})| |\gamma^2(X_{t_i})|R_i(1) = \Delta_n^{3 \beta - \alpha \beta - \frac{1}{2} + \tilde{\epsilon} - \epsilon }|\log(\Delta_n)| \frac{1}{n} \sum_{i = 0}^{n - 1}|f(X_{t_i})| |\gamma^2(X_{t_i})|R_i(1), $$
that goes to zero since we have chosen $\beta > \frac{1}{4 - \alpha} > \frac{1}{2(3 - \alpha)}$. Otherwise, if $\alpha \le 1$, then \eqref{eq: stima B1} gives us that the left hand side of \eqref{eq: conv A1n nuovo} is in module upper bounded by
$$\frac{\Delta_n^{ \alpha \beta} |\log(\Delta_n)|}{\Delta_n^{\frac{1}{2} - 2 \beta + \alpha \beta - \tilde{\epsilon}}}  \frac{1}{n} \sum_{i = 0}^{n - 1}|f(X_{t_i})| |\gamma^2(X_{t_i})|R_i(1) = \Delta_n^{2 \beta - \frac{1}{2} + \tilde{\epsilon} }|\log(\Delta_n)| \frac{1}{n} \sum_{i = 0}^{n - 1}|f(X_{t_i})| |\gamma^2(X_{t_i})|R_i(1),$$
that goes to zero because $\beta > \frac{1}{4 - \alpha} > \frac{1}{4}$. \\
 Using also \eqref{eq: quadrato A1}, the left hand side of \eqref{eq: conv A1n carre nuovo} turns out to be upper bounded by \\
 $\Delta_n^{-1 + 4 \beta - 2 \alpha \beta + 2 \tilde{\epsilon}} \Delta_n^{\alpha \beta \land 1} \frac{1}{n} \sum_{i=0}^{n-1} f^2(X_{t_i})\gamma^4(X_{t_i})R_i(1),$
 that goes to zero in norm 1 and so in probability since we have chosen $\beta > \frac{1}{4 - \alpha}$.
It follows \eqref{eq: conv A1n carre nuovo} and so \eqref{eq:tesi teo 2 e 3}. \\
\end{proof} 

\subsection{Proof of Proposition \ref{prop: conv hat Qn}} 
\begin{proof}
To prove the proposition we replace \eqref{eq: dl d} in the definition of $\hat{Q}_n$. It follows that our goal is to show that 
$$I_1^n + I_2^n :=( \frac{1}{n})^{\frac{2}{\alpha} - \beta(2 - \alpha)}\sum_{i = 0}^{n - 1} f(X_{t_i}) \gamma^2(X_{t_i})(o(|\Delta_{n}^{\frac{1}{\alpha} - \beta} \gamma(X_{t_i})|^{- \hat{\epsilon}} + |\Delta_{n}^{\frac{1}{\alpha} - \beta} \gamma(X_{t_i})|^{2 \alpha - 2- \hat{\epsilon}} )) = \tilde{\mathcal{E}}_n,$$
where $\tilde{\mathcal{E}}_n$ is always $o_\mathbb{P}(1)$ and, if $\alpha < \frac{4}{3}$, it is also $\frac{1}{\Delta_n^{\beta(2 - \alpha)}} o_\mathbb{P}(\Delta_n^{(\frac{1}{2} - \tilde{\epsilon}) \land (1 - \alpha \beta - \tilde{\epsilon} )})$. \\
We have that $I_1^n = o_\mathbb{P}(1)$ since it is upper bounded by 
$$\Delta_n^{\frac{2}{\alpha} - 1 - 2 \beta + \alpha \beta - \hat{\epsilon}(\frac{1}{\alpha} - \beta) } \frac{1}{n}\sum_{i = 0}^{n - 1} R_i(1)\,o(1),$$
that goes to zero in norm $1$ and so in probability since we can always find an $\hat{\epsilon} > 0$ such that the exponent on $\Delta_n$ is positive. \\
Also $I_2^n$ is $o_\mathbb{P}(1)$. Indeed it is upper bounded by
\begin{equation}
\Delta_n^{\frac{2}{\alpha} - 1 - 2 \beta + \alpha \beta -2(\frac{1}{\alpha} - \beta) +2 (1 - \alpha \beta) - \hat{\epsilon}(\frac{1}{\alpha} - \beta) } \frac{1}{n}\sum_{i = 0}^{n - 1} R_i(1)\,o(1).
\label{eq: remplace density per hat Q}
\end{equation}
We observe that the exponent on $\Delta_n$ is $1 - \alpha \beta  - \hat{\epsilon}(\frac{1}{\alpha} - \beta) $ and we can always find $\hat{\epsilon}$ such that it is more than zero, hence \eqref{eq: remplace density per hat Q} converges in norm $1$ and so in probability. \\
In order to show that $I_1^n = \frac{1}{\Delta_n^{\beta(2 - \alpha)}} o_\mathbb{P}(\Delta_n^{\frac{1}{2} - \tilde{\epsilon}}) = o_\mathbb{P}(\Delta_n^{\frac{1}{2} - \tilde{\epsilon} - \beta(2 - \alpha)})$ we observe that
$$\frac{I_1^n}{\Delta_n^{\frac{1}{2} - \tilde{\epsilon} - \beta(2 - \alpha)}} \le \Delta_n^{\frac{2}{\alpha } - 1  - \frac{1}{2} + \tilde{\epsilon} -\hat{\epsilon}(\frac{1}{\alpha} - \beta)}\frac{1}{n}\sum_{i = 0}^{n - 1} R_i(1)\,o(1). $$
If $\alpha < \frac{4}{3}$ we can always find $\tilde{\epsilon}$ and $\hat{\epsilon}$ such that the exponent on $\Delta_n$ is more than zero, getting the convergence wanted. It follows $I_1^n = \frac{1}{\Delta_n^{\beta(2 - \alpha)}} o_\mathbb{P}(\Delta_n^{(\frac{1}{2} - \tilde{\epsilon}) \land (1 - \alpha \beta - \tilde{\epsilon} )})$. \\
To conclude, $I_2^n = \frac{1}{\Delta_n^{\beta(2 - \alpha)}} o_\mathbb{P}(\Delta_n^{1 - \alpha \beta - \tilde{\epsilon}}) = o_\mathbb{P}(\Delta_n^{1 - 2 \beta - \tilde{\epsilon} })$. Indeed, 
\begin{equation}
\frac{I_2^n}{\Delta_n^{1 - 2 \beta - \tilde{\epsilon} }} \le \Delta_n^{\frac{2}{\alpha} - 1 - 1 + \alpha \beta + \tilde{\epsilon} - 2(\frac{1}{\alpha} - \beta) + 2 (1 - \alpha \beta) - \hat{\epsilon}(\frac{1}{\alpha} - \beta)} \frac{1}{n}\sum_{i = 0}^{n - 1} R_i(1)\,o(1).
\label{eq: estimation con dl f}
\end{equation}
The exponent on $\Delta_n$ is $2 \beta - \alpha \beta + \tilde{\epsilon} - \hat{\epsilon}(\frac{1}{\alpha} - \beta) $ and so we can always find $\tilde{\epsilon}$ and $\hat{\epsilon}$ such that it is positive. It follows the convergence in norm 1 and so in probability of \eqref{eq: estimation con dl f}. The proposition is therefore proved.
\end{proof} 
 
\subsection{Proof of Corollary \ref{cor: cond rimozione rumore}}
\begin{proof}
We observe that \eqref{eq: eq per le appli} is a consequence of \eqref{eq: limite hat Qn con densita} in the case where $\hat{Q}_n = 0$. Moreover, $\beta < \frac{1}{2 \alpha}$ implies that $\Delta_n^{1 - \alpha \beta - \tilde{\epsilon}}$ is negligible compared to $\Delta_n^{\frac{1}{2} - \tilde{\epsilon}}$. It follows \eqref{eq: eq per le appli}.
\end{proof}

\subsection{Proof of Theorem \ref{th: reformulation th T fixed}.}
\begin{proof}
The convergence \eqref{eq: convergenza hat Q tempo corto} clearly follows from \eqref{eq: limite hat Qn con densita}. \\
Concerning the proof of \eqref{eq: tesi finale tempo corto}, we can see its left hand side as $$Q_n - \frac{1}{n} \sum_{i= 0}^{n - 1} f(X_{t_i}) a^2_{t_i} + \frac{1}{n} \sum_{i= 0}^{n - 1} f(X_{t_i}) a^2_{t_i} - IV_1 $$
and so, using \eqref{eq:tesi teo 2 e 3} and the definition of $IV_1$, it turns out that our goal is to show that 
\begin{equation}
\frac{1}{n} \sum_{i= 0}^{n - 1} f(X_{t_i}) a^2_{t_i} - \int_0^1 f(X_s) \, a^2_s ds = o_\mathbb{P}(\Delta_n^{\beta(2 - \alpha)}).
\label{eq: tesi th tempo corto}
\end{equation}
The left hand side of \eqref{eq: tesi th tempo corto} is
$$\sum_{i= 0}^{n - 1} f(X_{t_i}) \int_{t_i}^{t_{i + 1}}(a^2_{t_i} - a^2_s ) ds + \sum_{i= 0}^{n - 1}  \int_{t_i}^{t_{i + 1}} a^2_s (f (X_{t_i}) - f (X_s)) ds = : B_n + R_n.  $$
$B_n$ in the equation here above is exactly the same term we have already dealt with in the proof of Theorem \ref{th: 2 e 3 insieme} (see \eqref{eq: def Bn}). As showed in \eqref{eq: Bn trascurabile in punto 1 teo 2 3} it is $\mathcal{E}_n$ and so, in particular, it is also $o_\mathbb{P}(\Delta_n^{\beta (2 - \alpha)})$. \\
On $R_n$ we act like we did on $B_n$,considering this time the development up to second order of the function $f$, getting
\begin{equation}
f(X_s) = f(X_{t_i}) + f'(X_{t_i}) (X_s - X_{t_i}) + \frac{f''(\tilde{X}_{t_i})}{2} (X_s - X_{t_i})^2,
\label{eq: development f}
\end{equation}
where $\tilde{X}_{t_i} \in [X_{t_i}, X_s]$. Replacing \eqref{eq: development f} in $R_n$ we get two terms that we denote $R_n^1$ and $R_n^2$. On them we can act like we did on \eqref{eq: development a2}. The estimations gathered in Lemma \ref{lemma: Moment inequalities} about the increments of $X$ and of $a$ have the same size (see points 2 and 4) and provide on $B_2^n$ and $R_n^2$ the same upper bound:
$$\mathbb{E}[|R_n^2|] \le c \sum_{i = 0}^{n - 1}\mathbb{E}[|f''(X_{t_i})| \int_{t_i}^{t_{i + 1}} \mathbb{E}_i[|a_s||X_s - X_{t_i}|^{2}] ds] \le c \Delta_n^2 \sum_{i = 0}^{n - 1}\mathbb{E}[|f''(X_{t_i})|R_i(1)],$$
where we have used Cauchy Schwartz inequality and the fourth point of Lemma \ref{lemma: Moment inequalities}.
It yields $R_n^2 = o_\mathbb{P}(\Delta_n^{\beta (2 - \alpha)})$, which is the same result found in the first inequality of \eqref{eq: stima B2n stabile} for the increments of $a$.  \\
To deal with $R_n^1$ we replace the dynamic of $X$ (as done with the dynamic of $a$ for $B_1^n$). Even if the volatility coefficient in the dynamic of $X$ is no longer bounded, the condition $\sup_{s \in [t_i, t_{i + 1}]} \mathbb{E}_i[|a_s|] < \infty$ (which is true according with Lemma \ref{lemma: Moment inequalities}) is enough to say that \eqref{eq: utile carre I2n} keep holding. \\
Following the method provided in the proof of Theorem \ref{th: 2 e 3 insieme} to show that $B_1^n = \mathcal{E}_n$ we obtain $R_n^1= \mathcal{E}_n$ and therefore $R_n^1 =o_\mathbb{P}(\Delta_n^{\beta (2 - \alpha)})$.
It yields \eqref{eq: tesi th tempo corto} and so the theorem is proved.
\end{proof}

\section{Proof of Proposition \ref{prop: estimation stable}.}\label{S:Proof_propositions}
This section is dedicate to the proof of Proposition \ref{prop: estimation stable}.
To do it, it is convenient to introduce an adequate truncation function and to consider a rescaled process, as explained in the next subsections. Moreover, the proof of Proposition \ref{prop: estimation stable} requires some Malliavin calculus; we recall in what follows all the technical tools to make easier the understanding of the paper.

\subsection{Localization and rescaling}\label{subsection: construction per maillavin}
We introduce a truncation function in order to suppress the big jumps of $(L_t)$. Let $\tau : \mathbb{R}  \rightarrow [0, 1]$ be a symmetric function, continuous with continuous derivative, such that $\tau = 1$ on $\left \{ |z| \le \frac{1}{4} \eta \right \}$ and $\tau = 0$ on $\left \{ |z| \ge \frac{1}{2} \eta \right \}$, with $\eta$ defined in the fourth point of Assumption 4. \\
On the same probability space $(\Omega, \mathcal{F}, (\mathcal{F}_t), \mathbb{P})$ we consider the L\'evy process $(L_t)$ defined below \eqref{eq: model} which measure is $F(dz) = \frac{\bar{g}(z)}{|z|^{1 + \alpha}} 1_{\mathbb{R} \setminus \left \{ 0 \right \}} (z) dz$, according with the third point of A4, and the truncated  L\'evy process $(L^\tau_t)$ with measure $F^\tau(dz)$ given by $F^\tau(dz) = \frac{\bar{g}(z) \tau(z)}{|z|^{1 + \alpha}} 1_{\mathbb{R} \setminus \left \{ 0 \right \}} (z) dz$. This can be done by setting $L_t := \int_0^t \int_\mathbb{R} z \tilde{\mu}(ds, dz)$, as we have already done, and $L_t^\tau : = \int_0^t \int_\mathbb{R} z \tilde{\mu}^\tau(ds, dz)$, where $\tilde{\mu}$ and $\tilde{\mu}^\tau$ are the compensated Poisson random measures associated respectively to
$$\mu(A): = \int_{[0,1]} \int_\mathbb{R} \int_{[0,T]} 1_A(t,z) \mu^{\bar{g}}(dt, dz,du), \qquad A \subset [0, T] \times \mathbb{R},$$
$$\mu^\tau(A): = \int_{[0,1]} \int_\mathbb{R} \int_{[0,T]} 1_A(t,z) 1_{u \le \tau(z)} \mu^{\bar{g}}(dt, dz,du), \qquad A \subset [0, T] \times \mathbb{R},$$
for $\mu^{\bar{g}}$ a Poisson random measure on $[0,T] \times \mathbb{R} \times [0,1]$ with compensator $\bar{\mu}^{\bar{g}}(dt, dz, du) = dt \frac{\bar{g}(z)}{|z|^{1 + \alpha}} 1_{\mathbb{R} \setminus \left \{ 0 \right \}} (z) dz du$. \\
By construction, the restrictions of the measures $\mu$ and $\mu^\tau$ to $[0, h] \times \mathbb{R}$ coincide on the set \\
$\left \{ (u, z) \mbox{ such that } u \le \tau(z) \right \}$, and thus coincide on the event 
$$\Omega_h := \left \{ \omega \in \Omega; \mu^{\bar{g}}([0, h] \times \left \{ z \in \mathbb{R}: |z| \ge \frac{\eta}{4}  \right \} \times [0,1]) = 0 \right \}.$$
Since $\mu^{\bar{g}}([0, h] \times \left \{ z \in \mathbb{R}: |z| \ge \frac{\eta}{4}  \right \} \times [0,1])$ has a Poisson distribution with parameter 
$$\lambda_h : = \int_0^{h} \int_{|z| \ge \frac{\eta}{4}} \int_0^1 \frac{\bar{g}(z)}{|z|^{1 + \alpha}} du \, dz\, dt \le c h;$$
we deduce that 
\begin{equation}
\mathbb{P}(\Omega_h^c) \le c \,h.
\label{eq: prob omega n}
\end{equation}
Then we have
\begin{equation}
\mathbb{P}((L_t)_{t \le h} \neq (L_t^\tau)_{t \le h}) \le \mathbb{P}(\Omega_h^c) \le c \, h.
\label{eq: proba diff processi}
\end{equation}
To prove Proposition \ref{prop: estimation stable} we have to rescale the process $(L_t)_{t \in [0,1]}$, we therefore introduce an auxiliary L\'evy process $(L_t^h)_{t \in [0,1]}$ defined possibly on another filtered space $(\tilde{\Omega}, \tilde{\mathcal{F}}, (\tilde{\mathcal{F}_t}), \tilde{\mathbb{P}})$ and admitting the decomposition $L_t^h := \int_0^t \int_\mathbb{R} z \tilde{\mu}^h(dt, dz)$, with $t \in  [0,1]$; where $\tilde{\mu}^h$ is a compensated Poisson random measure $\tilde{\mu}^h = \mu^h - \bar{\mu}^h$, with compensator 
\begin{equation}
\bar{\mu}^h(dt, dz) = dt \frac{\bar{g}(z h^\frac{1}{\alpha})}{|z|^{1 + \alpha}} \tau(z h^{\frac{1}{\alpha}})1_{\mathbb{R} \setminus \left \{ 0 \right \}} (z) dz.
\label{eq: bar mu n}
\end{equation}
By construction, the process $(L_t^h)_{t \in [0,1]}$ is equal in law to the rescaled truncated process $(h^{-\frac{1}{\alpha}} L^\tau_{h t})_{t \in [0,1]}$ that coincides with $(h^{ -\frac{1}{\alpha}} L_{h t})_{t \in [0,1]}$ on $\Omega_n$.

\subsection{Malliavin calculus}\label{section: Maillavin theory}
In this section, we recall some results on Malliavin calculus for jump processes. We refer to \cite{Spieg Maillavin} for a complete presentation and to \cite{Maillavin} for the adaptation to our framework. We will work on the Poisson space associated to the measure $\mu^h$ defining the process $(L_t^h)_{t \in [0,1]}$ of the previous section, assuming that $h$ is fixed. By construction, the support of $\mu^h$ is contained in $[0,1] \times E_h$, where $E_h := \left \{ z \in  \mathbb{R} : |z| < \frac{\eta}{2} \frac{1}{h^\frac{1}{\alpha}}  \right \}$, with $\eta$ defined in the fourth point of A4. We recall that the measure $\mu^h$ has compensator
\begin{equation}
\bar{\mu}^h(dt, dz) = dt \frac{\bar{g}(z h^\frac{1}{\alpha})}{|z|^{1 + \alpha}} \tau(z h^{\frac{1}{\alpha}})1_{\mathbb{R} \setminus \left \{ 0 \right \}} (z) dz := dt F_h(z) dz. 
\label{eq: definition Fn}
\end{equation}
In this section we assume that the truncation function $\tau$ satisfies the additional assumption
$$\int_\mathbb{R}|\frac{\tau '(z)}{\tau(z)}|^p \tau(z) dz < \infty, \qquad \forall p \ge 1. $$
We now define the Malliavin operators $L$ and $\Gamma$ (omitting their dependence in $h$) and their basic properties (see \cite{Spieg Maillavin} Chapter IV, sections 8-9-10). For a test function $f: [0,1] \times \mathbb{R} \rightarrow \mathbb{R}$ measurable, $\mathcal{C}^2$ with respect the second variable, with bounded derivative and such that $f \in \cap_{p \ge 1} L^p(\bar{\mu}^h(dt,dz))$, we set $\mu^h(f) = \int_0^1 \int_\mathbb{R} f(t,z) \mu^h(dt,dz).$ As auxiliary function, we consider $\rho : \mathbb{R} \rightarrow [0, \infty)$ such that $\rho$ is symmetric, two times differentiable and such that $\rho(z) = z^4$ if $z \in [0, \frac{1}{2}]$ and $\rho(z) = z^2$ if $z \ge 1$. Thanks to the truncation $\tau$, we do not need that $\rho$ vanishes at infinity. Assuming the fourth point of Assumption 4, we check that $\rho$, $\rho'$ and $\rho \frac{F_h'}{F_h}$ belong to $\cap_{p \ge 1} L^p(F_h(z)dz)$. With these notations, we define the Malliavin operator $L$ on the functional $\mu^h(f)$ as follows:
$$L(\mu^h(f)) := \frac{1}{2} \mu^h(\rho'f' + \rho \frac{F_h'}{F_h} f' + \rho f''),$$
where $f'$ and $f''$ are derivative with respect to the second variable. This definition permits to construct a linear operator on the space $D \subset \cap_{p \ge 1} L^p(F_h(z)dz)$ which is self-adjoint: $\forall \Phi, \Psi \in D$, $\mathbb{E}\Phi L \Psi = \mathbb{E}L \Phi \Psi$ (see Section 8 in \cite{Spieg Maillavin} for the details on the construction of $D$). \\
We associate to $L$ the symmetric bilinear operator $\Gamma$:
$$\Gamma(\Phi, \Psi)= L(\Phi, \Psi) - \Phi L(\Psi) - \Psi L(\Phi).$$
If $f$ and $g$ are two test functions, we have
\begin{equation}
\Gamma(\mu^h(f), \mu^h(g)) = \mu^h(\rho f'g').
\label{eq: definition Gamma}
\end{equation}
The operators $L$ and $\Gamma$ satisfy the chain rule property:
$$LF(\Phi) = F'(\Phi) L \Phi + \frac{1}{2} F''(\Phi) \Gamma (\Phi, \Phi), \qquad \Gamma(F(\Phi), \Psi) = F'(\Phi)\Gamma (\Phi, \Psi).$$
These operators permit to establish the following integration by parts formula (see \cite{Spieg Maillavin} Theorem 8-10 p.103).
\begin{theorem}
Let $\Phi$ and $\Psi$ be random variable in $D$ and $f$ be a bounded function with bounded derivatives up to order two. If $\Gamma(\Phi, \Phi)$ is invertible and $\Gamma^{-1}(\Phi, \Phi) \in \cap_{p \ge 1} L^p $, then we have
\begin{equation}
\mathbb{E}f'(\Phi) \Psi = \mathbb{E} f(\Phi) \mathcal{H}_\Phi(\Psi),
\label{eq: int per parti}
\end{equation}
with 
\begin{equation}
\mathcal{H}_\Phi(\Psi) = -2 \Psi \Gamma^{-1}(\Phi, \Phi) L\Phi - \Gamma(\Phi, \Psi \Gamma^{-1}(\Phi,\Phi)).
\label{eq: def peso Maillavin}
\end{equation}
\label{th: regola catena maillavin}
\end{theorem}
The random variable $L_1^h$ belongs to the domain of the operators $L$ and $\Gamma$. Computing $L(L_1^h)$, $\Gamma(L_1^h, L_1^h)$ and $\mathcal{H}_{L_1^h}(1)$ it is possible to deduce the following useful inequalities, proved in Lemma 4.3 in \cite{Maillavin}.
\begin{lemma}
We have 
$$\sup_n \mathbb{E}|\mathcal{H}_{L_1^h}(1)|^p \le C_p \qquad \forall p \ge 1,$$
$$\sup_n \mathbb{E}|\int_0^1 \int_{|z| > 1} |z| \mu^h(ds, dz) \mathcal{H}_{L_1^h}(1)|^p \le C_p \qquad \forall p \ge 1.$$
\label{lemma: lemma 4.3 Maillavin}
\end{lemma}
With this background we can proceed to the proof of Proposition \ref{prop: estimation stable}.

\subsection{Proof of Proposition \ref{prop: estimation stable}}
\begin{proof}
The first step is to construct on the same probability space two random variables whose laws are close to the laws of $h^{- \frac{1}{\alpha}}L_{h}$ and $S_1^\alpha$. We recall briefly the notation of Section \ref{subsection: construction per maillavin}: $\mu^h$ is a Poisson random measure with compensator $\bar{\mu}^h(dt, dz)$ defined in \eqref{eq: bar mu n} and the process $L_t^h$ is defined by 
\begin{equation}
L_t^h = \int_0^t \int_\mathbb{R} z \tilde{\mu}^h(ds,dz) = \int_0^t \int_{|z| \le h^{- \frac{1}{\alpha}}\frac{\eta}{2}} z \tilde{\mu}^h(ds, dz)
\label{eq: definition L n}
\end{equation}
with $\tilde{\mu}^h = \mu^h - \bar{\mu}^h$.
Using triangle inequality we have 
\begin{equation}
|\mathbb{E}[g(h^{- \frac{1}{\alpha}}L_{h})] - \mathbb{E}[g(S_1^\alpha)]| \le |\mathbb{E}[g(h^{- \frac{1}{\alpha}}L_{h})] - \mathbb{E}[g(L_1^h)]|+|\mathbb{E}[g(L_1^h) - g(S_1^\alpha)]|.
\label{eq: inizio prop maillavin}
\end{equation}
By the definition of $L_1^h$ it is 
\begin{equation}
|\mathbb{E}[g(h^{- \frac{1}{\alpha}}L_{h})] - \mathbb{E}[g(L_1^h)] | = |\mathbb{E}[g(h^{- \frac{1}{\alpha}}L_{h}) - g(h^{- \frac{1}{\alpha}}L^\tau_{h})] | \le 2 \left \| g \right \|_\infty \mathbb{P}(\Omega_n^c)\le c\left \| g \right \|_\infty h,
\label{eq: stima primo termine Maillavin}
\end{equation}
where in the last inequality we have used \eqref{eq: proba diff processi}. In order to get an estimation to the second term of \eqref{eq: inizio prop maillavin} we now construct a variable approximating the law of $S_1^\alpha$ and based on the Poisson measure $\mu^h$ :
\begin{equation}
L_t^{\alpha, h} := \int_0^t \int_{|z| \le h^{- \frac{1}{\alpha}}\frac{\eta}{2}} g_h(z)\tilde{\mu}^h(ds, dz),
\label{eq: definition L alpha n}
\end{equation}
where $g_h$ is an odd function built in the proof of Theorem 4.1 in \cite{Maillavin} for which the following lemma holds:
\begin{lemma}
\begin{enumerate}
    \item For each test function $f$, defined as in Section \ref{section: Maillavin theory}, we have
    \begin{equation}
    \int_0^1 \int_{|z| \le \frac{\eta}{2} h^{- \frac{1}{\alpha}}} f(t,g_h(z)) \bar{\mu}^h(dt, dz) = \int_0^1 \int_{|\omega| \le \frac{\eta}{2} h^{- \frac{1}{\alpha}}} f(t,\omega) \bar{\mu}^{\alpha, h}(dt, d\omega),
    \label{eq: trasformazione legge da salti a uniforme}
    \end{equation}
    where $\bar{\mu}^h(dt, dz)$ is the compensator defined in \eqref{eq: bar mu n} and 
    $$\bar{\mu}^{\alpha, h}(dt, d\omega) = dt \frac{\tau(\omega h^\frac{1}{\alpha})}{|\omega|^{1 + \alpha}} d\omega$$
    is the compensator of a measure associated to an $\alpha$- stable process whose jumps are truncated with the function $\tau$.
    \item There exists $\epsilon_0 > 0$ such that, for $|z| \le \epsilon_0 h^{- \frac{1}{\alpha}}$,
    $$|g_h(z) - z| \le c z^2 h^{\frac{1}{\alpha}} + c |z|^{1 + \alpha }h \qquad \mbox{if } \alpha \neq 1,$$
    $$|g_h(z) - z| \le c z^2 h |\log(|z|h)| \qquad \mbox{if } \alpha =1.$$
    \item The function $g_h$ is $\mathcal{C}^1$ on $(- \epsilon_0 h^{- \frac{1}{\alpha}}, \epsilon_0 h^{- \frac{1}{\alpha}})$ and for $|z| < \epsilon_0 h^{- \frac{1}{\alpha}}$,
    $$|g'_h(z) - 1| \le c |z| h^{\frac{1}{\alpha}} + c |z|^{ \alpha }h \qquad \mbox{if } \alpha \neq 1,$$
    $$|g'_h(z) - 1| \le c |z| h |\log(|z|h)| \qquad \mbox{if } \alpha =1.$$
\end{enumerate}
\label{lemma: 4.5 in Maillavin}
\end{lemma}
The second and the third point of the lemma here above are proved in Lemma 4.5 of \cite{Maillavin}, while the first point is proved in Theorem 4.1 \cite{Maillavin} and it shows us, using the exponential formula for Poisson measure, that $g_h$ is the function that turns our measure $\mu^h$ into the measure associated to an $\alpha$-stable process truncated with the function $\tau$. Thus $(L_t^{\alpha, h})_{t \in [0,1]}$ is a L\'evy process with jump intensity $\omega \mapsto \frac{\tau(\omega h^\frac{1}{\alpha})}{|\omega|^{1 + \alpha}}$ and we recognize the law of an $\alpha$-stable truncated process.  We deduce, similarly to \eqref{eq: stima primo termine Maillavin},
\begin{equation}
|\mathbb{E}[g(L_1^{\alpha, h})] - \mathbb{E}[g(S_1^{\alpha})]| \le  c\left \| g \right \|_\infty h.
\label{eq: differenza S1 e L1 alpha n}
\end{equation}
Proposition \ref{prop: estimation stable} is a consequence of \eqref{eq: inizio prop maillavin}, \eqref{eq: stima primo termine Maillavin}, \eqref{eq: differenza S1 e L1 alpha n} and the following lemma:
\begin{lemma}
Suppose that Assumptions 1 to 4 hold. Let $g$ be as in Proposition \ref{prop: estimation stable}. Then, for any $\epsilon > 0$ and for $p \ge \alpha$,
$$|\mathbb{E}[g(L_1^{h}) - g(L_1^{\alpha, h})]| \le   C_\epsilon h |\log(h)| \left \| g \right \|_\infty + C_\epsilon h^\frac{1}{\alpha}\left \| g \right \|_\infty^{1 - \frac{\alpha}{p} + \epsilon} \left \| g \right \|_{pol}^{\frac{\alpha}{p} - \epsilon} |\log(h)| +$$
$$+ C_\epsilon h^\frac{1}{\alpha}\left \| g \right \|_\infty^{1+ \frac{1}{p} - \frac{\alpha}{p} + \epsilon} \left \| g \right \|_{pol}^{- \frac{1}{p} + \frac{\alpha}{p} - \epsilon} |\log(h)|1_{\alpha > 1} .$$
\label{lemma: main prop stable}
\end{lemma}
\begin{proof}
The proof is based of the comparison of the representation of \eqref{eq: definition L n} and \eqref{eq: definition L alpha n}. Since in Lemma \ref{lemma: 4.5 in Maillavin} the difference $g_h(z) - z$ is controlled for $|z| \le \epsilon_0 h^{- \frac{1}{\alpha}}$, we need to introduce a localization procedure consisting in regularizing $1_{\left \{ \mu^h([0,1] \times \left \{ z \in \mathbb{R}: |z| > \epsilon_0 h^{- \frac{1}{\alpha}} \right \}) = 0 \right \}}$. Let $\mathcal{I}$ be a smooth function defined on $\mathbb{R}$ and with values in $[0,1]$, such that $\mathcal{I}(x) = 1$ for $x \le \frac{1}{2}$ and $\mathcal{I}(x) = 0$ for $x \ge 1$. Moreover, we denote by $\zeta$ a smooth function on $\mathbb{R}$, with values in $[0,1]$ such that $\zeta(z)=0$ for $|z| \le \frac{1}{2}$ and $\zeta(z) = 1$ for $|z| \ge 1$ and we set
$$V^h : = \int_0^1 \int_\mathbb{R} \zeta(\frac{z h^\frac{1}{\alpha}}{\epsilon_0}) \mu^h(ds, dz)= \int_0^1 \int_{\left \{ \frac{1}{2} \epsilon_0 h^{- \frac{1}{\alpha}} \le |z| \le \epsilon_0 h^{- \frac{1}{\alpha}} \right \}} \zeta(\frac{z h^{\frac{1}{\alpha}}}{\epsilon_0}) \mu^h(ds, dz) + \int_0^1 \int_{\left \{ |z| \ge \epsilon_0 h^{- \frac{1}{\alpha}} \right \}} \mu^h(ds, dz), $$
$$W^h:= \mathcal{I}(V^h).$$
From the construction, $W^h$ is a Malliavin differentiable random variable such that $W^h \neq 0$ implies $\mu^h([0,1] \times \left \{ z \in \mathbb{R}: |z| > \epsilon_0 h^{- \frac{1}{\alpha}} \right \}) = 0 $. It is possible to show, acting as we did in \eqref{eq: prob omega n}, that $\mathbb{P}(W^h \neq 1) \le \mathbb{P}(\mu^h \mbox{ has a jump of size} > \frac{1}{2} \epsilon_0 h^{- \frac{1}{\alpha}} ) = O(h)$. From the latter, it is clear that the proof of the lemma reduces in proving the result on 
$|\mathbb{E}[g(L_1^{h})W^h] - \mathbb{E}[g(L_1^{\alpha, h})W^h]|.$
Considering a regularizing sequence $(g_p)$ converging to $g$ in $L^1$ norm, such that $\forall p$ $g_p$ is $\mathcal{C}^1$ with bounded derivative and $\left \| g_p \right \|_\infty \le \left \| g \right \|_\infty$, we may assume that $g$ is $\mathcal{C}^1$ with bounded derivative too. Using the integration by part formula \eqref{eq: int per parti} and denoting by $G$ any primitive function of $g$ we can write $\mathbb{E}[g(L_1^h)W^h] = \mathbb{E}[G(L_1^h) \mathcal{H}_{L_1^h}(W^h)]$ where the Malliavin weight can be written, using \eqref{eq: def peso Maillavin} and the chain rule property of the operator $\Gamma$, as
\begin{equation}
\mathcal{H}_{L_1^h}(W^h) = W^h \mathcal{H}_{L_1^h}(1) - \frac{\Gamma (W^h, L_1^h)}{\Gamma(L_1^h, L_1^h)}.
\label{eq: peso maillavin su L1n}
\end{equation}
Using the triangle inequality, we are now left to find upper bounds for the following two terms:
$$\tilde{T}_1 := |\mathbb{E}[g(L_1^{\alpha, h})W^h] - \mathbb{E}[G(L_1^{ \alpha, h}) \mathcal{H}_{L_1^h}(W^h)]|,$$
$$\tilde{T}_2 := |\mathbb{E}[G(L_1^{ \alpha, h}) \mathcal{H}_{L_1^h}(W^h)] - \mathbb{E}[G(L_1^{ h}) \mathcal{H}_{L_1^h}(W^h)]|.$$
Let us start considering $\tilde{T}_2$. Using the Lipschitz property of the function $G$ and \eqref{eq: peso maillavin su L1n} we have it is upper bounded by
$$\mathbb{E}[|g(\hat{L}_1)||L_1^{ \alpha, h} - L_1^{h}||\mathcal{H}_{L_1^h}(W^h)|] \le \mathbb{E}[|g(\hat{L}_1)||L_1^{ \alpha, h} - L_1^{h}||W^h \mathcal{H}_{L_1^h}(1)|] + \mathbb{E}[|g(\hat{L}_1)||L_1^{ \alpha, h} - L_1^{h}||\frac{\Gamma (W^h, L_1^h)}{\Gamma(L_1^h, L_1^h)}|] =$$
$$ = : \tilde{T}_{2,1} + \tilde{T}_{2,2}, $$
where $\hat{L}_1$ is between $L_1^{ \alpha, h}$ and $L_1^{ h}$. We focus on $\tilde{T}_{2,1}$. Using the definitions \eqref{eq: definition L n} and \eqref{eq: definition L alpha n} of $L_1^h$ and $L_1^{\alpha, h}$ it is
$$\tilde{T}_{2,1} \le \mathbb{E}[|g(\hat{L}_1)||\int_0^1 \int_\mathbb{R} (g_h(z) - z) \tilde{\mu}^h(ds, dz) ||\mathcal{H}_{L_1^h}(1)W^h|] \le \mathbb{E}[|g(\hat{L}_1)||\int_0^1 \int_{|z| \le 1} (g_h(z) - z) \tilde{\mu}^h(ds, dz) ||\mathcal{H}_{L_1^h}(1)W^h|] +$$
\begin{equation}
 + \mathbb{E}[|g(\hat{L}_1)||\int_0^1 \int_{ 1 \le |z| \le \epsilon_0 h^{- \frac{1}{\alpha}} } (g_h(z) - z) {\mu}^h(ds, dz) ||\mathcal{H}_{L_1^h}(1)W^h|],
\label{eq: splitto T21}
\end{equation}
where we have used that $g_h$ is an odd function with the symmetry of the compensator $\bar{\mu}^h$ and the fact that on $W_h \neq 0$ we have $\mu^h([0,1] \times \left \{ z \in \mathbb{R}: |z| > \epsilon_0 h^{- \frac{1}{\alpha}} \right \}) = 0 $. 
For the sake of shortness, we only give the details of the proof in the case $\alpha \neq 1$. In the case $\alpha = 1$, one needs to modify this control with an additional logarithmic term. For the small jumps term, from inequality 2.1.37 in \cite{13 in Maillavin} and the second point of Lemma \ref{lemma: 4.5 in Maillavin} we deduce $\mathbb{E}[|\int_0^1 \int_{|z| \le 1} (g_h(z) - z) \tilde{\mu}^h(ds, dz) |^{q_1}] \le C_{q_1}(h + h^{\frac{1}{\alpha}})^{q_1}$, $\forall q_1 \ge 2$. Using it and Holder inequality with $q_1$ big and $q_2$ close to $1$ we have 
$$\mathbb{E}[|g(\hat{L}_1)||\int_0^1 \int_{|z| \le 1} (g_h(z) - z) \tilde{\mu}^h(ds, dz) ||\mathcal{H}_{L_1^h}(1)W^h|] \le C_{q_1}(h + h^{\frac{1}{\alpha}}) \mathbb{E}[|g(\hat{L}_1)|^{q_2}|\mathcal{H}_{L_1^h}(1)|^{q_2}W^h]^\frac{1}{q_2} \le $$
\begin{equation}
\le C_{q_1}(h + h^{\frac{1}{\alpha}}) \mathbb{E}[|g(\hat{L}_1)|^{p_1\,q_2}W^h]^\frac{1}{q_2 p_1} \mathbb{E}[|\mathcal{H}_{L_1^h}(1)|^{q_2 p_2}]^\frac{1}{q_2 p_2},
\label{eq: dettata da gigi}
\end{equation}
where in the last inequality we have used again Holder inequality, with $p_2$ big and $p_1$ close to $1$. Using the first point of Lemma \ref{lemma: lemma 4.3 Maillavin}, we know that $\mathbb{E}[|\mathcal{H}_{L_1^h}(1)|^{q_2 p_2}]^\frac{1}{q_2 p_2}$ is bounded, hence \eqref{eq: dettata da gigi} is upper bounded by
\begin{equation}
C_{q_1 q_2 p_2} h \left \| g \right \|_\infty + C_{q_1 q_2 p_2} h^\frac{1}{\alpha} \mathbb{E}[|g(\hat{L}_1)W^h|^{p_1\,q_2}]^\frac{1}{q_2 p_1}, 
\label{eq: prima parte T21}
\end{equation}
where we have bounded $|g(\hat{L}_1)|$ with its infinity norm and used that $0 \le W^h \le 1$. We remind that we are considering $q_2$ and $p_1$ next to $1$, hence we can write $q_2 p_1$ as $1 + \epsilon$. We now introduce $r$ in the following way:
$$\mathbb{E}[|g(\hat{L}_1)|^{1 + \epsilon}W^h]^\frac{1}{1 + \epsilon} = \mathbb{E}[|g(\hat{L}_1)|^{(1 + \epsilon) r}|g(\hat{L}_1)|^{(1 + \epsilon)(1 - r)}W^h]^\frac{1}{1 + \epsilon} \le \left \| g \right \|_\infty^r \mathbb{E}[|g(\hat{L}_1)|^{(1 + \epsilon)(1 - r)}W^h]^\frac{1}{1 + \epsilon} \le $$
\begin{equation}
\left \| g \right \|_\infty^r \left \| g \right \|_{pol}^{1 - r} \mathbb{E}[(1 + |\hat{L}_1|^p)^{(1 + \epsilon)(1 - r)}W^h]^\frac{1}{1 + \epsilon} \le c \left \| g \right \|_\infty^r \left \| g \right \|_{pol}^{1 - r} + c \left \| g \right \|_\infty^r \left \| g \right \|_{pol}^{1 - r} \mathbb{E}[|\hat{L}_1|^{p(1 + \epsilon)(1 - r)}W^h]^\frac{1}{1 + \epsilon};
\label{eq: spezzo h su norma inf e norma pol}
\end{equation}
where we have estimated $g$ with its norm $\infty$ and we have used the property \eqref{eq: conditon on h} of $g$ and that $0 \le W^h\le 1$. We observe that $\hat{L}_1$ is between $L_1^h$ and $L_1^{\alpha,h}$ hence $|\hat{L}_1| \le |L_1^h| + |L_1^{\alpha,h}|$. Moreover we choose $r$ such that $p(1 + \epsilon)(1 - r) = \alpha$; therefore $r = 1 - \frac{\alpha}{p(1 + \epsilon)}$. In this way we have that \eqref{eq: spezzo h su norma inf e norma pol} is upper bounded by
\begin{equation}
c \left \| g \right \|_\infty^{1 - \frac{\alpha}{p(1 + \epsilon)}} \left \| g \right \|_{pol}^{\frac{\alpha}{p(1 + \epsilon)}} \log(h^{- \frac{1}{\alpha}}),
\label{eq: finale per z <1}
\end{equation}
where we have used that $\mathbb{E}[|\hat{L}_1|^\alpha W^h] \le c \log(h^{- \frac{1}{\alpha}})$, that we justify now. Indeed, using Lemma 2.1.5 in the appendix of \cite{13 in Maillavin} if $\alpha \in [1,2]$ and Jensen inequality if $\alpha \in [0,1)$, we have
$$\mathbb{E}[|\hat{L}_1|^\alpha W^h] \le c \mathbb{E}[(|L_1^h|^\alpha + |L_1^{\alpha,h}|^\alpha ) W^h] \le c \mathbb{E}[|\int_0^1 \int_{|z| \le 1}z \tilde{\mu}^h(ds, dz)|] + c \mathbb{E}[|\int_0^1 \int_{|z| \le 1}g_h(z) \tilde{\mu}^h(ds, dz)|] +$$
$$+  c \mathbb{E}[\int_0^1 \int_{1 \le |z| \le \epsilon_0 h^{- \frac{1}{\alpha}}}|z|^\alpha \bar{\mu}^h(ds, dz)] +  c \mathbb{E}[\int_0^1 \int_{1 \le |z| \le \epsilon_0 h^{- \frac{1}{\alpha}}}|g_h(z)|^\alpha \bar{\mu}^h(ds, dz)].$$
We observe that, using Kunita inequality, the first term here above is bounded in $L^p$ and, as a consequence of the second point of Lemma \ref{lemma: 4.5 in Maillavin}, the second term here above so does.
Concerning the third term here above (and so, again, we act on the fourth in the same way), we have
\begin{equation}
 c \mathbb{E}[\int_0^1 \int_{1 \le |z| \le \epsilon_0 h^{- \frac{1}{\alpha}}}|z|^\alpha \bar{\mu}^h(ds, dz)] \le c \int_{1 \le |z| \le \epsilon_0 h^{- \frac{1}{\alpha}}}|z|^{\alpha - 1 - \alpha} dz \le c \, \log(h^{- \frac{1}{\alpha}}) \le c |\log(h)|,
\label{eq: hat L1 z < 1}
\end{equation}
where we have also used definition \eqref{eq: bar mu n} of $\bar{\mu}^h$. \\
Replacing \eqref{eq: finale per z <1} in \eqref{eq: prima parte T21} we get
\begin{equation}
\mathbb{E}[|g(\hat{L}_1)||\int_0^1 \int_{|z| \le 1} (g_h(z) - z) \tilde{\mu}^h(ds, dz) ||\mathcal{H}_{L_1^h}(1)W^h|] \le C_{q_1 q_2 p_2} h \left \| g \right \|_\infty + C_{q_1 q_2 p_2} h^\frac{1}{\alpha}\left \| g \right \|_\infty^{1 - \frac{\alpha}{p} + \epsilon} \left \| g \right \|_{pol}^{\frac{\alpha}{p} - \epsilon}\log(h^{- \frac{1}{\alpha}}),
\label{eq: finale unito z<1}
\end{equation}
where we have taken another $\epsilon$, using its arbitrariness. The constants depend also on it. \\
Let us now consider the large jumps term in \eqref{eq: splitto T21}. Using the second point of Lemma \ref{lemma: 4.5 in Maillavin} and the following basic inequality
$$\int_0^1 \int_{1 < |z| \le \epsilon_0 h^{- \frac{1}{\alpha}}}|z|^\delta \mu^h(ds,dz) \le \int_0^1 \int_{1 < |z| \le \epsilon_0 h^{- \frac{1}{\alpha}}}|z|^{\delta - 1} \mu^h(ds,dz) \int_0^1 \int_{1 < |z| \le \epsilon_0 h^{- \frac{1}{\alpha}}}|z| \mu^h(ds,dz) $$
for $\delta \ge 1$, we get it is upper bounded by
\begin{equation}
\mathbb{E}[|g(\hat{L}_1)|\int_0^1 \int_{1 < |z| \le \epsilon_0 h^{- \frac{1}{\alpha}}}(h^\frac{1}{\alpha}|z| + h |z|^\alpha )  \mu^h(ds,dz) \int_0^1 \int_{1 < |z| \le \epsilon_0 h^{- \frac{1}{\alpha}}}|z| \mu^h(ds,dz)|\mathcal{H}_{L_1^h}(1)|W^h].
\label{ eq: T21  seconda parte}
\end{equation}
We now use Holder inequality with $p_2$ big and $p_1$ next to $1$ and we observe that, from the second point of Lemma \ref{lemma: lemma 4.3 Maillavin}, it follows
$$\mathbb{E}[|\int_0^1 \int_{1 < |z| \le \epsilon_0 h^{- \frac{1}{\alpha}}}|z| \mu^h(ds,dz) \mathcal{H}_{L_1^h}(1)|^{p_2}]^\frac{1}{p_2} \le C_{p_2}.$$
Hence \eqref{ eq: T21  seconda parte} is upper bounded by
\begin{equation}
C_{p_2} \mathbb{E}[|g(\hat{L}_1)|^{p_1} |\int_0^1 \int_{1 < |z| \le \epsilon_0 h^{- \frac{1}{\alpha}}}(h^\frac{1}{\alpha}|z| + h |z|^\alpha)  \mu^h(ds,dz) |^{p_1} W^h]^\frac{1}{p_1} \le
\label{eq: uguale al caso con Gamma}
\end{equation}
\begin{equation}
\le  C_{p_2} \left \| g \right \|_\infty h \mathbb{E}[|\int_0^1 \int_{1 < |z| \le \epsilon_0 h^{- \frac{1}{\alpha}}}|z|^\alpha  \mu^h(ds,dz) |^{p_1}]^\frac{1}{p_1} + C_{p_2} h^\frac{1}{\alpha} \mathbb{E}[|g(\hat{L}_1)|^{p_1} |\int_0^1 \int_{1 < |z| \le \epsilon_0 h^{- \frac{1}{\alpha}}} |z|\mu^h(ds,dz) |^{p_1} W^h]^\frac{1}{p_1}. 
\label{eq: T21 riformulata}
\end{equation}
Concerning the first term of \eqref{eq: T21 riformulata}, we use Lemma 2.1.5 in the appendix of \cite{13 in Maillavin} with $p_1 = (1 + \epsilon) \in [1,2]$ and the definition of $F_h$ given in \eqref{eq: definition Fn}, getting
$$\mathbb{E}[|\int_0^1 \int_{1 < |z| \le \epsilon_0 h^{- \frac{1}{\alpha}}}|z|^\alpha  \mu^h(ds,dz) |^{1 + \epsilon}]^\frac{1}{1 + \epsilon} \le \mathbb{E}[\int_0^1 \int_{1 < |z| \le \epsilon_0 h^{- \frac{1}{\alpha}}}|z|^{\alpha(1 + \epsilon)}  \bar{\mu}^h(ds,dz)]^\frac{1}{1 + \epsilon} \le $$
\begin{equation}
\le c(\int_{1 < |z| \le \epsilon_0 h^{- \frac{1}{\alpha}}}|z|^{\alpha(1 + \epsilon) - 1 - \alpha} dz)^\frac{1}{1 + \epsilon} \le c h^{- \frac{\epsilon}{1 + \epsilon}}=c h^{- \epsilon},
\label{eq: stima int z alla alpha(1 + epsilon)}
\end{equation}
where we have used the arbitrariness of $\epsilon$ in the last equality. \\
On the second term of \eqref{eq: T21 riformulata} we act differently depending on whether or not $\alpha$ is more than $1$. If it does, we act as we did in \eqref{eq: spezzo h su norma inf e norma pol}, considering $p_1= 1 + \epsilon < \alpha$ and  introducing $r$, this time we set it such that the following equality holds: 
\begin{equation}
p(1 + \epsilon )(1 - r) + (1 + \epsilon) = \alpha.
\label{eq: definition r}
\end{equation}
We also use the property \eqref{eq: conditon on h} on $g$, hence it is upper bounded by
\begin{equation}
C_{p_2} h^\frac{1}{\alpha}  \left \| g \right \|_\infty^r  \left \| g \right \|_{pol}^{1 - r} \mathbb{E}[(1 + |\hat{L}_1|^{p(1 + \epsilon )(1 - r)})|\int_0^1 \int_{1 < |z| \le \epsilon_0 h^{- \frac{1}{\alpha}}} |z|\mu^h(ds,dz) |^{1 + \epsilon} W^h]^\frac{1}{1 + \epsilon}.
\label{eq: L1 hat}
\end{equation}
Now on the first term here above we use that $0 \le W^h \le 1$ and Lemma 2.1.5 in the appendix of \cite{13 in Maillavin} as we did in \eqref{eq: stima int z alla alpha(1 + epsilon)} in order to get
\begin{equation}
\mathbb{E}[|\int_0^1 \int_{1 < |z| \le \epsilon_0 h^{- \frac{1}{\alpha}}} |z|\mu^h(ds,dz) |^{1 + \epsilon} ]^\frac{1}{1 + \epsilon} \le c.
\label{eq: stima z alla 1+ epsilon}
\end{equation}
Moreover we observe, as we have already done, that $|\hat{L}_1| \le |L_1^h| + |L_1^{\alpha, h}|$ and that, from the second point of Lemma \ref{lemma: 4.5 in Maillavin}, there exists $c > 0$ such that $|g_h(z)| \le c |z|$; so we get
\begin{multline}
\mathbb{E}[|\hat{L}_1|^{p(1 + \epsilon )(1 - r)}|\int_0^1 \int_{1 < |z| \le \epsilon_0 h^{- \frac{1}{\alpha}}} |z|\mu^h(ds,dz) |^{1 + \epsilon} W^h]^\frac{1}{1 + \epsilon} \le \\
\le c + \mathbb{E}[|\int_0^1 \int_{1 < |z| \le \epsilon_0 h^{- \frac{1}{\alpha}}} |z|\mu^h(ds,dz) |^{p(1 + \epsilon)(1 - r) + (1 + \epsilon)}]^\frac{1}{1 + \epsilon} \le \\
\le c (\int_{1 < |z| \le \epsilon_0 h^{- \frac{1}{\alpha}}} |z|^\alpha |z|^{- 1 - \alpha}dz)^\frac{1}{1 + \epsilon} \le c \, \frac{1}{1 + \epsilon} \log(h^{- \frac{1}{\alpha}}) \le c|\log(h)|,
\label{eq: stima L1 alla alpha}
\end{multline}
having chosen a particular $r$ just in order to have the exponent here above equal to $\alpha$ and so having found out the same computation of \eqref{eq: hat L1 z < 1}. We have not considered the integral on $|z| \le 1$ because, as we have already seen above \eqref{eq: hat L1 z < 1}, the integral is bounded in $L^p$ and so we simply get \eqref{eq: stima z alla 1+ epsilon} again. From \eqref{eq: definition r} we obtain $r = 1 + \frac{1}{p} - \frac{\alpha}{p(1 + \epsilon)}$. Replacing it and using \eqref{eq: stima z alla 1+ epsilon} and \eqref{eq: stima L1 alla alpha} we get \eqref{eq: L1 hat} is upper bounded by 
\begin{equation}
C_{p_2} h^\frac{1}{\alpha}  \left \| g \right \|_\infty^{1 + \frac{1}{p} - \frac{\alpha}{p(1 + \epsilon)}}  \left \| g \right \|_{pol}^{-\frac{1}{p} + \frac{\alpha}{p(1 + \epsilon)}}(c + |\log(h)|) = C_{p_2} h^\frac{1}{\alpha}  \left \| g \right \|_\infty^{1 + \frac{1}{p} - \frac{\alpha}{p(1 + \epsilon)}}  \left \| g \right \|_{pol}^{-\frac{1}{p} + \frac{\alpha}{p(1 + \epsilon)}} |\log(h)|.
\label{eq: fine parte 2 T21}
\end{equation}
If otherwise $\alpha$ is less than $1$, then the second term of \eqref{eq: T21 riformulata} is upper bounded by
\begin{equation}
C_{p_2} h^{\frac{1}{\alpha}}\left \| g \right \|_\infty \mathbb{E}[|\int_0^1 \int_{1 < |z| \le \epsilon_0 h^{- \frac{1}{\alpha}}} |z|\mu^h(ds,dz) |^{p_1} W^h]^\frac{1}{p_1} \le C_{p_2} h^{\frac{1}{\alpha}}\left \| g \right \|_\infty h^{\frac{1}{1 + \epsilon} - \frac{1}{\alpha}} = C_{p_2} h^{\frac{1}{1 + \epsilon}}\left \| g \right \|_\infty,
\label{eq: secondo termine, alpha < 1}
\end{equation}
where we have taken $p_1 = 1 + \epsilon$ and we have used the fact that $0 \le W^h \le 1$ and that, for $\alpha < 1$,
$$\mathbb{E}[|\int_0^1 \int_{1 < |z| \le \epsilon_0 h^{- \frac{1}{\alpha}}} |z|\mu^h(ds,dz) |^{1 + \epsilon} ]^\frac{1}{1 + \epsilon} \le c h^{\frac{1}{1 + \epsilon} - \frac{1}{\alpha}}.$$
Using \eqref{eq: T21 riformulata}, \eqref{eq: stima int z alla alpha(1 + epsilon)}, \eqref{eq: fine parte 2 T21} and \eqref{eq: secondo termine, alpha < 1} it follows
$$\mathbb{E}[|g(\hat{L}_1)||\int_0^1 \int_{ 1 \le |z| \le \epsilon_0 h^{- \frac{1}{\alpha}} } (g_h(z) - z) {\mu}^h(ds, dz) ||\mathcal{H}_{L_1^h}(1)W^h|] \le$$
\begin{equation}
\le C_{p_2} h^{1 - \epsilon} \left \| g \right \|_\infty + C_{p_2} h^\frac{1}{\alpha}  \left \| g \right \|_\infty^{1 + \frac{1}{p} - \frac{\alpha}{p(1 + \epsilon)}}  \left \| g \right \|_{pol}^{-\frac{1}{p} + \frac{\alpha}{p(1 + \epsilon)}}  |\log(h)|1_{\alpha > 1}.
\label{eq: grandi salti T21}
\end{equation}
Now from \eqref{eq: splitto T21}, \eqref{eq: finale unito z<1}, and \eqref{eq: grandi salti T21} it follows \begin{equation}
\tilde{T}_{2,1} \le  C_{q_1 q_2 p_2} h^{1 - \epsilon} \left \| g \right \|_\infty + C_{q_1 q_2 p_2} h^\frac{1}{\alpha}\left \| g \right \|_\infty^{1 - \frac{\alpha}{p} + \epsilon} \left \| g \right \|_{pol}^{\frac{\alpha}{p} - \epsilon}  |\log(h)| + C_{q_1 q_2 p_2} h^\frac{1}{\alpha}\left \| g \right \|_\infty^{1 + \frac{1}{p}- \frac{\alpha}{p} + \epsilon} \left \| g \right \|_{pol}^{- \frac{1}{p } + \frac{\alpha}{p} - \epsilon}  |\log(h)|1_{\alpha > 1}.
\label{es: stima T21 completa}
\end{equation}
Concerning $\tilde{T}_{2,2}$, it is already proved in Theorem 4.2 in \cite{Maillavin} that
\begin{equation}
\tilde{T}_{2,2} \le c h \left \| g \right \|_\infty.
\label{eq: estimation T22}
\end{equation}
Let us now consider $\tilde{T}_1$. Using \eqref{eq: definition Gamma} and \eqref{eq: def peso Maillavin} we can write
$$\mathcal{H}_{L_1^h}(W^h) = \frac{- W^h \, L(L_1^h)}{\Gamma(L_1^h, L_1^h)} + L(\frac{W^h}{\Gamma(L_1^h, L_1^h)}) L_1^h - L(\frac{L_1^h \, W^h}{\Gamma(L_1^h, L_1^h)}).$$
With computations using that L is a self-adjoint operator we get 
\begin{equation}
\tilde{T}_1 = |\mathbb{E}[g(L_1^{\alpha, h}) W^h] - \mathbb{E}[g(L_1^{\alpha, h}) \frac{\Gamma(L_1^{\alpha, h}, L_1^h)}{\Gamma(L_1^h, L_1^h)} W^h]| \le \mathbb{E}[|g(\hat{L}_1)||\frac{\Gamma(L_1^h - L_1^{\alpha, h}, L_1^h)}{\Gamma(L_1^h, L_1^h)}| W^h].
\label{eq: T1}
\end{equation}
Using equation \eqref{eq: definition Gamma}, we have 
$$\Gamma(L_1^h - L_1^{\alpha, h}, L_1^h) = \int_0^1 \int_{|z| < \frac{\eta}{2} h^{- \frac{1}{\alpha}}} \rho(z) (1 - g_h'(z)) \mu^h(ds, dz).$$
Using the third point of Lemma \ref{lemma: 4.5 in Maillavin} we deduce the following on the event $W^h \neq 0$:
$$|\Gamma(L_1^h - L_1^{\alpha, h}, L_1^h)| \le c \int_0^1 \int_{|z| \le \epsilon_0 h^{- \frac{1}{\alpha}}} \rho(z) (h^{\frac{1}{\alpha}} |z| + h |z|^\alpha) \mu^h(ds, dz) \le  c \int_0^1 \int_{ |z| \le 1} \rho(z) (h^{\frac{1}{\alpha}} |z| + h |z|^\alpha) \mu^h(ds, dz) +$$
$$ + c \int_0^1 \int_{1 < |z| \le \epsilon_0 h^{- \frac{1}{\alpha}}} \rho(z) \mu^h(ds, dz)\int_0^1 \int_{ 1 < |z| \le \epsilon_0 h^{- \frac{1}{\alpha}}}(h^{\frac{1}{\alpha}} |z| + h |z|^\alpha) \mu^h(ds, dz) \le $$ 
$$\le c \int_0^1 \int_{\mathbb{R}} \rho(z) \mu^h(ds, dz)( h^\frac{1}{\alpha} + h) + c\int_0^1 \int_{\mathbb{R}} \rho(z) \mu^h(ds, dz)\int_0^1 \int_{ 1 < |z| \le \epsilon_0 h^{- \frac{1}{\alpha}}}(h^{\frac{1}{\alpha}} |z| + h |z|^\alpha) \mu^h(ds, dz) = $$
\begin{equation}
= c( h^\frac{1}{\alpha} + h) \Gamma(L_1^h, L_1^h) + c \Gamma(L_1^h, L_1^h)(\int_0^1 \int_{ 1 < |z| \le \epsilon_0 h^{- \frac{1}{\alpha}}}(h^{\frac{1}{\alpha}} |z| + h |z|^\alpha) \mu^h(ds, dz)),
\label{eq: Gamma in T1}
\end{equation}
where we have used that $z$ is always less than $1$ in the first integral and that, since $\rho$ is a positive function, we can upper bound the integrals considering whole set $\mathbb{R}$. Moreover, we have used the definition of $\Gamma(L_1^h, L_1^h)$. Replacing \eqref{eq: Gamma in T1} in \eqref{eq: T1} we get
\begin{equation}
\tilde{T}_1 \le c ( h^\frac{1}{\alpha} + h) \mathbb{E}[|g(\hat{L}_1)|] + c\mathbb{E}[|g(\hat{L}_1)|\int_0^1 \int_{ 1 < |z| \le \epsilon_0 h^{- \frac{1}{\alpha}}}(h^{\frac{1}{\alpha}} |z| + h |z|^\alpha) \mu^h(ds, dz))] = : \tilde{T}_{1,1} + \tilde{T}_{1,2}.
\label{eq: splitto T1}
\end{equation}
Concerning $\tilde{T}_{1,1}$, we have
\begin{equation}
\tilde{T}_{1,1} \le c h \left \| g \right \|_\infty + c h^\frac{1}{\alpha}\mathbb{E}[|g(\hat{L}_1)|] \le c h \left \| g \right \|_\infty + c h^\frac{1}{\alpha} \left \| g \right \|_\infty^{1 - \frac{\alpha}{p}} \left \| g \right \|_{pol}^{\frac{\alpha}{p}} |\log(h)|,
\label{eq: T11}
\end{equation}
where in the last inequality we have acted exactly like we did in \eqref{eq: spezzo h su norma inf e norma pol} and \eqref{eq: finale per z <1} with the exponent on $g$ that is exactly equal to $1$ instead of $1 + \epsilon$ and so we have chosen $r$ such that $p(1 - r) = \alpha$. Let us now consider $\tilde{T}_{1,2}$. We observe that it is exactly like \eqref{eq: uguale al caso con Gamma} but with $p_1 = 1$ instead of $p_1 = 1 + \epsilon$, with the only difference that computing \eqref{eq: stima int z alla alpha(1 + epsilon)} now we get $c \, \log(h^{- \frac{1}{\alpha}})$ instead of $c h^{- \epsilon}$ and in the definition \eqref{eq: definition r} we choose $r$ such that $p(1 - r) + 1 = \alpha$. Acting exactly like we did above it follows
\begin{equation}
\tilde{T}_{1,2} \le C_{p_2} h |\log(h)| \left \| g \right \|_\infty + C_{p_2} h^\frac{1}{\alpha}  \left \| g \right \|_\infty^{1 + \frac{1}{p} - \frac{\alpha}{p}}  \left \| g \right \|_{pol}^{-\frac{1}{p} + \frac{\alpha}{p}} |\log(h)|1_{\alpha > 1}.
\label{eq: T12}
\end{equation}
Using \eqref{es: stima T21 completa}, \eqref{eq: estimation T22}, \eqref{eq: T11} and \eqref{eq: T12}, the lemma is proved.
\end{proof}
It follows Proposition \ref{prop: estimation stable}, using also \eqref{eq: inizio prop maillavin}, \eqref{eq: stima primo termine Maillavin} and \eqref{eq: differenza S1 e L1 alpha n}.
\end{proof}

\appendix
\section{Appendix}
In this section we will prove the technical proposition and lemmas we have used.

\subsection{Proof of Lemma \ref{lemma: Moment inequalities}}
\begin{proof}
We start proving 1. From the dynamic \eqref{eq: model vol} of $a$ it is
$$ \mathbb{E}[|a_t - a_s|^p] \le \mathbb{E}[|\int_s^t \tilde{b}_u du|^p] + \mathbb{E}[|\int_s^t \tilde{a}_u dW_u|^p] + \mathbb{E}[|\int_s^t \hat{a}_u d\hat{W}_u|^p] +$$
$$ + \mathbb{E}[|\int_s^t \int_{\mathbb{R} \backslash \left \{0 \right \}} \tilde{\gamma}_u \, z \, \tilde{\mu}(du, dz)|^p] + \mathbb{E}[|\int_s^t \int_{\mathbb{R} \backslash \left \{0 \right \}} \hat{\gamma}_u \, z \, \tilde{\mu}_2(du, dz)|^p] = : \sum_{j = 1}^5 I_j.$$
In the following, since we will act on the two Brownian motions $W$ and $\hat{W}$ in the same way, we will not report $I_3$ anymore. Also considering the Poisson random measures, we will deal only with $I_4$ in detail, underlining that on $I_5$ the same reasoning applies.    
We use Burkholder - Davis - Gundy inequalities on the stochastic integral and Kunita inequality on the jump part, in addition to a repeated use of Jensen inequality to get
$$I_1 + I_2 + I_4 \le |t -s|^{p - 1} \int_s^t \mathbb{E}[|\tilde{b}_u|^p] du + \mathbb{E}[|\int_t^s (\tilde{a}_u)^2 du|^\frac{p}{2}] + \mathbb{E}[\int_t^s \int_{\mathbb{R} \backslash \left \{0 \right \}} |\tilde{\gamma}_u |^p  |z|^p \bar{\mu}(du, dz)] + $$
$$ + \mathbb{E}[|\int_t^s \int_{\mathbb{R} \backslash \left \{0 \right \}} (\tilde{\gamma}_u )^2  (z)^2 \bar{\mu}(du, dz)|^\frac{p}{2}] \le c |t -s|^p + |t -s|^{\frac{p}{2} - 1} \int_t^s \mathbb{E}[|\tilde{a}_u|^p ] du +$$
$$ + \int_s^t \mathbb{E}[|\tilde{\gamma}_u|^p] ds (\int_{\mathbb{R} \backslash \left \{0 \right \}} |z|^p F(z) dz ) + |t -s|^{\frac{p}{2} - 1} \int_s^t \mathbb{E}[|\tilde{\gamma}_u|^2] ds (\int_{\mathbb{R} \backslash \left \{0 \right \}} |z|^2 F(z) dz )  \le$$
$$\le c(|t -s|^p + |t -s|^\frac{p}{2} + |t -s| + |t -s|^\frac{p}{2}) \le c |t -s|,$$
with the inequalities above holding true also because all the coefficients in the dynamic of $a$ are supposed to be bounded. The reasoning here above joint with A3 also yields that, for all $q > 0$, $\sup_{t \ge 0} \mathbb{E}[|a_t|^q] < \infty$.\\
The proof of 2 follows the same lines as the proof of 1 above. \\
As we proved in point 1 that the volatility has bounded moments, it is possible to get points 3 and 4 from Theorem 66 of \cite{Protter GLM} and Proposition 3.1 in \cite{Shimizu}. The fifth point is showed in \cite{Chapitre 1}, below Lemma 1, and the last one in Section $8$ of \cite{GLM}.
\end{proof}

\subsection{Proof of Proposition 3}
\begin{proof}
In order to show \eqref{eq: espansione salti}, we reformulate $(\Delta X_i^J)^2 \varphi_{\Delta_{n}^\beta}(\Delta X_i)$ as 
\begin{equation}
(\Delta X_i^J)^2[\varphi_{\Delta_{n}^\beta}(\Delta X_i) - \varphi_{\Delta_{n}^\beta}(\Delta X_i^J)] +(\Delta X_i^J)^2[\varphi_{\Delta_{n}^\beta}(\Delta X_i^J) - \varphi_{\Delta_{n}^\beta}(\Delta \tilde{X}_i^J)] + (\Delta X_i^J - \Delta \tilde{X}_i^J)^2 \varphi_{\Delta_{n}^\beta}(\Delta \tilde{X}_i^J) + 
\label{eq: prop 1 riformulata}
\end{equation}
$$ + 2\Delta \tilde{X}_i^J (\Delta X_i^J - \Delta \tilde{X}_i^J) \varphi_{\Delta_{n}^\beta}(\Delta \tilde{X}_i^J) + (\Delta \tilde{X}_i^J)^2 \varphi_{\Delta_{n}^\beta}(\Delta \tilde{X}_i^J) = : \sum_{k = 1}^{5}I_k^n(i). $$
Comparing \eqref{eq: espansione salti} with \eqref{eq: prop 1 riformulata} it turns out that our goal is to show that $\sum_{k = 1}^{4}I_k^n(i) = o_{L^1}(\Delta_{n}^{\beta(2 - \alpha) + 1)})$. In the sequel will prove that $\sum_{k = 1}^{4}\mathbb{E}[|I_k^n(i)|] \le c \Delta_{n}^{\beta(2 - \alpha) + 1}$; the same reasoning applies to the conditional version, that is  $\sum_{k = 1}^{4}\mathbb{E}_i[|I_k^n(i)|] \le R_i(\Delta_{n}^{\beta(2 - \alpha) + 1})$. \\ Let us start considering $I_1^n(i)$.
We know that $\Delta X_i = \Delta X_i^c + \Delta X_i^J$, where we have denoted by $\Delta X_i^c$ the continuous part of the increments of the process $X$.
We study
\begin{equation}
I_1^n(i) = I_{1,1}^n + I_{1,2}^n := I_1^n(i) 1_{\left \{ |\Delta X_i| \ge 3 \Delta_{n}^\beta \right \} } + I_1^n(i) 1_{ \left \{ |\Delta X_i| < 3 \Delta_{n}^\beta \right \}},
\label{eq: splitto I1}
\end{equation}
having omitted the dependence upon i in $I_{1,1}^n$ and $I_{1,2}^n$ in order to make the notation easier.
Concerning $I_{1,1}^n$, we split again on the sets $\left \{ |\Delta X_i^J| \ge 2 \Delta_{n}^\beta \right \}$ and $\left \{ |\Delta X_i^J| < 2 \Delta_{n}^\beta \right \}$. Recalling that $\varphi(\zeta) = 0$ for $|\zeta| \ge 2 \Delta_{n}^\beta$, we observe that if $|\Delta X_i^J| \ge 2 \Delta_{n}^\beta $ then $I_{1,1}^n$ is just $0$. Otherwise, if $|\Delta X_i^J| < 2 \Delta_{n}^\beta $, then it means that $|\Delta X^c_i|$ must be more than $\Delta_{n}^\beta$, so we can use \eqref{eq: prob parte continua}. In the sequel the constant $c$ may change value from line to line.  Using the bound on $(\Delta X_i^J)^2$ and the boundedness of $\varphi$ we get
\begin{equation}
\mathbb{E}[|I_{1,1}^n|] \le c \Delta_{n}^{2 \beta} \mathbb{E}[1_{\left \{ |\Delta X_i| \ge 3 \Delta_{n}^\beta, |\Delta X_i^J| < 2 \Delta_{n}^\beta \right \} }] \le c \Delta_{n}^{2 \beta} \mathbb{P}(|\Delta X^c_i| \ge \Delta_{n}^\beta ) \le c \Delta_{n}^{2 \beta + (\frac{1}{2} - \beta) r }.
\label{eq: E I11}
\end{equation}
Hence
\begin{equation}
\frac{1}{\Delta_{n}^{1 + \beta(2 - \alpha)}} \mathbb{E}[|I_{1,1}^n|] \le c \Delta_{n}^{(\frac{1}{2} - \beta) r - 1 + \alpha \beta },
\label{eq: conv I11}
\end{equation}
that goes to $0$ for $n \rightarrow \infty$ since for each choice of $\beta \in (0, \frac{1}{2})$ and $\alpha \in (0,2)$ we can always find $r$ big enough such that the exponent on $\Delta_{n}$ is positive. \\
We now consider $I_{1,2}^n$ on the sets $\left \{ |\Delta X_i^J| \ge 4 \Delta_{n}^\beta \right \}$ and $\left \{ |\Delta X_i^J| < 4 \Delta_{n}^\beta \right \}$. Using the boundedness of $\varphi$ we have
$$\mathbb{E}[|I_{1,2}^n|1_{\left \{ |\Delta X_i^J| \ge 4 \Delta_{n}^\beta \right \} } ] \le c \mathbb{E}[(\Delta X_i^J)^2 1_{\left \{ |\Delta X_i| < 3 \Delta_{n}^\beta, |\Delta X_i^J| \ge 4 \Delta_{n}^\beta \right \} } ] . $$
We observe that also in this case $|\Delta X_i| < 3 \Delta_{n}^\beta$ and $|\Delta X_i^J| \ge 4 \Delta_{n}^\beta$ involve $|\Delta X^c_i| \ge \Delta_{n}^\beta $. Moreover $(\Delta X_i^J)^2 \le c (\Delta X_i)^2 + c (\Delta X_i^c)^2 \le c \Delta_{n}^{2 \beta} + c (\Delta X_i^c)^2 $,
hence 
$$\mathbb{E}[|I_{1,2}^n|1_{\left \{ |\Delta X_i^J| \ge 4 \Delta_{n}^\beta \right \} } ] \le c \Delta_{n}^{2 \beta} \mathbb{P}(|\Delta X^c_i| \ge \Delta_{n}^\beta) + c \mathbb{E}[(\Delta X_i^c)^2 1_{\left \{ |\Delta X^c_i| \ge \Delta_{n}^\beta \right \} }] \le $$
\begin{equation}
\le c \Delta_{n}^{2 \beta + r(\frac{1}{2} - \beta)} + c \mathbb{E}[(\Delta X_i^c)^4]^\frac{1}{2}\mathbb{P}(|\Delta X^c_i| \ge \Delta_{n}^\beta)^\frac{1}{2} \le c \Delta_{n}^{[2 \beta + r(\frac{1}{2} - \beta)] \land [1+ \frac{r}{2}(\frac{1}{2} - \beta) ]}, 
\label{eq: E I12 salti grandi}
\end{equation}
where we have used Cauchy Schwartz inequality, \eqref{eq: prob parte continua} and the sixth point of Lemma \ref{lemma: Moment inequalities}.
Therefore we get
\begin{equation}
\frac{1}{\Delta_{n}^{1 + \beta(2 - \alpha)}} \mathbb{E}[|I_{1,2}^n|1_{\left \{ |\Delta X_i^J| \ge 4 \Delta_{n}^\beta \right \} } ] \le c \Delta_{n}^{ [r(\frac{1}{2} - \beta) - 1 + \alpha \beta] \land [\frac{r}{2}(\frac{1}{2} - \beta) - \beta(2 - \alpha)]},
\label{eq: conv prima parte I12}
\end{equation}
that converges to $0$ for $n \rightarrow \infty$ since we can always find $r \ge 1$ such that the exponent $\Delta_{n}$ is positive.\\
In order to conclude the study of $I_1^n(i)$, we study $I_{1,2}^n 1_{\left \{ |\Delta X_i^J| < 4 \Delta_{n}^\beta \right \} }$.
\begin{equation}
\mathbb{E}[|I_{1,2}^n| 1_{\left \{ |\Delta X_i^J| < 4 \Delta_{n}^\beta \right \} }] \le c \left \| \varphi' \right \|_\infty \Delta_{n}^{-\beta} \mathbb{E}[ (\Delta X_i^J)^2|\Delta X_i - \Delta X_i^J|1_{\left \{ |\Delta X_i| \le 3 \Delta_{n}^\beta, |\Delta X_i^J| \le 4 \Delta_{n}^\beta  \right \}}],
\label{eq: E I12 salti piccoli}
\end{equation}
where we have used the smoothness of $\varphi$. Using Holder inequality and the sixth point of Lemma \ref{lemma: Moment inequalities} it is upper bounded by
\begin{equation}
 c \Delta_{n}^{ -\beta} \mathbb{E}[|\Delta X_i^c|^p]^\frac{1}{p} \mathbb{E}[|(\Delta X_i^J)^{2q} 1_{\left \{ |\Delta X_i| \le 3 \Delta_{n}^\beta, |\Delta X_i^J| \le 4 \Delta_{n}^\beta  \right \}} ]^\frac{1}{q} \le c \Delta_{n}^{\frac{1}{2} - \beta} \mathbb{E}[|(\Delta X_i^J)^{2q} 1_{\left \{ |\Delta X_i| \le 3 \Delta_{n}^\beta, |\Delta X_i^J| \le 4 \Delta_{n}^\beta  \right \}}]^\frac{1}{q}. 
\label{eq: passaggio per I1}
\end{equation}
Now, since our indicator function $1_{\left \{ |\Delta X_i| \le 3 \Delta_{n}^\beta, |\Delta X_i^J| \le 4 \Delta_{n}^\beta  \right \}}$ is less then $1_{\left \{ |\Delta X_i^J| \le 4 \Delta_{n}^\beta  \right \}}$ , we can use the first point of Lemma \ref{lemma: estensione 10 capitolo 1}. Through the use of the conditional expectation we get
\begin{equation}
\mathbb{E}[|(\Delta X_i^J)^{2q} 1_{\left \{ |\Delta X_i| \le 3 \Delta_{n}^\beta, |\Delta X_i^J| \le 4 \Delta_{n}^\beta  \right \}}]^\frac{1}{q} \le c\Delta_{n}^\frac{1 + \beta(2 q - \alpha)}{q}\mathbb{E}[R_i( 1)] \le c\Delta_{n}^\frac{1 + \beta(2 q - \alpha)}{q}.
\label{eq: I1 salti}
\end{equation}
Replacing \eqref{eq: I1 salti} in \eqref{eq: passaggio per I1} and taking $q$ small (next to $1$), we obtain 
$\mathbb{E}[|I_{1,2}^n| 1_{\left \{ |\Delta X_i^J| < 4 \Delta_{n}^\beta \right \} }] \le c \Delta_{n}^{\frac{1}{2} + \beta + 1 - \alpha \beta - \epsilon}$.
It follows
\begin{equation}
\frac{\mathbb{E}[|I_{1,2}^n| 1_{\left \{ |\Delta X_i^J| < 4 \Delta_{n}^\beta \right \} }]}{\Delta_{n}^{\beta(2 - \alpha) + 1}} \le c \Delta_{n}^{\frac{1}{2} - \beta - \epsilon},
\label{eq: finale I1}
\end{equation}
that goes to $0$ for $n \rightarrow \infty$ since we can always find an $\epsilon$ as small as the exponent on $\Delta_{n} $ is positive, for $\beta \in (0, \frac{1}{2})$. \\
Let us now consider $I_2^n(i)$. 
\begin{equation}
I_2^n (i) = I_2^n (i) \, 1_{\left \{ |\Delta X_i^J| \le 2 \Delta_{n}^\beta \right \}}  + I_2^n (i) \, 1_{\left \{ |\Delta X_i^J| > 2 \Delta_{n}^\beta \right \}} = : I_{2,1}^n + I_{2,2}^n.
\label{eq: splitto I21 e I22}
\end{equation}
Concerning the first term of \eqref{eq: splitto I21 e I22}, we have
$$\mathbb{E}[|I_{2,1}^n|] \le \Delta_{n}^{- \beta} \left \| \varphi' \right \|_\infty  \mathbb{E}[(\Delta X_i^J)^2 |\Delta X_i^J - \Delta \tilde{X}_i^J|1_{\left \{ |\Delta X_i^J| \le 2 \Delta_{n}^\beta \right \}}] \le$$
\begin{equation}
 \le c \Delta_{n}^{- \beta} \mathbb{E}[(\Delta X_i^J )^41_{\left \{ |\Delta X_i^J| \le 2 \Delta_{n}^\beta \right \}} ]^\frac{1}{2} \mathbb{E}[|\Delta X_i^J - \Delta \tilde{X}_i^J|^2]^\frac{1}{2} , 
\label{eq: I21}
\end{equation}
where we have used the smoothness of $\varphi$ and Cauchy-Schwartz inequality. Using again the first point of Lemma \ref{lemma: estensione 10 capitolo 1}, we have that
\begin{equation}
\mathbb{E}[(\Delta X_i^J )^41_{\left \{ |\Delta X_i^J| \le 2 \Delta_{n}^\beta \right \}} ]^\frac{1}{2} = \mathbb{E}[\mathbb{E}_i[(\Delta X_i^J )^41_{\left \{ |\Delta X_i^J| \le 2 \Delta_{n}^\beta \right \}}] ]^\frac{1}{2} \le \Delta_{n}^{\frac{1 + \beta(4 - \alpha)}{2}} \mathbb{E}[R_i( 1)]\le c\Delta_{n}^{\frac{1}{2} + 2 \beta - \frac{\alpha \beta}{2} }.
\label{eq: I21 salti}
\end{equation}
We now introduce a lemma that will be proved later:
\begin{lemma}
Suppose that A1 -A4 hold. Then
\begin{enumerate}
\item $\forall q \ge 2$ we have
    \begin{equation}
    \mathbb{E}[|\Delta X_i^J - \Delta \tilde{X}_i^J|^q] \le c \Delta_{n}^2,
    \label{eq: diff salti alpha >1}
     \end{equation}
    \item for $q \in [1, 2]$ and $\alpha < 1$, we have 
    \begin{equation}
    \mathbb{E}[|\Delta X_i^J - \Delta \tilde{X}_i^J|^q]^\frac{1}{q} \le c \Delta_{n}^{\frac{1}{2} + \frac{1}{q}}.
    \label{eq: diff salti alpha <1}
    \end{equation}
\end{enumerate}
\label{lemma: differenza dei salti}
\end{lemma}
Replacing \eqref{eq: I21 salti} and \eqref{eq: diff salti alpha >1} in \eqref{eq: I21} we get
\begin{equation}
\mathbb{E}[|I_{2,1}^n|] \le c \Delta_{n}^{- \beta + \frac{1}{2} + 2 \beta - \frac{\alpha \beta}{2} + 1 }= c \Delta_{n}^{\frac{3}{2} + \beta - \frac{\alpha \beta}{2}}.
\label{eq: E I21}
\end{equation}
Hence
\begin{equation}
\frac{\mathbb{E}[|I_{2,1}^n|]}{ \Delta_{n}^{1 + \beta(2 - \alpha)}} \le c \Delta_{n}^{\frac{1}{2} - \beta + \frac{\alpha \beta}{2}},
\label{eq: con I21}
\end{equation}
that goes to $0$ for $n \rightarrow \infty$ since the exponent on $\Delta_{n}$ is positive for $\beta < \frac{1}{2(1 - \frac{\alpha}{2})}$, that is always true with $\alpha$ and $\beta$ in the intervals chosen. \\
We now want to show that also $I_{2,2}^n$ is $o_{L^1}(\Delta_{n}^{\beta(2 - \alpha) + 1})$. We split $I_{2,2}^n$ on the sets $\left \{ |\Delta \tilde{X}_i^J| \le 2 \Delta_{n}^\beta  \right \}$ and $\left \{ |\Delta \tilde{X}_i^J| > 2 \Delta_{n}^\beta  \right \}$. We observe that, by the definition of $\varphi$, $I_{2,2}^n$ is null on the second set. Adding and subtracting $\Delta \tilde{X}_i^J $ in $I_{2,2}^n 1_{\left \{ |\Delta \tilde{X}_i^J| \le 2 \Delta_{n}^\beta  \right \}}$ we have 
$$\mathbb{E}[|I_{2,2}^n| 1_{\left \{ |\Delta \tilde{X}_i^J| \le 2 \Delta_{n}^\beta  \right \}}] \le c\mathbb{E}[(\Delta X_i^J - \Delta \tilde{X}_i^J)^2|\varphi_{\Delta_{n}^\beta}(\Delta X_i^J) - \varphi_{\Delta_{n}^\beta}(\Delta \tilde{X}_i^J)|1_{\left \{ |\Delta \tilde{X}_i^J| \le 2 \Delta_{n}^\beta , |\Delta X_i^J| > 2 \Delta_{n}^\beta   \right \}}] +$$
\begin{equation}
+ c\mathbb{E}[(\Delta \tilde{X}_i^J)^2|\varphi_{\Delta_{n}^\beta}(\Delta X_i^J) - \varphi_{\Delta_{n}^\beta}(\Delta \tilde{X}_i^J)|1_{\left \{ |\Delta \tilde{X}_i^J| \le 2 \Delta_{n}^\beta    \right \}}].
\label{eq: I22 inizio}
\end{equation}
On the second term of \eqref{eq: I22 inizio} we can act exactly as we have done in $I_{2,1}^n$, with $\Delta \tilde{X}_i^J$ instead of $\Delta X_i^J$ (and so using \eqref{eq: estensione tilde salti lemma 10} instead of \eqref{eq: estensione salti lemma 10}). We get
\begin{equation}
\mathbb{E}[(\Delta \tilde{X}_i^J)^2|\varphi_{\Delta_{n}^\beta}(\Delta X_i^J) - \varphi_{\Delta_{n}^\beta}(\Delta \tilde{X}_i^J)|1_{\left \{ |\Delta \tilde{X}_i^J| \le 2 \Delta_{n}^\beta \right \}}] \le c \Delta_{n}^{\frac{3}{2} + \beta - \frac{\alpha \beta}{2}}.
\label{eq: primo termine I22}
\end{equation}
Concerning the first term of \eqref{eq: I22 inizio}, by the definition of $\varphi$ we know it is 
\begin{equation}
\mathbb{E}[(\Delta X_i^J - \Delta \tilde{X}_i^J)^2|- \varphi_{\Delta_{n}^\beta}(\Delta \tilde{X}_i^J)|1_{\left \{ |\Delta \tilde{X}_i^J| \le 2 \Delta_{n}^\beta , |\Delta X_i^J| > 2 \Delta_{n}^\beta   \right \}}] \le c\mathbb{E}[(\Delta X_i^J - \Delta \tilde{X}_i^J)^2] \le c\Delta_{n}^2,
\label{eq: secondo termine I22}
\end{equation}
where in the last inequality we have used \eqref{eq: diff salti alpha >1}. Using \eqref{eq: I22 inizio} - \eqref{eq: secondo termine I22} it follows
\begin{equation}
\mathbb{E}[|I_{2,2}^n|] = \mathbb{E}[|I_{2,2}^n| 1_{\left \{ |\Delta \tilde{X}_i^J| \le 2 \Delta_{n}^\beta  \right \}}] \le c \Delta_{n}^{\frac{3}{2} + \beta - \frac{\alpha \beta}{2}} + c\Delta_{n}^2 = c\Delta_{n}^{\frac{3}{2} + \beta - \frac{\alpha \beta}{2}},
\label{eq: E I22}
\end{equation}
considering that $\Delta_{n}^2$ is negligible compared to $\Delta_{n}^{\frac{3}{2} + \beta - \frac{\alpha \beta}{2}}$ since $\beta < \frac{1}{2(1 - \frac{\alpha}{2})}$. Hence
\begin{equation}
\frac{\mathbb{E}[|I_{2,2}^n|]}{ \Delta_{n}^{1 + \beta(2 - \alpha)}} \le c \Delta_{n}^{\frac{1}{2} - \beta + \frac{\alpha \beta}{2}},
\label{eq: I22 finale}
\end{equation}
that goes to $0$ for $n \rightarrow \infty$. \\
Concerning $I_3^n(i)$, we have 
\begin{equation}
\mathbb{E}[|I_3^n(i)|] \le c \mathbb{E}[(\Delta X_i^J - \Delta \tilde{X}_i^J)^2] \le c \Delta_{n}^2,
\label{eq: E I3}
\end{equation}
where the last inequality follows from \eqref{eq: diff salti alpha >1}. Hence $I_3^n(i) = o_{L^1}(\Delta_{n}^{\beta(2 - \alpha) + 1})$, indeed
\begin{equation}
\frac{\mathbb{E}[|I_3^n (i)|]}{ \Delta_{n}^{1 + \beta(2 - \alpha)}} \le c \Delta_{n}^{1 - 2\beta + \alpha \beta},
\label{eq: I3}
\end{equation}
that goes to $0$ for $n \rightarrow \infty$ considering that the exponent on $\Delta_{n}$ is positive for $\beta < \frac{1}{2 - \alpha}$, condition that is always satisfied for $\beta \in (0, \frac{1}{2})$ and $\alpha \in (0,2)$. \\
Let us now consider $I_4^n(i)$. Using Cauchy-Schwartz inequality it is
\begin{equation}
\mathbb{E}[|I_4^n(i)|] \le c \mathbb{E}[(\Delta X_i^J - \Delta \tilde{X}_i^J)^2]^\frac{1}{2} \mathbb{E}[(\Delta \tilde{X}_i^J)^2 \varphi^2_{\Delta_{n}^\beta}(\Delta \tilde{X}_i^J) ]^\frac{1}{2} \le c \Delta_{n} \Delta_{n}^{\frac{1}{2} + \frac{\beta}{2}(2 - \alpha)} = c \Delta_{n}^{\frac{3}{2} + \beta - \frac{\alpha \beta}{2} },
\label{eq: E I4}
\end{equation}
where we have used \eqref{eq: diff salti alpha >1} and the first point of Lemma \ref{lemma: estensione 10 capitolo 1}.
It follows
\begin{equation}
\frac{\mathbb{E}[|I_4^n(i)|]}{ \Delta_{n}^{1 + \beta(2 - \alpha)}} \le c \Delta_{n}^{\frac{1}{2} - \beta + \frac{\alpha \beta}{2}},
\label{eq: I4}
\end{equation}
that goes to $0$ for $n \rightarrow \infty$ since the exponent on $\Delta_{n}$ is more than $0$ if $\beta < \frac{1}{2(1 - \frac{\alpha}{2})}$, that is always true.
Using \eqref{eq: prop 1 riformulata}, \eqref{eq: conv I11}, \eqref{eq: conv prima parte I12}, \eqref{eq: finale I1}, \eqref{eq: con I21}, \eqref{eq: I22 finale}, \eqref{eq: I3} and \eqref{eq: I4} we obtain \eqref{eq: espansione salti}. \\ \\
In order to prove \eqref{eq: aggiunta prop1 salti}, we use again reformulation \eqref{eq: prop 1 riformulata}. Replacing it in the left hand side of \eqref{eq: aggiunta prop1 salti} it turns out that our goal is to show that 
\begin{equation}
\sum_{ i = 0}^{n - 1}(\sum_{k = 1}^4 I_k^n(i)) f(X_{t_i}) = o_\mathbb{P}(\Delta_n^{(\frac{1}{2} - \tilde{\epsilon}) \land (1 - \alpha \beta - \tilde{\epsilon} )}).
\label{eq: tesi per aggiunta 1 salti}
\end{equation}
Using a conditional on $\mathcal{F}_{t_i}$ version of \eqref{eq: splitto I21 e I22}, \eqref{eq: E I21} and \eqref{eq: E I22} we have
$$ \sum_{i = 0}^{n - 1} \mathbb{E}_i[|I_2^n(i) f(X_{t_i})|] \le \frac{1}{n} \sum_{ i = 0}^{n - 1} R_i (\Delta_n^{\frac{3}{2} + \beta - \frac{\alpha \beta}{2} - 1 - \epsilon}) = \frac{1}{n} \sum_{ i = 0}^{n - 1} R_i( \Delta_n^{\frac{1}{2} + \beta - \frac{\alpha \beta}{2} - \epsilon}).$$
Since $\beta (1 - \frac{\alpha}{2})$ is always more than zero and, $\forall \tilde{\epsilon} > 0$ we can always find $\epsilon$ smaller than it, we get
\begin{equation}
\sum_{ i = 0}^{n - 1}I_2^n(i) f(X_{t_i}) = o_{L^1}(\Delta_n^{\frac{1}{2} - \tilde{\epsilon}}) = o_{L^1}(\Delta_n^{(\frac{1}{2} - \tilde{\epsilon}) \land (1 - \alpha \beta - \tilde{\epsilon} )}).
\label{eq: I2 nuovo salti}
\end{equation}
From a conditional version of \eqref{eq: E I3} we get that $ \sum_{ i = 0}^{n - 1}I_3^n(i) f(X_{t_i})$ is upper bounded in $L^1$ norm by the $L^1$ norm of $\frac{1}{n} \sum_{ i = 0}^{n - 1}R_i( \Delta_n^{2 - 1 - \epsilon}) =\frac{1}{n} \sum_{ i = 0}^{n - 1} R_i( \Delta_n^{1 - \epsilon})$ and so
\begin{equation}
\sum_{ i = 0}^{n - 1}I_3^n(i) f(X_{t_i}) = o_{L^1}(\Delta_n^{(\frac{1}{2} - \tilde{\epsilon}) \land (1 - \alpha \beta - \tilde{\epsilon} )}).
\label{eq: I3 nuovo salti}   
\end{equation}
Using a conditional version of \eqref{eq: E I4} we get that $ \sum_{ i = 0}^{n - 1}I_4^n(i)f(X_{t_i})$ is upper bounded in $L^1$ norm by the $L^1$ norm of $\frac{1}{n} \sum_{ i = 0}^{n - 1}R_i(\Delta_n^{\frac{3}{2} + \beta - \frac{\alpha \beta}{2}- 1 - \epsilon}) = \frac{1}{n} \sum_{ i = 0}^{n - 1} R_i( \Delta_n^{\frac{1}{2} + \beta - \frac{\alpha \beta}{2} - \epsilon})$, hence
\begin{equation}
\sum_{ i = 0}^{n - 1}I_4^n(i) f(X_{t_i}) = o_{L^1}(\Delta_n^{(\frac{1}{2} - \tilde{\epsilon}) \land (1 - \alpha \beta - \tilde{\epsilon} )}).
\label{eq: I4 nuovo salti}   
\end{equation}
Concerning $I_1^n(i)$, we consider $I_{1,1}^n(i)$ and $I_{1,2}^n(i)$ as defined in \eqref{eq: splitto I1}. Using a conditional version of \eqref{eq: E I11} on $I_{1,1}^n(i)$ it follows that $n^{\frac{1}{2} - \tilde{\epsilon}} \sum_{ i = 0}^{n - 1}I_{1,1}^n(i) f(X_{t_i})$  is upper bounded in $L^1$ norm by the $L^1$ norm of $\frac{1}{n} \sum_{ i = 0}^{n - 1}R_i( \Delta_n^{(\frac{1}{2} - \beta)r + 2 \beta - 1 - \frac{1}{2} + \tilde{\epsilon}}) = \frac{1}{n} \sum_{ i = 0}^{n - 1} R_i(\Delta_n^{(\frac{1}{2} - \beta)r + 2 \beta - \frac{3}{2} + \tilde{\epsilon}})$, that goes to zero because we can find $r$ big enough such that the exponent on $\Delta_n$ is positive, hence
\begin{equation}
 \sum_{ i = 0}^{n - 1}I_{1,1}^n(i) f(X_{t_i}) = o_{L^1}(\Delta_n^{\frac{1}{2} - \tilde{\epsilon}}) = o_{L^1}(\Delta_n^{(\frac{1}{2} - \tilde{\epsilon}) \land (1 - \alpha \beta - \tilde{\epsilon} )}).
\label{eq: I11 nuovo salti}
\end{equation}
Acting as we did in the proof of \eqref{eq: espansione salti}, we consider $I_{1,2}^n(i)$ on the sets $\left \{ |\Delta X_i^J| \ge 4 \Delta_{n}^\beta \right \}$ and $\left \{ |\Delta X_i^J| < 4 \Delta_{n}^\beta \right \}$. Again, from \eqref{eq: E I12 salti grandi} and the arbitrariness of $r > 0$ it follows 
\begin{equation}
 \sum_{ i = 0}^{n - 1}I_{1,2}^n(i) 1_{\left \{ |\Delta X_i^J| \ge 4 \Delta_{n,i}^\beta \right \} } f(X_{t_i}) = o_{L^1}(\Delta_n^{(\frac{1}{2} - \tilde{\epsilon}) \land (1 - \alpha \beta - \tilde{\epsilon} )}).
\label{eq: I12 salti grandi nuovo salti}
\end{equation}
When $|\Delta X_i^J| < 4 \Delta_{n}^\beta$ we act in a different way, considering the development up to second order of $\varphi_{\Delta_{n}^\beta}$, centered in $\Delta X_i^J$:
$$I_{1,2}^n(i)1_{\left \{ |\Delta X_i^J| < 4 \Delta_{n}^\beta \right \} } = [(\Delta X_i^J)^2 \Delta X_i^c \varphi'_{\Delta_{n}^\beta}(\Delta X_i^J) \Delta_{n}^{- \beta} + (\Delta X_i^J)^2 (\Delta X_i^c)^2 \varphi''_{\Delta_{n}^\beta}(X_u) \Delta_{n}^{- 2\beta}]  1_{\left \{ |\Delta X_i| \le 3 \Delta_{n}^\beta, \, |\Delta X_i^J| < 4 \Delta_{n}^\beta \right \} } =$$
$$ = : \hat{I}_1^n(i) 1_{\left \{ |\Delta X_i| \le 3 \Delta_{n}^\beta, \, |\Delta X_i^J| < 4 \Delta_{n}^\beta \right \} } + \hat{I}_2^n (i) 1_{\left \{ |\Delta X_i| \le 3 \Delta_{n}^\beta, \, |\Delta X_i^J| < 4 \Delta_{n}^\beta \right \} },$$
where $X_u \in [\Delta X_i^J, \Delta X_i]$.
Now, acting like we did in \eqref{eq: E I12 salti piccoli}, \eqref{eq: passaggio per I1} and \eqref{eq: I1 salti}, taking $q$ next to $1$ we get
$$\mathbb{E}_i[|\hat{I}_2^n(i)  1_{\left \{ |\Delta X_i| \le 3 \Delta_{n}^\beta, \, |\Delta X_i^J| < 4 \Delta_{n}^\beta \right \} }|] \le R_i( \Delta_n^{1 + \beta(2 - \alpha) - \epsilon + 1 - 2 \beta}) = R_i( \Delta_n^{2 - \alpha \beta - \epsilon}).$$
Since for each $\tilde{\epsilon} > 0$ we can find an $\epsilon$ such that $\tilde{\epsilon} - \epsilon > 0$ it follows, taking the conditional expectation
\begin{equation}
\sum_{ i = 0}^{n - 1}\hat{I}_{2}^n (i) 1_{\left \{ |\Delta X_i| \le 3 \Delta_{n}^\beta, \, |\Delta X_i^J| < 4 \Delta_{n}^\beta \right \} } f(X_{t_i}) = o_{L^1}(\Delta_n^{1 - \alpha \beta - \tilde{\epsilon} }) =  o_{L^1}(\Delta_n^{(\frac{1}{2} - \tilde{\epsilon}) \land (1 - \alpha \beta - \tilde{\epsilon} )}).
\label{eq: hat I2 nuovo salti}
\end{equation}
Concerning $\hat{I}_1^n(i) 1_{\left \{ |\Delta X_i| \le 3 \Delta_{n}^\beta, \, |\Delta X_i^J| < 4 \Delta_{n}^\beta \right \} }$, we no longer consider the indicator function because it is
$$(\Delta X_i^J)^2 \Delta X_i^c \varphi'_{\Delta_{n}^\beta}(\Delta X_i^J) \Delta_{n}^{- \beta} + (\Delta X_i^J)^2 \Delta X_i^c \varphi'_{\Delta_{n}^\beta}(\Delta X_i^J) \Delta_{n}^{- \beta}( 1_{\left \{ |\Delta X_i| \le 3 \Delta_{n}^\beta, \, |\Delta X_i^J| < 4 \Delta_{n}^\beta \right \} } - 1) $$
and the second term here above is different from zero only on a set smaller that ${\left \{ |\Delta X_i| \ge 3 \Delta_{n}^\beta \right \} }$ or ${\left \{ |\Delta X_i^J| \ge 4 \Delta_{n}^\beta \right \} }$ , on which we have already proved the result (see the study of $I_{1,1}^n(i)$ in \eqref{eq: I11 nuovo salti} and $I_{1,2}^n(i)$ in \eqref{eq: I12 salti grandi nuovo salti}).
We want to show that
\begin{equation}
\sum_{ i = 0}^{n - 1}\hat{I}_{1}^n(i) f(X_{t_i})  =  o_{\mathbb{P}}(\Delta_n^{(\frac{1}{2} - \tilde{\epsilon}) \land (1 - \alpha \beta - \tilde{\epsilon} )}).
\label{eq: I1 nuovo salti}
\end{equation}
We start from the reformulation
$$\hat{I}_{1}^n (i)= \Delta X_i^c \Delta_{n}^{- \beta}[(\Delta X_i^J)^2(\varphi'_{\Delta_{n}^\beta}(\Delta X_i^J) - \varphi'_{\Delta_{n}^\beta}(\Delta \tilde{X}_i^J)) + (\Delta X_i^J - \Delta \tilde{X}_i^J)^2 \varphi'_{\Delta_{n}^\beta}(\Delta \tilde{X}_i^J)  + $$
$$ +2 \Delta \tilde{X}_i^J (\Delta X_i^J - \Delta \tilde{X}_i^J) \varphi'_{\Delta_{n}^\beta}(\Delta \tilde{X}_i^J) +(\Delta \tilde{X}_i^J)^2 \varphi'_{\Delta_{n}^\beta}(\Delta \tilde{X}_i^J)] = \sum_{j =1}^4 \hat{I}_{1, j}^n(i).$$
and we observe that, after have used Holder inequality and have remarked that $\varphi'_{\Delta_{n}^\beta}$ acts like $\varphi_{\Delta_{n}^\beta}$, we can act on $\hat{I}_{1, 1}^n$ as we did on $I_2^n$, on $\hat{I}_{1, 2}^n$ as on $I_3^n$ and on $\hat{I}_{1, 3}^n$ as on $I_4^n$. So we get, using also Holder inequality and the sixth point of Lemma \ref{lemma: Moment inequalities},
\begin{equation}
\mathbb{E}_i[|\hat{I}_{1, 1}^n(i) + \hat{I}_{1, 2}^n(i) + \hat{I}_{1, 3}^n(i)|] \le R_i(\Delta_{n}^{\frac{1}{2} - \beta})(\mathbb{E}_i[|I_2^n(i) |^q]^\frac{1}{q} +\mathbb{E}[|I_3^n(i) |^q]^\frac{1}{q} + \mathbb{E}[|I_4^n(i) |^q]^\frac{1}{q}).
\label{eq: da hat I a I}
\end{equation}
Now, taking $q$ next to $1$, we need the following lemma that we will prove later:
\begin{lemma}
Suppose that A1 - A4 hold. Then, $\forall \epsilon >0$,
\begin{equation}
\mathbb{E}_i[|I_2^n(i) |^{1 + \epsilon} +|I_3^n(i) |^{1 + \epsilon} + |I_4^n(i) |^{1 + \epsilon}]^{\frac{1}{1 + \epsilon}} \le R_i(\Delta_{n}^{\frac{3}{2} + \beta - \frac{\alpha \beta}{2} - \epsilon}),
\label{eq: tesi I con 1+epsilon}
\end{equation}
with $I_2^n(i)$, $I_3^n(i)$ and $I_4^n(i)$ as defined in \eqref{eq: prop 1 riformulata}.
\label{lemma: da hatI a I}
\end{lemma}
From \eqref{eq: da hat I a I} and \eqref{eq: tesi I con 1+epsilon} it follows
\begin{equation}
 \sum_{ i = 0}^{n - 1}[\hat{I}_{1, 1}^n(i) + \hat{I}_{1, 2}^n(i) + \hat{I}_{1, 3}^n(i)] f(X_{t_i})  = o_{L^1}(\Delta_n^{\frac{1}{2} - \tilde{\epsilon}}) =  o_{L^1}(\Delta_n^{(\frac{1}{2} - \tilde{\epsilon}) \land (1 - \alpha \beta - \tilde{\epsilon} )}).
\label{eq: nuovo hat 123}
\end{equation}
On $ \sum_{ i = 0}^{n - 1}\hat{I}_{1, 4}^n  f(X_{t_i})  = : \sum_{ i = 0}^{n - 1} \zeta_{n,i} $ we want to use Lemma 9 in \cite{Genon Catalot}. By the independence between $L$ and $W$ we get 
\begin{equation}
\frac{1}{\Delta_n^{\frac{1}{2} - \tilde{\epsilon}}} \sum_{i = 0}^{n - 1} \mathbb{E}_i[\zeta_{n,i}] = \frac{1}{\Delta_n^{\frac{1}{2} - \tilde{\epsilon}}} \sum_{i = 0}^{n - 1} f(X_{t_i}) \Delta_{n,i}^{ - \beta} \mathbb{E}_i[(\Delta \tilde{X}_i^J)^2 \varphi'_{\Delta_{n}^\beta}(\Delta \tilde{X}_i^J)] \mathbb{E}_i[\Delta X_i^c] = 0
\label{eq: hat I1 genon catalot}
\end{equation}
and 
\begin{equation}
\Delta_n^{- 2(\frac{1}{2} - \tilde{\epsilon})} \sum_{i = 0}^{n - 1} f^2(X_{t_i}) \Delta_{n,i}^{ - 2\beta} \mathbb{E}_i[(\Delta \tilde{X}_i^J)^4 \varphi'^2_{\Delta_{n}^\beta}(\Delta \tilde{X}_i^J)] \mathbb{E}_i[(\Delta X_i^c)^2] \le  c\Delta_n^{2\tilde{\epsilon} + 2 \beta - \alpha \beta },
\label{eq: hat I1 carre genon catalot}
\end{equation}
where we have also used the sixth point of Lemma \ref{lemma: Moment inequalities} and the first point of Lemma \ref{lemma: estensione 10 capitolo 1}. Using \eqref{eq: hat I1 genon catalot} and \eqref{eq: hat I1 carre genon catalot} we have 
$$ \sum_{ i = 0}^{n - 1}\hat{I}_{1, 4}^n  f(X_{t_i})  = o_{\mathbb{P}}(\Delta_n^{(\frac{1}{2} - \tilde{\epsilon}) \land (1 - \alpha \beta - \tilde{\epsilon} )})$$
that, joint with \eqref{eq: nuovo hat 123} and the fact that the convergence in norm 1 implies the convergence in probability, give us \eqref{eq: I1 nuovo salti}. Using also \eqref{eq: I2 nuovo salti} - \eqref{eq: hat I2 nuovo salti} we get \eqref{eq: tesi per aggiunta 1 salti} and so \eqref{eq: aggiunta prop1 salti}. \\ \\
In order to prove \eqref{eq: salti con browniano}, we reformulate $\Delta X_i^J \varphi_{\Delta_{n}^\beta}(\Delta X_i)$ as we have already done in \eqref{eq: prop 1 riformulata} getting
$$(\int_{t_i}^{t_{i+1}}a_s dW_s)\Delta X_i^J \varphi_{\Delta_{n}^\beta}(\Delta X_i) = (\int_{t_i}^{t_{i+1}}a_s dW_s)(\Delta X_i^J)[\varphi_{\Delta_{n}^\beta}(\Delta X_i) - \varphi_{\Delta_{n}^\beta}(\Delta X_i^J)] +$$
\begin{equation}
 + (\int_{t_i}^{t_{i+1}}a_s dW_s)(\Delta X_i^J)[\varphi_{\Delta_{n}^\beta}(\Delta X_i^J) - \varphi_{\Delta_{n}^\beta}(\Delta \tilde{X}_i^J)] + (\int_{t_i}^{t_{i+1}}a_s dW_s)(\Delta X_i^J - \Delta \tilde{X}_i^J) \varphi_{\Delta_{n}^\beta}(\Delta \tilde{X}_i^J) +
\label{eq: a con salti riformulata}
\end{equation}
$$ + (\int_{t_i}^{t_{i+1}}a_s dW_s)(\Delta \tilde{X}_i^J)\varphi_{\Delta_{n}^\beta}(\Delta \tilde{X}_i^J) = : \sum_{j= 1}^4 \tilde{I}_j^n(i).$$
Comparing \eqref{eq: a con salti riformulata} with \eqref{eq: salti con browniano} it turns out that our goal is to prove that $ \frac{1}{\Delta_{n,i}^{\beta(2  - \alpha) + 1}}\sum_{j = 1}^3 \mathbb{E}[|\tilde{I}_j^n(i)|] \rightarrow 0$, for $n \rightarrow \infty$ (again, acting as we do in the sequel it is also possible to show that $\sum_{j = 1}^3 \mathbb{E}_i[|\tilde{I}_j^n(i)|] \le R_i(\Delta_{n,i}^{\beta(2  - \alpha) + 1})$. Let us start considering $\tilde{I}_1^n(i)$. Using Holder inequality, its expected value is upper bounded by
\begin{equation}
\mathbb{E}[|\int_{t_i}^{t_{i+1}}a_s dW_s|^{p_1}]^\frac{1}{p_1} \mathbb{E}[|\Delta X_i^J|^{p_2}|\varphi_{\Delta_{n}^\beta}(\Delta X_i) - \varphi_{\Delta_{n}^\beta}(\Delta X_i^J)|^{p_2}]^\frac{1}{p_2}.
\label{eq: tilde I1}
\end{equation}
We now act on $\mathbb{E}[|\Delta X_i^J|^{p_2}|\varphi_{\Delta_{n}^\beta}(\Delta X_i) - \varphi_{\Delta_{n}^\beta}(\Delta X_i^J)|^{p_2}]^\frac{1}{p_2}$ as we did in the study of $I_1^n(i)$:
$$|\Delta X_i^J|^{p_2}|\varphi_{\Delta_{n}^\beta}(\Delta X_i) - \varphi_{\Delta_{n}^\beta}(\Delta X_i^J)|^{p_2} = |\Delta X_i^J|^{p_2}|\varphi_{\Delta_{n}^\beta}(\Delta X_i) - \varphi_{\Delta_{n}^\beta}(\Delta X_i^J)|^{p_2} 1_{\left \{ |\Delta X_i|\ge 3 \Delta_{n}^\beta \right \}} + $$
$$ + |\Delta X_i^J|^{p_2}|\varphi_{\Delta_{n}^\beta}(\Delta X_i) - \varphi_{\Delta_{n}^\beta}(\Delta X_i^J)|^{p_2} 1_{\left \{ |\Delta X_i| < 3 \Delta_{n}^\beta \right \}} = : \tilde{I}_{1,1}^n + \tilde{I}_{1,2}^n. $$
Concerning $\tilde{I}_{1,1}^n$, if $|\Delta X_i^J| \ge 2 \Delta_{n}^\beta$ it is just $0$, otherwise we can act exactly as we have done on $I_{1,1}^n$, taking $p_2 = 2$. Hence, $\forall r \ge 1$, 
\begin{equation}
\mathbb{E}[|\tilde{I}_{1,1}^n|]^\frac{1}{2} \le (c \Delta_{n}^{2 \beta + r(\frac{1}{2} - \beta)})^\frac{1}{2} = c \Delta_{n}^{\beta + \frac{r}{2}(\frac{1}{2} - \beta)}.
\label{eq: tilde I11}
\end{equation}
Let us now consider $\tilde{I}_{1,2}^n$. If $|\Delta X_i^J| \ge 4 \Delta_{n}^\beta$, we act again like we did on $I_{1,2}^n$, taking $p_2 = 2$. It yields again
\begin{equation}
\mathbb{E}[|\tilde{I}_{1,2}^n|1_{\left \{ |\Delta X_i^J| \ge 4 \Delta_{n}^\beta \right \}}]^\frac{1}{2} \le c \Delta_{n}^{\beta + \frac{r}{2}(\frac{1}{2} - \beta)}.
\label{eq: prima parte tildeI12}
\end{equation}
If $|\Delta X_i^J| < 4 \Delta_{n}^\beta$ we use the smoothness of $\varphi$ and Holder inequality getting
$$\mathbb{E}[|\tilde{I}_{1,2}^n|1_{\left \{ |\Delta X_i^J| < 4 \Delta_{n}^\beta \right \}}] \le \Delta_{n}^{- \beta} \mathbb{E}[|\Delta X_i^J|^{p_2}|\varphi'(\zeta)|^{p_2}|\Delta X_i^c|^{p_2}1_{\left \{ |\Delta X_i| < 3 \Delta_{n}^\beta, |\Delta X_i^J| < 4 \Delta_{n}^\beta \right \}}]^\frac{1}{p_2} \le$$
\begin{equation}
\le \Delta_{n}^{- \beta} \mathbb{E}[|\Delta X_i^c|^{p_2 \, p}]^\frac{1}{p_2 \, p} \mathbb{E}[|\varphi'(\zeta)|^{p_2\, q}|\Delta X_i^J|^{p_2 \, q}1_{\left \{ |\Delta X_i| < 3 \Delta_{n}^\beta, |\Delta X_i^J| < 4 \Delta_{n}^\beta \right \}}]^\frac{1}{p_2 \, q},
\label{eq: intermedia I12 parte 2}
\end{equation}
with $\zeta$ a point between $\Delta X_i^J$ and $\Delta X_i$. \\
Now we observe that, if $|\Delta X_i^c| \ge \frac{\Delta_{n}^\beta}{4}$, then taking $p_2 \, q = 1 + \epsilon$ we have
$$\mathbb{E}[|\varphi'(\zeta)|^{1 + \epsilon}|\Delta X_i^J|^{1 + \epsilon}1_{\left \{ |\Delta X_i| < 3 \Delta_{n}^\beta, |\Delta X_i^J| < 4 \Delta_{n}^\beta, |\Delta X_i^c| \ge \frac{\Delta_{n}^\beta}{4} \right \}}]^\frac{1}{1 + \epsilon} \le c \Delta_{n}^{\beta + r(\frac{1}{2} - \beta)\frac{1}{1 + \epsilon}}$$
where we have used the bound on $|\Delta X_i^J|$ given by the indicator function, the boundedness of $\varphi'$ and \eqref{eq: prob parte continua}. Otherwise, by the definition of $\varphi$, we know that $|\varphi'(\zeta)| \neq 0$ only if $|\zeta| \in (\Delta_{n}^\beta, 2 \Delta_{n}^\beta)$. Then $\Delta_{n}^\beta \le |\zeta| \le |\Delta X_i| + |\Delta X_i^J| \le 2 |\Delta X_i^J| + |\Delta X_i^c|\le 2|\Delta X_i^J| + \frac{\Delta_{n}^\beta}{4} $, hence $|\Delta X_i^J| \ge \frac{3}{8} \Delta_{n}^\beta \ge \frac{\Delta_{n}^\beta}{4}$ and so we can say it is 
$$\mathbb{E}[|\varphi'(\zeta)|^{1 + \epsilon}|\Delta X_i^J|^{1 + \epsilon}1_{\left \{ |\Delta X_i| < 3 \Delta_{n}^\beta, |\Delta X_i^J| < 4 \Delta_{n}^\beta, |\Delta X_i^c| < \frac{\Delta_{n}^\beta}{4} \right \}}]^\frac{1}{1 + \epsilon} \le c \mathbb{E}[|\Delta X_i^J|^{1 + \epsilon}1_{\left \{  \frac{\Delta_{n}^\beta}{4} \le |\Delta X_i^J| < 4 \Delta_{n}^\beta,  \right \}}].$$
Using the second point of Lemma \ref{lemma: estensione 10 capitolo 1}, passing through the conditional expected value we get it is upper bounded by
$$\Delta_{n}^{1 + \beta (1 + \epsilon - \alpha)} \mathbb{E} [R_i( 1)] \le c \Delta_{n}^{1 + \beta(1 + \epsilon -  \alpha)}.$$
 Hence 
\begin{equation}
\mathbb{E}[|\varphi'(\zeta)|^{1 + \epsilon}|\Delta X_i^J|^{1 + \epsilon}1_{\left \{ |\Delta X_i| < 3 \Delta_{n}^\beta, |\Delta X_i^J| < 4 \Delta_{n}^\beta \right \}}]^\frac{1}{1 + \epsilon} \le c \Delta_{n}^{[\beta + r(\frac{1}{2} - \beta) - \epsilon] \land [1 + \beta(1 + \epsilon -  \alpha)]\frac{1}{1 + \epsilon}} = c\Delta_{n}^{[1 + \beta(1 + \epsilon -  \alpha)] \frac{1}{1 + \epsilon}}. 
\label{eq:I12 parte finale}
\end{equation}
The last equality follows from the fact that, for each choice of $\beta \in (0, \frac{1}{2})$ and $\alpha \in (0,2)$, we can always find $r \ge 1$ and $\epsilon > 0$ such that $\beta + r(\frac{1}{2} - \beta) - \epsilon > 1 + \beta(1 + \epsilon -  \alpha) $. \\
Replacing \eqref{eq:I12 parte finale} in \eqref{eq: intermedia I12 parte 2} and using the sixth point of Lemma \ref{lemma: Moment inequalities} we have that 
\begin{equation}
\mathbb{E}[|\tilde{I}_{1,2}^n|1_{\left \{ |\Delta X_i^J| < 4 \Delta_{n}^\beta \right \}}]^\frac{1}{p_2} \le c \Delta_{n}^{[\frac{1}{2} - \beta + 1 + \beta(1 + \epsilon - \alpha)]\frac{1}{p_2}} = c \Delta_{n}^{(\frac{3}{2} - \alpha \beta - \epsilon) \frac{1}{p_2}} = c \Delta_{n}^{\frac{3}{2} - \alpha \beta - \epsilon},
\label{eq: tilde I12 finale}
\end{equation}
the last equality follows from the choice of both $p_2$ and $q$ next to $1$.
Using \eqref{eq: tilde I11}, \eqref{eq: prima parte tildeI12} and \eqref{eq: tilde I12 finale} we get 
\begin{equation}
\mathbb{E}[|\Delta X_i^J|^{p_2}|\varphi_{\Delta_{n}^\beta}(\Delta X_i) - \varphi_{\Delta_{n}^\beta}(\Delta X_i^J)|^{p_2}]^\frac{1}{p_2} \le c \Delta_{n}^{[\beta + \frac{r}{2}(\frac{1}{2} - \beta)] \land [\frac{3}{2} - \alpha \beta - \epsilon]} =  c \Delta_{n}^{\frac{3}{2} - \alpha \beta - \epsilon}.  
\label{eq: finale tilde I1}
\end{equation}
Replacing \eqref{eq: bdg} and \eqref{eq: finale tilde I1} in \eqref{eq: tilde I1} it follows
\begin{equation}
\mathbb{E}[|\tilde{I}_1^n(i)|] \le c\Delta_{n}^{2 - \alpha \beta - \epsilon},
\label{eq: tilde I1 nuovo}
\end{equation}
hence
\begin{equation}
\frac{\mathbb{E}[|\tilde{I}_1^n(i)|]}{\Delta_{n}^{1 + \beta(2 - \alpha)}} \le c \Delta_{n}^{1 - 2 \beta - \epsilon}.
\label{eq: tilde I1 finale}
\end{equation}
Since we can always find an $\epsilon > 0$ such that $1 - 2 \beta - \epsilon > 0$, the expected value above goes to $0$ for $n \rightarrow \infty$. \\
Concerning $\tilde{I}_2^n(i)$, we split again on $\tilde{I}_{2,1}^n: = \tilde{I}_2^n(i) 1_{\left \{  |\Delta X_i^J|  \le 2 \Delta_{n}^\beta \right \}}$ and $\tilde{I}_{2,2}^n: = \tilde{I}_2^n (i) 1_{\left \{  |\Delta X_i^J|  > 2 \Delta_{n}^\beta \right \}}$.
$$\mathbb{E}[|\tilde{I}_{2,1}^n|] = \mathbb{E}[|\tilde{I}_2^n(i)| 1_{\left \{ |\Delta X_i^J|  \le 2 \Delta_{n}^\beta \right \}}] \le c \Delta_{n}^{- \beta} \mathbb{E}[|\int_{t_i}^{t_{i+1}}a_s dW_s||\Delta X_i^J||\Delta X_i^J - \Delta \tilde{X}_i^J|1_{\left \{ |\Delta X_i^J|  \le 2 \Delta_{n}^\beta \right \}}] \le$$
$$\le c \Delta_{n}^{- \beta} \mathbb{E}[|\int_{t_i}^{t_{i+1}}a_s dW_s|^2|\Delta X_i^J|^21_{\left \{ |\Delta X_i^J|  \le 2 \Delta_{n}^\beta \right \}}]^\frac{1}{2} \mathbb{E}[|\Delta X_i^J - \Delta \tilde{X}_i^J|^2]^\frac{1}{2} \le$$
$$ \le c\Delta_{n,i}^{1 - \beta}\mathbb{E}[|\int_{t_i}^{t_{i+1}}a_s dW_s|^{2p}]^\frac{1}{2p} \mathbb{E}[|\Delta X_i^J|^{2q} 1_{\left \{ |\Delta X_i^J|  \le 2 \Delta_{n}^\beta \right \}}]^\frac{1}{2q},$$
where we have used Cauchy-Schwartz inequality, \eqref{eq: diff salti alpha >1} and Holder inequality. Now we take $p$ big and $q$ next to $1$, using \eqref{eq: bdg} and the first point of Lemma \ref{lemma: estensione 10 capitolo 1} we get
\begin{equation}
\mathbb{E}[|\tilde{I}_{2,1}^n|] \le c \Delta_{n}^{1 - \beta + \frac{1}{2} + \frac{1}{2} + \frac{\beta}{2}(2 - \alpha) - \epsilon} 
\label{eq: tilde I21 nuovo}
\end{equation}
and so
\begin{equation}
\frac{1}{\Delta_{n}^{1 + \beta (2 - \alpha)}} \mathbb{E}[|\tilde{I}^n_{2,1}|]\le \Delta_{n}^{1 - 2 \beta + \frac{\alpha \beta}{2} - \epsilon}.
\label{eq: conv I21}
\end{equation}
It goes to $0$ for $n \rightarrow \infty$ because we can always find an $\epsilon > 0$ such that the exponent in $\Delta_{n}$ is positive. Let us now consider $\tilde{I}_{2,2}^n= \tilde{I}_{2,2}^n1_{\left \{ |\Delta \tilde{X}_i^J| \le 2 \Delta_{n}^\beta \right \}} + \tilde{I}_{2,2}^n1_{\left \{ |\Delta \tilde{X}_i^J| > 2 \Delta_{n}^\beta \right \}}$. From the definition of $\varphi$, $\tilde{I}_{2,2}^n1_{\left \{ |\Delta \tilde{X}_i^J| > 2 \Delta_{n}^\beta \right \}} = 0$. 
$$\mathbb{E}[|\tilde{I}_{2,2}^n|1_{\left \{ |\Delta \tilde{X}_i^J| \le 2 \Delta_{n}^\beta \right \}}] = \mathbb{E}[|\int_{t_i}^{t_{i+1}}a_s dW_s||\Delta \tilde{X}_i^J||\varphi_{\Delta_{n}^\beta}(\Delta X_i^J) - \varphi_{\Delta_{n}^\beta}(\Delta \tilde{X}_i^J)| 1_{\left \{ |\Delta \tilde{X}_i^J| \le 2 \Delta_{n}^\beta, |\Delta X_i^J| > 2 \Delta_{n}^\beta \right \}}] +$$
$$+ \mathbb{E}[|\int_{t_i}^{t_{i+1}}a_s dW_s||\Delta X_i^J - \Delta \tilde{X}_i^J||\varphi_{\Delta_{n}^\beta}(\Delta X_i^J) - \varphi_{\Delta_{n}^\beta}(\Delta \tilde{X}_i^J)| 1_{\left \{ |\Delta \tilde{X}_i^J| \le 2 \Delta_{n}^\beta, |\Delta X_i^J| > 2 \Delta_{n,i}^\beta \right \}}] \le $$
$$\le c \Delta_{n}^{2 - \frac{\alpha \beta}{2} - \epsilon} + \mathbb{E}[|\int_{t_i}^{t_{i+1}}a_s dW_s||\Delta X_i^J - \Delta \tilde{X}_i^J||-\varphi_{\Delta_{n}^\beta}(\Delta \tilde{X}_i^J)|], $$
where we have acted exactly like we did in $\tilde{I}_{2,1}^n$, using that $\Delta \tilde{X}_i^J$ is less then $2 \Delta_{n}^\beta$. We have also used that, by the definition of $\varphi$, evaluated in $\Delta X_i^J$ it is zero. Now we use Holder inequality, \eqref{eq: bdg} and the boundedness of $\varphi$ to get
$$\mathbb{E}[|\tilde{I}_{2,2}^n|] \le c \Delta_{n}^{2 - \frac{\alpha \beta}{2} - \epsilon} + \mathbb{E}[|\int_{t_i}^{t_{i+1}}a_s dW_s|^p]^\frac{1}{p} \mathbb{E}[|\Delta X_i^J - \Delta \tilde{X}_i^J|^q]^\frac{1}{q} \le c \Delta_{n}^{2 - \frac{\alpha \beta}{2} - \epsilon} + c \Delta_{n}^{\frac{1}{2}}\mathbb{E}[|\Delta X_i^J - \Delta \tilde{X}_i^J|^q]^\frac{1}{q}. $$
Now, if $\alpha < 1$ we use \eqref{eq: diff salti alpha <1}, with $q = 1 + \epsilon$, getting
\begin{equation}
\mathbb{E}[|\tilde{I}_{2,2}^n|] \le c \Delta_{n}^{2 - \frac{\alpha \beta}{2} - \epsilon} + c \Delta_{n}^{\frac{1}{2} + \frac{1}{2} + \frac{1}{1 + \epsilon}} = c \Delta_{n}^{2 - \frac{\alpha \beta}{2} - \epsilon}.
\label{eq: tilde I22 nuovo}
\end{equation}
Therefore, for $\alpha < 1$, we have 
\begin{equation}
\frac{1}{\Delta_{n}^{1 + \beta (2 - \alpha)}} \mathbb{E}[|I_{2,2}^n|] \le c \Delta_{n}^{1 - 2 \beta + \frac{\alpha \beta}{2} - \epsilon}.
\label{eq: conv I22 alpha <1}
\end{equation}
We can find an $\epsilon > 0$ such that the exponent on $\Delta_{n}$ is positive hence, if $\alpha < 1$, then $I_{2,2}^n = o_{L^1}(\Delta_{n}^{1 + \beta(2 - \alpha)})$. Otherwise, if $\alpha \ge 1$, we use \eqref{eq: diff salti alpha >1} having taken $q = 2$. We get
$$\mathbb{E}[|\tilde{I}_{2,2}^n|] \le c \Delta_{n}^{2 - \frac{\alpha \beta}{2} - \epsilon} + c \Delta_{n}^{\frac{1}{2} + 1} = c \Delta_{n}^{\frac{3}{2}}.$$
It follows that, for $\alpha \ge 1$, it is
\begin{equation}
\frac{1}{\Delta_{n}^{1 + \beta (2 - \alpha)}} \mathbb{E}[|I_{2,2}^n|] \le c \Delta_{n}^{\frac{1}{2} -\beta(2 - \alpha)}.
\label{eq: conv I22 alpha >1}
\end{equation}
We observe that the exponent on $\Delta_{n}$ is more than $0$ if $\beta < \frac{1}{2}\frac{1}{(2 - \alpha)}$, that is always true for $\beta \in (0, \frac{1}{2})$ and $\alpha \in [1,2)$. \\
To conclude, we use on $\tilde{I}_3(i)$ Holder inequality, \eqref{eq: bdg}, the boundedness of $\varphi$ and then we act as we did on $\tilde{I}_{2,2}^n$, using \eqref{eq: diff salti alpha <1} or \eqref{eq: diff salti alpha >1}, depending on whether or not $\alpha$ is less than $1$. 
In the case $\alpha < 1$ we get
\begin{equation}
 \frac{1}{\Delta_{n}^{1 + \beta(2 - \alpha)}} \mathbb{E}[|\tilde{I}_3^n(i)|] \le \frac{1}{\Delta_{n}^{1 + \beta(2 - \alpha)}} c \Delta_{n}^{\frac{1}{2} + \frac{1}{2} + \frac{1}{1 + \epsilon}}= c \Delta_{n}^{1 - \beta(2 - \alpha) - \epsilon},
\label{eq: conv tilde I3 alpha <1}
 \end{equation}
that goes to $0$ for $n \rightarrow \infty$ since we can always find $\epsilon > 0$ such that the exponent on $\Delta_{n}$ is positive. Otherwise it follows 
\begin{equation}
 \frac{1}{\Delta_{n}^{1 + \beta(2 - \alpha)}} \mathbb{E}[|\tilde{I}_3^n(i)|] \le \frac{1}{\Delta_{n}^{1 + \beta(2 - \alpha)}} c \Delta_{n}^{\frac{3}{2}}= c \Delta_{n,i}^{\frac{1}{2} - \beta(2 - \alpha)}.
\label{eq: conv tilde I3 alpha >1}
 \end{equation}
 The exponent on $\Delta_{n}$ is positive if $\beta < \frac{1}{2} \frac{1}{(2 - \alpha)}$, that is always true since we are in the case $\alpha \ge 1$. Hence $\tilde{I}_3^n(i) = o_{L^1}(\Delta_{n}^{1 + \beta(2 - \alpha)})$. \\
 From \eqref{eq: tilde I1 finale} - \eqref{eq: conv tilde I3 alpha >1} and the reformulation \eqref{eq: a con salti riformulata}, it follows \eqref{eq: salti con browniano}. \\ \\
 Replacing reformulation \eqref{eq: a con salti riformulata} in the left hand side of \eqref{eq: aggiunta prop1 browniano}, it turns out that the theorem is proved if 
 \begin{equation}
  \sum_{ i = 0}^{n - 1}(\sum_{k = 1}^3 \tilde{I}_k^n(i)) f(X_{t_i}) = o_{L^1}(\Delta_n^{(\frac{1}{2} - \tilde{\epsilon}) \land (1 - \alpha \beta - \tilde{\epsilon} )}).
 \label{eq: tesi prop1 aggiunta brown}
 \end{equation}
 Using a conditional version of equations \eqref{eq: tilde I1 nuovo}, \eqref{eq: tilde I21 nuovo}, \eqref{eq: tilde I22 nuovo}, \eqref{eq: conv tilde I3 alpha <1} and \eqref{eq: conv tilde I3 alpha >1} (adding in the last two $\beta(2 - \alpha)$ in the exponent of $\Delta_{n}$) we easily get \eqref{eq: tesi prop1 aggiunta brown} and so \eqref{eq: aggiunta prop1 browniano}.
\end{proof}

\subsection{Proof of Lemma \ref{lemma: dl d}}
\begin{proof}
By the definition of $d(\zeta_n)$, as in law we have that $S_1^\alpha = - S_1^\alpha$, we get $d(\zeta_n) = d(|\zeta_n|)$ and thus we can assume that $\zeta_n > 0$.
Using a change of variable we obtain
\begin{equation}
d(\zeta_n) = \mathbb{E}[(S_1^\alpha)^2\varphi(S_1^\alpha \zeta_n)] = \int_\mathbb{R} z^2 \varphi(z \zeta_n) f_\alpha(z)dz = (\zeta_n)^{- 3}\int_\mathbb{R} u^2 \varphi(u) f_\alpha(\frac{u}{\zeta_n})du.
\label{eq: dl d start}
\end{equation}
We want to use an asymptotic expansion of the density (see Theorem 7.22 in \cite{Menozzi}, with $d=1$ and $\sigma = 1$) which states that, if $z$ is big enough, then a development up to order $N$ of $f_\alpha(z)$ is
\begin{equation}
\frac{c_\alpha}{|z|^{1 + \alpha}} + \frac{1}{\pi} \frac{1}{|z|} \sum_{k = 2}^N \frac{a_k}{k!} (|z|^{- \alpha})^k + o(|z|^{- \alpha N}),
\label{eq: asymptotic expansion density}
\end{equation}
for some coefficients $a_k$.
We therefore take $M > 0$ big enough such that, for $\frac{u}{\zeta_n} > M$, we can use \eqref{eq: asymptotic expansion density}. Hence the right hand side of \eqref{eq: dl d start} can be seen as 
\begin{equation}
(\zeta_n)^{- 3}\int_{|u| \le \zeta_n M} u^2 \varphi(u) f_\alpha(\frac{u}{\zeta_n})du + (\zeta_n)^{- 3}\int_{|u| > \zeta_n M} u^2 \varphi(u) f_\alpha(\frac{u}{\zeta_n})du= : I_1^n + I_2^n.
\label{eq: dl d meta}
\end{equation}
We have that, $\forall \hat{\epsilon} > 0$, $I_1^n = o(\zeta_n^{- \hat{\epsilon}})$. Indeed, using that $\varphi$ and $f_\alpha$ are both bounded, we get
\begin{equation}
\frac{I_1^n}{\zeta_n^{- \hat{\epsilon}}} \le \zeta_n^{- 3 + \hat{\epsilon}}\int_{|u| \le \zeta_n M} u^2 du \le c \zeta_n^{\hat{\epsilon}},
\label{eq: estim I1 density}
\end{equation}
that goes to zero because we have assumed that $\zeta_n \rightarrow 0$. $I_2^n$ is
\begin{equation}
(\zeta_n)^{- 3}\int_{|u| > \zeta_n M} u^2 \varphi(u) c_\alpha (\zeta_n)^{1 + \alpha}|u|^{- 1 - \alpha} du  + (\zeta_n)^{- 3} \int_{|u| > \zeta_n M} u^2 \varphi(u) [f_\alpha(\frac{u}{\zeta_n}) - \frac{c_\alpha}{|u|^{1 + \alpha}} |\zeta_n|^{1 + \alpha}] du. 
\label{eq: estim I2 density}
\end{equation}
The first term here above can be seen as 
$$(\zeta_n)^{\alpha - 2} c_\alpha \int_\mathbb{R} |u|^{1 - \alpha} \varphi(u) du - (\zeta_n)^{\alpha - 2} c_\alpha\int_{|u| \le \zeta_n M} |u|^{1 - \alpha} \varphi(u) du  = (\zeta_n)^{\alpha - 2} c_\alpha \int_\mathbb{R} |u|^{1 - \alpha} \varphi(u) du + o((\zeta_n)^{-\hat{\epsilon}}). $$
Indeed, using that $\varphi$ is bounded, we have
\begin{equation}
\frac{1}{(\zeta_n)^{-\hat{\epsilon}}}|(\zeta_n)^{\alpha - 2} c_\alpha \int_{|u| \le \zeta_n M} |u|^{1 - \alpha} \varphi(u) du| \le c (\zeta_n)^{\hat{\epsilon} + \alpha - 2}\int_{|u| \le \zeta_n M} |u|^{1 - \alpha}du \le c (\zeta_n)^{\hat{\epsilon}}.
\label{eq: densita su r}
\end{equation}
that goes to zero for $n \rightarrow \infty$. \\
Replacing \eqref{eq: estim I1 density}, \eqref{eq: estim I2 density} and \eqref{eq: densita su r} in \eqref{eq: dl d meta} and comparing it with \eqref{eq: dl d}, it turns out that our goal is to show that the second term of \eqref{eq: estim I2 density} is $o(\zeta_n^{( -\hat{\epsilon} ) \land (2 \alpha - 2 - \hat{\epsilon})})$. Using on it \eqref{eq: asymptotic expansion density} with $N = 2$, which implies $|f_\alpha(z) - \frac{c_\alpha}{|z|^{1 + \alpha}}| \le \frac{c}{|z|^{1 + 2 \alpha}}$ for $|z| > M$ and some $c > 0$, we can upper bound it with $c(\zeta_n)^{2 \alpha - 2} \int_{|u| \le \zeta_n M} |u|^{1 - 2\alpha}du$ .  By the definition of $\varphi$ we have
\begin{equation}
\int_{|u| > \zeta_n M} |u|^{1 - 2\alpha} \varphi(u) du = \int_{-2}^{-\zeta_n M} (-u)^{1 - 2\alpha} \varphi(u) du + \int_2^{\zeta_n M} u^{1 - 2\alpha} \varphi(u) du \le c + c (\zeta_n)^{2 - 2\alpha }.
\label{eq: integrale u grande}
\end{equation}
Therefore we get that the second term of \eqref{eq: estim I2 density} is upper bounded by $$c \zeta_n^{2 \alpha - 2} + c.$$
The first term here above is clearly $o(\zeta_n^{2 \alpha - 2 - \hat{\epsilon}})$ while the second is $o(\zeta_n^{ -\hat{\epsilon}})$, hence the sum is $o(\zeta_n^{( -\hat{\epsilon} ) \land (2 \alpha - 2 - \hat{\epsilon})})$. 
The lemma is therefore proved.
\end{proof}

\subsection{Proof of Lemma \ref{lemma: differenza dei salti}}
\begin{proof}
We observe that, $\forall \alpha \in [0, 2]$, we have
$$\mathbb{E}[|\Delta X_i^J - \Delta \tilde{X}_i^J|^2] = \mathbb{E}[(\int_{t_i}^{t_{i + 1}} \int_{\mathbb{R}} [\gamma(X_{s-}) - \gamma (X_{t_i})] z \tilde{\mu}(ds,dz))^2] = \mathbb{E}[\int_{t_i}^{t_{i + 1}} \int_{\mathbb{R}} [\gamma(X_{s-}) - \gamma (X_{t_i})]^2 |z|^2 \bar{\mu}(ds,dz)] \le $$
\begin{equation}
\le c \int_{t_i}^{t_{i + 1}}\mathbb{E}[ |X_s - X_{t_i}|^2]ds \int_\mathbb{R} |z|^2 F(z) dz \le c \int_{t_i}^{t_{i + 1}} \Delta_{n} ds \le c \Delta_{n}^2,
\label{eq: I21 incrementi}
\end{equation}
where we have used Ito isometry, the regularity of $\gamma$ and the third point of Lemma \ref{lemma: Moment inequalities}. \\ 
We have in this way proved \eqref{eq: diff salti alpha >1} and showed that \eqref{eq: diff salti alpha <1} holds with $q= 2$. 
For $q > 2$, using Kunita inequality and acting like we did here above we get
$$\mathbb{E}[|\Delta X_i^J - \Delta \tilde{X}_i^J|^q] \le \mathbb{E}[\int_{t_i}^{t_{i + 1}} \int_{\mathbb{R}} [\gamma(X_{s-}) - \gamma (X_{t_i})]^q |z|^q \bar{\mu}(ds,dz)] + \mathbb{E}[(\int_{t_i}^{t_{i + 1}} \int_{\mathbb{R}} [\gamma(X_{s-}) - \gamma (X_{t_i})]^2 |z|^2 \bar{\mu}(ds,dz))^\frac{q}{2}] \le$$
$$\le c \int_{t_i}^{t_{i + 1}}\mathbb{E}[ |X_s - X_{t_i}|^q]ds + \mathbb{E}[(\int_{t_i}^{t_{i + 1}} |X_s - X_{t_i}|^2ds )^\frac{q}{2}] \le c \Delta_{n}^2 + c  \Delta_{n}^{\frac{q}{2} - 1} \int_{t_i}^{t_{i + 1}} \mathbb{E}[|X_s - X_{t_i}|^q]ds = c \Delta_{n}^2 + c \Delta_{n}^{\frac{q}{2} - 1} \le  c \Delta_{n}^2, $$
where we have also used Jensen inequality. \\
In order to prove \eqref{eq: diff salti alpha <1} we observe that, if $\alpha < 1$, then we have
\begin{equation}
\mathbb{E}[|\Delta X_i^J - \Delta \tilde{X}_i^J|] \le \mathbb{E}[|\int_{t_i}^{t_{i + 1}} \int_{|z| \ge 2 \Delta_{n}^\beta} [\gamma(X_{s-}) - \gamma (X_{t_i})] z \tilde{\mu}(ds,dz)|] +  \mathbb{E}[|\int_{t_i}^{t_{i + 1}} \int_{|z| \le 2 \Delta_{n}^\beta} [\gamma(X_{s-}) - \gamma (X_{t_i})] z \tilde{\mu}(ds,dz)|].
\label{eq: incrementi salti q=1}
\end{equation}
The first term in the right hand side of \eqref{eq: incrementi salti q=1} is upper bounded by
$$\left \| \gamma' \right \|_\infty \mathbb{E}[\int_{t_i}^{t_{i+ 1}}\int_{|z| \ge 2 \Delta_{n}^\beta}|X_{s-} - X_{t_i}| |z| F(z) dz ds] \le c \int_{t_i}^{t_{i+ 1}}\int_{|z| \ge 2 \Delta_{n}^\beta}\mathbb{E}[|X_{s-} - X_{t_i}|^2]^\frac{1}{2} ds |z| F(z) dz \le$$
\begin{equation}
\le c \int_{t_i}^{t_{i+ 1}} \Delta_{n}^\frac{1}{2} (\int_{|z| \ge 2 \Delta_{n}^\beta} |z| F(z) dz) ds \le c \Delta_{n}^\frac{3}{2},  
\label{eq: q=1 primo termine}
\end{equation}
where we have used the compensation formula, the regularity of $\gamma$, Cauchy-Schwartz inequality in order to use the third point of Lemma \ref{lemma: Moment inequalities}  and the boundedness of the integral for $|z| \ge 2 \Delta_{n}^\beta$. Moreover, acting in the same way, the second term in the right hand side of \eqref{eq: incrementi salti q=1} is upper bounded by
\begin{equation}
\left \| \gamma' \right \|_\infty \mathbb{E}[\int_{t_i}^{t_{i+ 1}}\int_{|z| \le 2 \Delta_{n}^\beta}|X_{s-} - X_{t_i}| |z| F(z) dz ds] \le c\int_{t_i}^{t_{i+ 1}} \Delta_{n}^\frac{1}{2} (\int_{|z| \ge 2 \Delta_{n}^\beta} |z|^{- \alpha} dz) ds \le c \Delta_{n}^{\frac{3}{2} + \beta(1 - \alpha)}, 
\label{eq: q=1 secondo termine}
\end{equation}
using again compensation formula, the regularity of $\gamma$ and Cauchy-Schwartz inequality in order to use the third point of Lemma \ref{lemma: Moment inequalities}. We have also used the third point of A4 and computed the integral on $z$. Using \eqref{eq: incrementi salti q=1} - \eqref{eq: q=1 secondo termine} we get
\begin{equation}
\mathbb{E}[|\Delta X_i^J - \Delta \tilde{X}_i^J|] \le c \Delta_{n}^{\frac{3}{2} \land [\frac{3}{2} + \beta(1 - \alpha)]} = c \Delta_{n}^\frac{3}{2},
\label{eq: incrementi q=1 finale}
\end{equation}
since $\alpha < 1$ and so $(1 - \alpha ) > 0$.
We now use interpolation theorem (see below Theorem 1.7 in Chapter 4 of \cite{Interpolation}) getting
$$\mathbb{E}[|\Delta X_i^J - \Delta \tilde{X}_i^J|^q]^\frac{1}{q} \le \mathbb{E}[|\Delta X_i^J - \Delta \tilde{X}_i^J|]^{\theta}(\mathbb{E}[|\Delta X_i^J - \Delta \tilde{X}_i^J|^2]^\frac{1}{2})^{ 1 -\theta},$$
with $\frac{1}{q} = \theta + \frac{1 - \theta}{2}$, hence $\theta = \frac{2}{q} - 1$. Using \eqref{eq: I21 incrementi} and \eqref{eq: incrementi q=1 finale} it follows
$$\mathbb{E}[|\Delta X_i^J - \Delta \tilde{X}_i^J|^q]^\frac{1}{q} \le c\Delta_{n}^{\frac{3}{2}\theta} \Delta_{n,i}^{1 - \theta} = c\Delta_{n}^{\frac{1}{2}\theta + 1} = c\Delta_{n}^{\frac{1}{q}+ \frac{1}{2}}, $$
where we have also replaced $\theta$.
\end{proof}

\subsection{Proof of Lemma \ref{lemma: da hatI a I}}
\begin{proof}
We want to use a conditional version of the interpolation theorem, therefore we have to estimate the norm $2$ of $I_2^n(i)$, $I_3^n(i)$ and $I_4^n(i)$. Observing that $\varphi$ is a bounded function and using Kunita inequality we get
$$\mathbb{E}_i[|I_2^n(i)|^2] \le \mathbb{E}_i[|\Delta X_i^J|^4] \le c \mathbb{E}_i[\int_{t_i}^{t_{i + 1}} \int_{\mathbb{R}} |\gamma(X_{s-})|^4 |z|^4 \bar{\mu}(ds,dz)] + c \mathbb{E}_i[(\int_{t_i}^{t_{i + 1}} \int_{\mathbb{R}} |\gamma(X_{s-})|^2 |z|^2 \bar{\mu}(ds,dz))^2] \le$$
$$\le c (\int_{\mathbb{R}}|z|^4 F(z)dz) \mathbb{E}_i[\int_{t_i}^{t_{i + 1}} |\gamma(X_{s-})|^4ds] + c \mathbb{E}_i[(\int_{\mathbb{R}}|z|^2 F(z)dz)^2(\int_{t_i}^{t_{i + 1}}|\gamma(X_{s-})|^2 ds)^2] \le$$
\begin{equation}
\le R_i(\Delta_{n}) + R_i(\Delta_{n}^2) = R_i(\Delta_{n}),
\label{eq: Ei I2 carre}
\end{equation}
where in the last inequality we have also used the polynomial growth of $\gamma$ and the fifth point of Lemma \ref{lemma: Moment inequalities}. \\
Concerning the norm $2$ of $I_3^n(i)$, we use the conditional version of the first point of Lemma \ref{lemma: differenza dei salti} for $q=2$ to get
\begin{equation}
\mathbb{E}_i[|I_3^n(i)|^2] \le \mathbb{E}_i[|\Delta X_i^J - \Delta \tilde{X}_i^J|^4] \le R_i(\Delta_{n}^2). 
\label{eq: Ei I3 carre}
\end{equation}
We now consider $I_4^n(i)$. Using Cauchy-Schwartz inequality and a conditional version of both the first point of Lemma \ref{lemma: differenza dei salti} for $q=2$ and \eqref{eq: estensione tilde salti lemma 10} in Lemma \ref{lemma: estensione 10 capitolo 1}, where $\varphi$ acts like the indicator function, we have
\begin{equation}
\mathbb{E}_i[|I_4^n(i)|^2]^\frac{1}{2} \le c\mathbb{E}_i[|\Delta X_i^J - \Delta \tilde{X}_i^J|^4]^\frac{1}{2} \mathbb{E}_i[|\Delta \tilde{X}_i^J \varphi_{\Delta_{n}^\beta}(\Delta \tilde{X}_i^J)|^4]^\frac{1}{2} \le R_i(\Delta_{n}^{\frac{3}{2} + \frac{\beta}{2}(4 - \alpha)}). 
\label{eq: Ei I4 carre}
\end{equation}
Using interpolation theorem it follows, $\forall j \in \left \{ 2,3,4 \right \}$,
\begin{equation}
\mathbb{E}_i[|I_j^n(i)|^{1 + \epsilon}]^\frac{1}{1 + \epsilon} \le \mathbb{E}_i[|I_j^n(i)|]^\theta (\mathbb{E}_i[|I_j^n(i)|^2]^\frac{1}{2})^{1 - \theta},
\label{eq: interpolation I}
\end{equation}
with $\theta$ such that $\frac{1}{1 + \epsilon} = \theta + \frac{1 - \theta}{2}$, hence $\theta = \frac{2}{1 + \epsilon} - 1 = 1 - \frac{2 \epsilon}{1 + \epsilon}$. \\
From a conditional version of  \eqref{eq: splitto I21 e I22}, \eqref{eq: E I21}, \eqref{eq: E I22} and equations \eqref{eq: Ei I2 carre} and \eqref{eq: interpolation I} it follows
$$$$
\begin{equation}
\mathbb{E}_i[|I_2^n(i)|^{1 + \epsilon}]^\frac{1}{1 + \epsilon} \le R_i(\Delta_{n}^{\frac{3}{2} + \beta - \frac{\alpha \beta}{2}})^\theta R_i(\Delta_{n}^\frac{1}{2})^{1 - \theta} = R_i(\Delta_{n}^{(\frac{3}{2} + \beta - \frac{\alpha \beta}{2})(1 - \frac{2 \epsilon}{1 + \epsilon} ) + \frac{ \epsilon}{1 + \epsilon}}) = R_i(\Delta_{n,i}^{\frac{3}{2} + \beta - \frac{\alpha \beta}{2} - \frac{ \epsilon}{1 + \epsilon}(2 + 2 \beta - \alpha \beta)}).
\label{eq: Ei I2 1 + epsilon}
\end{equation}
Since $2 + 2 \beta - \alpha \beta $ is always more than zero we can just see the exponent on $\Delta_{n,i}$ as $\frac{3}{2} + \beta - \frac{\alpha \beta}{2} - \epsilon.$ \\
From a conditional version of \eqref{eq: E I3}, \eqref{eq: Ei I3 carre} and \eqref{eq: interpolation I} it follows
\begin{equation}
\mathbb{E}_i[|I_3^n(i)|^{1 + \epsilon}]^\frac{1}{1 + \epsilon} \le R_i(\Delta_{n}^2)^\theta R_i(\Delta_{n})^{1 - \theta} = R_i(\Delta_{n}^{1 + \theta}) = R_i(\Delta_{n}^{2 - \frac{2 \epsilon}{1 + \epsilon}}).
\label{eq: Ei I3 1 + epsilon}
\end{equation}
In the same way, using a conditional version of \eqref{eq: E I4}, \eqref{eq: Ei I4 carre} and \eqref{eq: interpolation I} it follows
\begin{equation}
\mathbb{E}_i[|I_4^n(i)|^{1 + \epsilon}]^\frac{1}{1 + \epsilon} \le R_i(\Delta_{n}^{(\frac{3}{2} + \beta - \frac{\alpha \beta}{2})(1 - \frac{2 \epsilon}{1 + \epsilon} ) + \frac{ 2\epsilon}{1 + \epsilon}(\frac{3}{2} + 2\beta - \frac{\alpha \beta}{2})}) = R_i(\Delta_{n}^{\frac{3}{2} + \beta - \frac{\alpha \beta}{2} + \frac{2 \beta \epsilon}{1 + \epsilon}}).
\label{eq: Ei I4 1 + epsilon}
\end{equation}
The result \eqref{eq: tesi I con 1+epsilon} is a consequence of \eqref{eq: Ei I2 1 + epsilon}, \eqref{eq: Ei I3 1 + epsilon}, \eqref{eq: Ei I4 1 + epsilon} and that $2$ is always more than $\frac{3}{2} + \beta - \frac{\alpha \beta}{2}$.
\end{proof}

\end{document}